\documentclass{amsart}

\usepackage{amssymb}
\usepackage[dvips]{epsfig}
\usepackage{graphicx}
\usepackage{color}
\usepackage[notcite,notref]{showkeys}

 \usepackage[latin1]{inputenc}
 \usepackage[T1]{fontenc}
 \usepackage[normalem]{ulem}
\usepackage{verbatim}
 \usepackage{graphicx}
\usepackage[latin1]{inputenc}
\usepackage{amsmath}
\usepackage{amssymb}
\usepackage{amsfonts}
\usepackage{amsthm}
\usepackage{bbm}

\renewcommand{\u}{\mathfrak{u}}

\newcommand{\R}{\mathbb{R}}

\renewcommand{\Im}{{\rm Im}}

\newcommand{\sign}{{\rm sign}}

\newtheorem{theorem}{Theorem}[section]
\newtheorem{lemma}[theorem]{Lemma}
\newtheorem{proposition}[theorem]{Proposition}

\theoremstyle{remark}
\newtheorem{remark}{Remark}[section]
\theoremstyle{definition}

\newtheorem*{merci}{Acknowledgements}
\numberwithin{equation}{section}
\begin{document}

\title[The Whitham equation]{On Whitham and related equations }
\author[C. Klein]{Christian Klein}
\address{ Institut de Math\'ematiques de Bourgogne, UMR 5584\\   
Universit\'e de Bourgogne-Franche-Comt\'e, 9 avenue Alain Savary, 21078 Dijon
                Cedex, France}
\email{ Christian.Klein@u-bourgogne.fr}
\author[F. Linares]{Felipe Linares}
\address{ IMPA\\ Estrada Dona Castorina 110\\ Rio de Janeiro 22460-320, RJ Brasil}
\email{ linares@impa.br}
\author[D. Pilod]{Didier Pilod}
\address{Instituto de Matem\' atica, Universidade Federal do Rio de Janeiro, Caixa Postal 68530 CEP 21941-97, Rio de Janeiro, RJ Brasil}
\email{didier@im.ufrj.br}
\author[J.-C. Saut]{Jean-Claude Saut}
\address{Laboratoire de Math\' ematiques, UMR 8628,\\
Universit\' e Paris-Sud et CNRS,\\ 91405 Orsay, France}
\email{jean-claude.saut@u-psud.fr}

\date{June 25, 2017}
\maketitle

\begin{abstract}
The aim of this paper is to study,  via theoretical analysis and 
numerical simulations, the dynamics of Whitham and related equations. 
In particular we establish rigorous bounds between solutions of the 
Whitham and KdV equations and provide some insights into the dynamics of the Whitham equation in different regimes, some of them being outside the range of validity of the Whitham equation as a water waves model. 
\end{abstract}

\large
\section{introduction}
The Whitham equation was introduced formally in \cite{Whi}   as an 
alternative to the KdV equation, by keeping the exact dispersion of the water waves system.  No rigorous derivation from the water waves system is known, and there is no consistent scaling allowing to derive it.

Actually the idea in \cite{Whi} was to propose a model describing also the occurrence of waves of greatest height with the Stokes 120 degrees  angle at the crest which is clearly impossible with the KdV equation.  

As was conjectured in \cite{Whi} and later proved in \cite{EW}, this 
is actually the case (with a different angle though) but this \lq\lq 
wave breaking\rq\rq \, phenomena   seems to be irrelevant in the (KdV) regime where the Whitham equation can be  really viewed as a consistent water wave model, as will be discussed in the present paper. 

In short, and as was noticed by Whitham, \lq\lq the desired qualitative effect is in\rq\rq \, but this wave breaking effect cannot be quantitatively linked with the water waves system and thus to \lq\lq real\rq\rq \, waves. From the modeling point of view, the advantage of the Whitham equation with respect to KdV seems  to be the enlargement of  the range of validity of this asymptotic model in terms of frequencies though this is  not so easy to quantify.

On the other hand, due to the very different behavior of the dispersion at low and large frequencies, the Whitham equation is a fascinating mathema-tical object since, as will be discussed later, it has three interesting different asymptotic regimes. 

One should thus distinguish between the usefulness of the Whitham 
equation as a relevant water waves model which seems to be poor (see  
\cite{CG, Ca} for some preliminary comparisons with experiments) and 
its mathematical properties that are rich, in particular as a useful 
and relevant toy model to provide some insights into  the effect of 
adding  a \lq\lq weak\rq\rq \, dispersion  term to a nonlinear hyperbolic equation.

We will restrict in this paper to the Whitham equation posed on the 
real line. Many interesting aspects of the periodic problem, in particular concerning periodic solitary waves, can be found in \cite{EK, EW, HJ, RK,SKCK}. However our results concerning the Cauchy problem and its link with the KdV one are straightforwardly valid in the periodic case.

\vspace{0.3cm}
In order to make the link with the KdV equation (and thus with the full water wave system) we will write the Whitham equation as  in \cite{La1} :

\begin{equation}\label{Whit}
u_t+\mathcal L_\epsilon u_x+\epsilon uu_x=0,
\end{equation}
where the non local operator $\mathcal L_\epsilon$ is related to the dispersion relation of the (linearized) water waves system and is defined by
$$\mathcal L_\epsilon =l(\sqrt{\epsilon}D):=\left(\frac{\tanh \sqrt \epsilon |D|}{\sqrt \epsilon |D|}\right)^{1/2} \quad \text{and} \quad D=-i\nabla=-i\frac{\partial}{\partial x}\, .$$

\vspace{0.3cm}
The (small) parameter $\epsilon$ measures the comparable effects of nonlinearity and dispersion (see below). The Whitham equation is supposed to be an approximation of the full water waves system on time scales of order $\frac{1}{\epsilon},$ see \cite{La1} and Section 2 below.

\vspace{0.3cm}
Taking the formal limit $\sqrt \epsilon |\xi| \to 0$ in  $\mathcal L_\epsilon,$ \eqref{Whit} reduces to the KdV equation

\begin{equation}\label{KdV}
u_t+u_x +\epsilon uu_x+\frac{\epsilon}{6}u_{xxx}=0.
\end{equation}

We do not know of a complete rigorous justification of the Whitham equation from the water waves system. More precisely no correct scaling seems to exist allowing to connect directly and rigorously the Whitham equation and the water wave system. See however {\cite{La1} pages 213-214 and Section 2 below where a comparison with the KdV equation is displayed, justifying thus the Whitham equation via the KdV approximation of weakly nonlinear  long surface water waves. 

\vspace{0.3cm}
On the other hand the Whitham equation  can be viewed as the one-dimensional restriction of the {\it full dispersion KP equation} introduced in \cite{La1} to overcome the \lq\lq bad\rq\rq \, behavior of the dispersion relation of the usual KP equations at low frequencies in $x$ (see also the analysis in  \cite{LS}). We refer to \cite{LPS3} for a further study of the Cauchy problem and to \cite{EhGr} for the existence of localized solitary waves, \lq\lq close\rq\rq \, to the usual KP I ones in the case of strong surface tension):

 \begin{equation}\label{FDbis}
\partial_t u+\tilde 
c_{WW}(\sqrt\epsilon|D^{\epsilon}|)\left(1+\epsilon 
\frac{D_2^2}{D^2_1}\right)^{1/2} u_x+\epsilon \frac{3}{2} uu_x=0,
\end{equation}
with
$$
\tilde c_{WW}(\sqrt \epsilon k)=(1+\beta \epsilon k^2)^{\frac{1}{2}}\left(\frac{\tanh \sqrt \epsilon k}{\sqrt \epsilon k}\right)^{1/2},
$$
where $ \beta\geq0$ is a dimensionless coefficient measuring the surface tension
effects and

$$|D^\epsilon|=\sqrt{D_1^2+\epsilon D_2^2}, \quad D_1=\frac{1}{i} \partial_x,\quad D_2=\frac{1}{i}\partial_y.$$

\vspace{0.3cm}
The link with the full water wave system is via the choice of parameters (see \cite{La1}).
Denoting by $h$ a typical depth of the fluid layer, $a$ a typical amplitude of the wave, $\lambda_x$ and $\lambda_y$ typical wave lengths in $x$ and $y$ respectively, the relevant regime here is when

$$\epsilon\sim \frac{a}{h}\sim \left(\frac{\lambda_x}{\lambda_y}\right)^2\sim \left(\frac{h}{\lambda_x}\right)^2\ll 1.$$

For purely gravity waves, $\beta =0, $ \eqref {FDbis} becomes

\begin{equation}\label{FDter}
\partial_t u+ c_{WW}(\sqrt\epsilon|D^{\epsilon}|)\left(1+\epsilon 
\frac{D_2^2}{D^2_1}\right)^{1/2} u_x+\epsilon \frac{3}{2} uu_x=0,
\end{equation}
with
$$
 c_{WW}(\sqrt \epsilon k)=\left(\frac{\tanh \sqrt \epsilon k}{\sqrt \epsilon k}\right)^{1/2},
$$
 which reduces to the  Whitham equation \eqref{Whit} when $u$ does not depend on $y.$

\vspace{0.3cm}

In presence of surface tension ($\beta>0$) \eqref{FDbis} reduces when $u$ does not depend on $y$ to the capillary Whitham equation (see \cite{RK,DMDK})
\begin{equation}\label{WhitST}
u_t+\widetilde{\mathcal L}_\epsilon u_x+\epsilon uu_x=0,
\end{equation}
where
\begin{equation}\label{dispST}
\widetilde{\mathcal L}_\epsilon =(1+\beta\epsilon|D|^2)^{1/2}\left(\frac{\tanh \sqrt \epsilon |D|}{\sqrt \epsilon |D|}\right)^{1/2}, \;\beta>0.
\end{equation}
Taking the formal limit $\sqrt \epsilon |\xi| \to 0$ in  $\widetilde{\mathcal {L}}_\epsilon,$ \eqref{WhitST} reduces to the KdV equation
\begin{equation}\label{KdVST}
u_t+u_x+\frac{\epsilon}{2}\left(\frac{1}{3}-\beta\right) u_{xxx}+\epsilon uu_x=0.
\end{equation}

\vspace{0.3cm}

Note also that the Whitham equation with surface tension looks, for high frequencies, like the following fractional KdV (fKdV) equation with $\alpha=1/2$
\begin{equation}\label{fKdVST}
u_t+u_x+uu_x-|D|^\alpha u_x=0.
\end{equation}
We refer to \cite{KS, LPS, LPS2} for some properties of the fKdV equations viewed as toy models to study the influence of a \lq\lq weak\rq\rq \, dispersive perturbation on the dynamics of the Burgers equation.

\vspace{0.3cm}
There is a \lq\lq Boussinesq like\rq\rq, system version of \eqref{Whit} for waves propagating in both directions. As for the Whitham equation it cannot been derived directly from the water wave system by a consistent asymptotic analysis but by a rather heuristic and formal argument.

Actually, one obtains a {\it full dispersion system} when in the Boussinesq regime one keeps (formally) the original dispersion of the water waves system (see \cite{La1}, \cite{DKM}, and \cite{AMP} where interesting numerical simulations of the propagation of solitary waves are performed).\footnote{As noticed in \cite{AMP}, the use of nonlocal models for shallow water waves is also suggested in \cite{Za}.}

Setting again $\mathcal L_\epsilon =\left(\frac{\tanh \sqrt \epsilon 
|D|}{\sqrt \epsilon |D|}\right)^{1/2},$ we get with $D=-i\nabla $ or $-i\partial_x:$
\begin{equation}
    \label{FD1d}
    \left\lbrace
    \begin{array}{l}
    \eta_t+\mathcal L_\epsilon ^2u_x+\epsilon (\eta u)_x=0 \\
     u_t+ \eta_x+\epsilon uu_x=0,
 \end{array}\right.
    \end{equation}
when $d=1$ and
\begin{equation}
    \label{FD2d}
    \left\lbrace
    \begin{array}{l}
    \eta_t+\mathcal L_\epsilon^2 \nabla\cdot {\bf u}+\epsilon \nabla \cdot(\eta {\bf u})=0 \\
    {\bf u}_t+\nabla \eta+\frac{\epsilon}{2}\nabla |{\bf u}|^2=0,
 \end{array}\right.
    \end{equation}
when $d=2.$

\vspace{0.3cm}
Taking  the limit $\sqrt \epsilon |\xi| \to 0$ in $\mathcal L_\epsilon,$  \eqref {FD1d} reduces formally to
\begin{equation}
\label{Kaup}
\left\lbrace
\begin{array}{l}
\eta_t+u_x+\frac{\epsilon}{3} u_{xxx}+\epsilon (\eta u)_x=0 \\
u_t+ \eta_x+\epsilon uu_x=0,
\end{array}\right.
\end{equation}
while in the two-dimensional case,  \eqref {FD2d} reduces in the same limit to
\begin{equation}
    \label{FD2dbis}
    \left\lbrace
    \begin{array}{l}
    \eta_t+\nabla \cdot {\bf u}+\frac{\epsilon}{3}\Delta \nabla\cdot {\bf u}+\epsilon \nabla \cdot(\eta {\bf u})=0 \\
    {\bf u}_t+\nabla \eta+\frac{\epsilon}{2}\nabla |{\bf u}|^2=0,
 \end{array}\right.
    \end{equation}
that is to the (linearly ill-posed) system one gets  by expanding to first order the Dirichlet to Neumann operator with respect to $\epsilon$ in the full water wave system (see \cite{La1}). On the other hand, \eqref{FD1d} and \eqref{FD2d} are linearly well-posed. As will be seen below the Cauchy problem for the nonlinear system is locally well-posed under a non physical condition on the initial data though.

System  \eqref{Kaup}   is also known in the Inverse Scattering community as   the Kaup-Kupperschmidt  system (see \cite{Ka, Ku}). It is completely integrable though linearly ill-posed since the eigenvalues of the dispersion matrix are  $\pm i\xi(1-\frac{\epsilon}{3}\xi^2)^{1/2}$ (we refer to \cite{ABM} for an analysis of the ill-posedness of the {\it nonlinear} Kaup and related systems).  It  has explicit solitary waves (see \cite{Ka}). 

The Boussinesq system \eqref{FD1d} can therefore be seen as a (well-posed) regularization of the Kaup-Kupperschmidt  system. It is not known  to be completely integrable (this is unlikely, see in particular the simulations in Section 7).

The full dispersion Boussinesq system has  the following Hamiltonian structure
\begin{displaymath}
\partial_t\begin{pmatrix} \eta\\{{\bf u}}
\end{pmatrix}+J\text{grad}\;H_\epsilon(\eta,{{\bf u}})=0
\end{displaymath}
where
\begin{displaymath}
J=\begin{pmatrix} 0 & \partial_{x} & \partial_{y} \\
                 \partial_{x} & 0 & 0 \\
                 \partial_{y} & 0 & 0 \end{pmatrix},
\end{displaymath}
\begin{displaymath}
H_{\epsilon}(U)=\frac12 \int_{ \R^2}\big(|\mathcal L_\epsilon {{\bf u}}|^2+\eta^2+\epsilon
\eta|{{\bf u}}|^2\big)dxdy,
\end{displaymath}
\begin{displaymath}
U=\begin{pmatrix} \eta  \\
                {{\bf u}}
                  \end{pmatrix},
\end{displaymath}
when $d=2$ and
\begin{displaymath}
\partial_t\begin{pmatrix} \eta\\u
\end{pmatrix}+J\text{grad}\;H_\epsilon(\eta,u)=0
\end{displaymath}
where
\begin{displaymath}
J=\begin{pmatrix} 0&\partial_x\\
\partial_x&0
\end{pmatrix}
\end{displaymath}
and
$$H_\epsilon(\eta,u)=\frac{1}{2}\int_\R (|\mathcal L_\epsilon u|^2+\eta^2+\epsilon u^2\eta)dx,$$
when $d=1.$

We will see in the next section that the  full dispersion Boussinesq system has mathematical properties that make it doubtful as a relevant water wave model.

\vspace{0.5cm}
When surface tension is taken into account, one should replace the 
operator $\mathcal L_\epsilon^2$ by $\mathcal P_\epsilon= 
(I+\beta\epsilon|D|^2)\left(\frac{\tanh(\sqrt \epsilon |D|)}{\sqrt \epsilon |D|}\right)$ where again the parameter $\beta>0$ measures surface tension (see \cite{La1}), yielding a more dispersive full dispersion Boussinesq system. When $\beta>\frac{1}{3},$ this full dispersion Boussinesq system yields in  the limit  $\sqrt \epsilon |\xi| \to 0$ in $\mathcal P_\epsilon,$ Boussinesq systems of the class
$a<0, b=c=d=0$ (see \cite{BCS1}) for which long time (that is on time scales of order $1/\epsilon$) well-posedness is established in \cite{SWX}, Theorem 4.5.

When $\beta<\frac{1}{3},$ the full dispersion Boussinesq system reduces in the formal limit  $\sqrt \epsilon |\xi| \to 0$ in $\mathcal P_\epsilon,$ to an ill-posed system, analogous to the Kaup system in dimension $1.$

\begin{remark}
Another  Whitham-Boussinesq system is introduced in 
\cite{Hu-Pa}.\footnote{This system seems a bit artificial contrary to 
\eqref{FD1d}, \eqref{FD2d} which are the exact counter parts of the 
Whitham equation with respect to the original Boussinesq system} It writes

\begin{equation}
    \label{HPFD1d}
    \left\lbrace
    \begin{array}{l}
    \eta_t+u_x+\epsilon (\eta u)_x=0 \\
     u_t+\mathcal P_\epsilon \eta_x+\epsilon uu_x=0,
 \end{array}\right.
    \end{equation}
when $d=1$ and
\begin{equation}
    \label{HPFD2d}
    \left\lbrace
    \begin{array}{l}
    \eta_t+ \nabla\cdot {\bf u}+\epsilon \nabla \cdot(\eta {\bf u})=0 \\
    {\bf u}_t+\mathcal P_\epsilon \nabla \eta+\frac{\epsilon}{2}\nabla |{\bf u}|^2=0,
 \end{array}\right.
    \end{equation}
when $d=2.$

In  the limit $\sqrt \epsilon |\xi| \to 0$ in $\mathcal L_\epsilon,$  those systems  reduce formally to

\begin{equation}
\label{limHP1d}
\left\lbrace
\begin{array}{l}
\eta_t+u_x+\epsilon (\eta u)_x=0 \\
u_t+\eta_x + \epsilon(\frac{1}{3}-\beta)\eta_{xxx}+\epsilon uu_x=0,
\end{array}\right.
\end{equation}
in dimension one  and to 
\begin{equation}
    \label{limHP2d}
    \left\lbrace
    \begin{array}{l}
    \eta_t+\nabla \cdot {\bf u}+\epsilon \nabla \cdot(\eta {\bf u})=0 \\
    {\bf u}_t+\nabla \eta+ \epsilon(\frac{1}{3}-\beta)\nabla \Delta \eta +\frac{\epsilon}{2}\nabla |{\bf u}|^2=0,
 \end{array}\right.
    \end{equation}
 in the two-dimensional case. Both systems are ill-posed when $0\leq \beta <\frac{1}{3}$ while when $\beta >\frac{1}{3}$ they belong to the class of $a=b=d=0, c<0$ of classical (a,b,c,d) Boussinesq systems for which existence on time scales of order $1/\epsilon$ is established in \cite{SWX}, Theorems 4.6 and 4.7.
    The local well-posedness for \eqref{HPFD1d} is established in \cite{Hu-Pa}.

\end{remark}

\vspace{0.5cm}
According to previous theoretical results and numerical simulations,  one expects at least  three different regimes for the Whitham equation \eqref{Whit}  (without surface tension):

1. Scattering for \lq\lq small\rq\rq \, initial data (see the simulations in  \cite{KS} and in the last Section 7 of the present paper).

2. Finite time blow-up (cusplike), see \cite{KS} for various 
simulations displaying the structure of a shock like blow-up (blow-up 
of the gradient with bounded sup-norm of the solution). The 
occurrence of such phenomena is rigorously proven in  \cite{Hu-Tao} 
for a class of fractional KdV equations and in \cite{Hu} for the 
Whitham equation itself  (see also the numerical simulations in 
\cite{KS}) but these phenomena have probably nothing to do with the breaking of real water waves. In fact, when one keeps the small parameter $\epsilon$ in the equation (that is not done in \cite{Hu, Hu-Tao}), the blow-up should occur on time scales much larger than $1/\epsilon, $ the time scale on which the Whitham equation is supposed to approximate the full water wave system via the KdV equation. We refer again to the simulations in Section 7.

3.  A KdV, long wave  regime. In fact it is shown in  \cite{EGW} 
that  \eqref{Whit} possesses specific solitary waves, close to  those of KdV and formally stable. In this regime one can expect a \lq\lq KdV like\rq\rq \, behavior, namely the soliton resolution, at least on sufficiently long time scales. Those solitary waves and their perturbations are investigated numerically in Section 7 of the present paper.

\vspace{0.3cm} One aim of the present paper is to give further evidence  of the relevance of those conjectures.

\vspace{0.3cm}
The dynamics of the Whitham equation with surface tension \eqref{WhitST} should be different because of the different behavior of the dispersion at high frequencies. This makes the equation more dispersive and the expected dynamics is that of $L^2$ critical KdV type equations. In particular the (expected) finite time blow-up should be similar to that of the $L^2$ critical generalized KdV equation
\begin{equation}\label{cgKdV}
u_t+u^4u_x+u_{xxx}=0,
\end{equation}
or of the modified Benjamin-Ono equation (also $L^2$ critical)
\begin{equation}\label{mBO}
u_t+u^2u_x-\mathcal{H}u_{xx}=0,
\end{equation}
where $\mathcal{H}$ denotes the Hilbert transform.

\vspace{0.3cm}
The rigorous analysis of blow-up for those equations can be found respectively in \cite{MM} and \cite{MP}.


\vspace{0.5cm}
As aforementioned one aim of the present paper is to  give some 
evidence for the above conjectures via mathematical analysis and mainly by careful numerical simulations.

We will also give  some hints on the qualitative behavior of the  full-dispersion Boussinesq systems \eqref{FD1d}, \eqref{FD2d}.

\vspace{0.5cm}
The paper is organized as follows. In  a first section we give the expected error estimates on the correct time scales,  between the solutions of the Whitham  and KdV equation for smooth  initial data. Together with the classical results on the KdV approximation of surface water waves (see \cite{La1}) this implies a rigorous justification of the Whitham equation in the Boussinesq-KdV regime.

The next section concerns the Cauchy problem for the Whitham equation with surface tension. Contrary to  \eqref{Whit} dispersive estimates can be used here to enlarge the space of resolution to the Cauchy problem.

We then comment on the local well-posedness of the Cauchy problem for the Boussinesq full dispersion systems \eqref{FD1d} and \eqref{FD2d}  and of their capillary waves versions. We will see that the well-posedness of the system for pure gravity waves is obtained under a very restrictive, non physical condition (positivity of the wave elevation). When this condition is not satisfied the system is (Hadamard) ill posed and those facts invalidate it as a relevant model for water waves.

The presence of surface tension on the other hand prevents the 
appearance of Hadamard unstable modes when the initial elevation is not positive, the possible unstable modes being bounded then.

The two next sections review (and comment on) known results concerning finite time blow-up and solitary wave solutions.

Finally the last two sections display many accurate numerical 
simulations aiming to  illustrate and to detail various properties of the Whitham equations and systems, allowing  to propose convincing conjectures on their dynamics.

Section 7 is devoted to the Whitham equation. We first construct numerically the solitary waves to the Whitham equation without surface tension and simulate their perturbations for various values  of the small parameter $\epsilon.$

Then we solve numerically the Cauchy problem for Gaussian initial data $\lambda \exp(-x^2).$ Depending on the size of $\epsilon$ and $\lambda$ a finite time blow-up may occur, at a time outside the physically relevant time scales  $O(1/\epsilon)$ though.

The situation is quite different when surface tension is included. Actually  one shows a finite time blow-up very similar to  the one of the $L^2$ critical KdV equation

$$u_t+u^4u_x+u_{xxx}=0.$$

In Section 8 we consider the one-dimensional Whitham-Boussinesq systems, with and without surface tension. We construct numerically solitary waves and study their stability. Then we simulate solutions of the Cauchy problem with initial data satisfying or not the well-posedness condition.

\section{Comparison between the Whitham and KdV equations}

We will compare here the solutions $v$ and $u$ of respectively  the Whitham equation \eqref{Whit} and  the KdV equation \eqref{KdV} with the same initial data $v_0=u_0=\phi\in H^{\infty}(\mathbb R)$.

\smallskip
Here is the main result of this section.

\begin{theorem} \label{maintheo}
Let $\phi \in H^{\infty}(\mathbb R)$. Then, for all $j \in \mathbb N$, $j \ge 0$, there exists
$M_j=M_j(\|\phi\|_{H^{j+8}})>0$  such that
\begin{equation} \label{maintheo.1}
\|(u-v)(t)\|_{H^j_x} \le M_j \epsilon^2t,
\end{equation}
for all $0 \le t \lesssim \epsilon^{-1}$.
\end{theorem}

\begin{remark}

The implicit constant in the notation $t \lesssim \epsilon^{-1}$ 
depends on $\|\phi\|_{H^2}$ for $j=0$ and $1$ and on $\|\phi\|_{H^{j+1}}^{-1}$  for $j\geq 2.$
\end{remark}

\begin{remark}
A similar resul holds {\it mutatis mutandi} for the {\it periodic} problem since we use energy type methods.
\end{remark}

Recall that such a theorem was proved by Bona, Pritchard and Scott \cite{BoPrSc} for the comparison between BBM and KdV. We also refer to Albert and Bona \cite{AlBo} for other comparison results in the long wave regime.

\subsection{The linear case}
In this section, we compare the linear versions of \eqref{KdV} and \eqref{Whit}, \textit{i.e.} the Airy equation
\begin{equation} \label{Airy}
\partial_tu+\partial_xu+\frac{\epsilon}6\partial_x^3u=0,
\end{equation}
and the linear Whitham equation
\begin{equation} \label{linWhitham}
\partial_tv+l(\sqrt{\epsilon}D)\partial_xv=0 \, ,
\end{equation}
where  $l(\sqrt{\epsilon}D)$  denotes the Fourier multiplier of symbol $l(\sqrt{\epsilon}\xi)$ defined by
$$l(\sqrt{\epsilon}\xi)=\left(\frac{\tanh \sqrt \epsilon \xi}{\sqrt \epsilon \xi}\right)^{1/2}.$$

Then, we have the following result
\begin{theorem} \label{lintheo}
Let $\phi \in H^{\infty}(\mathbb R)$. Then, for all $j \in \mathbb N$, $j \ge 0$, there exists $N_j=N_j(\|\phi\|_{H^{j+7}})>0$ such that
the solutions $u$ of \eqref{Airy} and $v$ of \eqref{linWhitham} associated to the same initial datum $\phi$ satisfy
\begin{equation} \label{lintheo.1}
\|\partial_x^j(u-v)(t)\|_{L^{\infty}_x} \le N_j\epsilon^2(1+t) \, ,
\end{equation}
for all $t \in \mathbb R_+$.
\end{theorem}

\begin{remark}
Note that we could also obtain bounds for the differences of $u$ and $v$ in $L^2$ arguing as in the proof of Theorem \ref{maintheo}.
\end{remark}

The following technical result which compares the symbol of $l(\sqrt{\epsilon}D)\partial_x$ and $\partial_x+\frac{\epsilon}6\partial_x^3$ will be needed below.
\begin{lemma} \label{symbcomp}
Assume that $\sqrt{\epsilon}|\xi| \le 1$. Then,
\begin{equation} \label{symbcomp.1}
\Big|l(\sqrt{\epsilon}\xi)\xi -\left(\xi-\frac{\epsilon}6 \xi^3\right)   \Big| \lesssim \epsilon^2|\xi|^5 \, .
\end{equation}
\end{lemma}

\begin{proof}
The proof  follows directly from the expansion
\begin{displaymath}
\left(\frac{\tanh(x)}{x}\right)^{\frac12}=1-\frac16x^2+\mathcal{O}(x^4), \quad \text{for} \quad |x|<1 \, .
\end{displaymath}
\end{proof}

\begin{proof}[Proof of Theorem \ref{lintheo}]
The solutions of \eqref{Airy} and \eqref{linWhitham} are respectively given by the unitary groups
\begin{equation} \label{Airygroup}
u(x,t)=e^{-t(\partial_x+\frac{\epsilon}6\partial_x^3)}\phi(x)=\int_{\mathbb R} e^{-it(\xi-\frac{\epsilon}6\xi^3)}e^{ix\xi}\widehat{\phi}(\xi) \, d\xi \, ,
\end{equation}
and
\begin{equation} \label{linWhithamgroup}
v(x,t)=e^{-tl(\sqrt{\epsilon}D)\partial_x}\phi(x)=\int_{\mathbb R} e^{-itl(\sqrt{\epsilon}\xi)\xi}e^{ix\xi}\widehat{\phi}(\xi) \, d\xi \, .
\end{equation}
Then, it follows from the mean value inequality and Lemma \ref{symbcomp}  that
\begin{displaymath}
\begin{split}
\big|\partial_x^j&\big(u(x,t)-v(x,t)\big) \big|\\ &\le \int_{\sqrt{\epsilon}|\xi|<1}\big|e^{-it(\xi-\frac{\epsilon}6\xi^3)}- e^{-itl(\sqrt{\epsilon}\xi)\xi} \big|  |\xi^j\widehat{\phi}(\xi)| \, d\xi +2\int_{\sqrt{\epsilon}|\xi|>1}|\xi^j\widehat{\phi}(\xi)|\\ &
\lesssim \int_{\sqrt{\epsilon}|\xi|<1}  t\epsilon^2|\xi|^{j+5} |\widehat{\phi}(\xi)| \, d\xi +\int_{\sqrt{\epsilon}|\xi|>1}\epsilon^2|\xi|^{j+4}|\widehat{\phi}(\xi)| \, .
\end{split}
\end{displaymath}
Then, we deduce from the Cauchy-Schwarz inequality that
\begin{displaymath}
\|\partial_x^j(u-v)(\cdot,t)\|_{L^{\infty}_x} \lesssim \epsilon^2(1+t)\|\phi\|_{H^{j+7}} \, ,
\end{displaymath}
which concludes the proof of the theorem.
\end{proof}

\begin{remark}

Note that  fundamental solutions of both the linear Whitham and KdV equations have a quite different behavior.

The KdV one is given by the Airy function
$$G_{KdV}(x,t)=\frac{C}{(\sqrt \epsilon t)^{1/3}}Ai\left( \frac{x-t}{(\sqrt \epsilon t)^{1/3}}\right).
$$

Set

$$K(x,t)=\frac{1}{2\pi}\int_\R e^{ix\xi}e^{it\xi\left(\frac{\tanh \xi}{\xi}\right)^{1/2}}d\xi.$$

\footnote{Note that this function is different from the one used in \cite{EW} to study the properties of the solitary waves of the Whitham equation.}

The fundamental solution of the Whitham equation is thus
$$
K_\epsilon(x,t)=\frac{1}{\sqrt \epsilon}K\left(\frac{x}{\sqrt \epsilon},\frac{t}{\sqrt \epsilon}\right).
$$
One can easily establish that contrary to the Airy function, K is an unbounded function. 

Here we follow the analysis in \cite{BS} in a different setting. 
Actually, we write
$$
\xi\left(\frac{\tanh |\xi|}{|\xi|}\right)^{1/2}=(\sign\; \xi)|\xi|^{1/2}\left(1-\frac{2}{1+e^{2|\xi|}}\right)^{1/2}=(\sign\; \xi)|\xi|^{1/2}+r(\xi)
$$
where $r$ is a continuous function exponentially decaying to zero at $\pm \infty.$
Using the elementary identity
$$
e^{ia}e^{ib}=(1+2i\sin \frac{a}{2}e^{ia/2})e^{ib/2},
$$
one obtains the decomposition
$$
K(x,t)=\int_\R e^{it(\sign \xi)\;i|\xi|^{1/2}}e^{ix\xi}d\xi+\int_\R f_t(\xi)e^{it(\sign \xi)\;|\xi|^{1/2}}e^{ix\xi}d\xi=I_t^1(x)+I_t^2(x),
$$
where
$$
f_t(\xi)=2i\sin\frac{t r(\xi)}{2}e^{it\frac{r(\xi)}{2}}.
$$

Since $f_t$ decays exponentially to zero when $|\xi| \to \infty,$ Riemann-Lebesgue lemma implies that for every $t>0,$ $I_t^2$ is a continuous function of $x$ decaying to zero at infinity.

On the other hand, following the analysis in Section 3 of \cite{BS} one can prove that $I_t^1$ decays algebraically  to $0$ when $|x|\to \infty$ for t fixed while for instance $I_1^1(x)\sim |x|^{-3/2}\exp (\frac{i\xi}{x})$ when $x\to 0$ for some non zero $\xi$.
\end{remark}
\vspace{0.3cm}

Although we will not use it, we recall for the sake of completeness a dispersive estimate derived in \cite{Mes}, Theorem 2.5,  on the free Whitham group (see also \cite{Me}, Lemma 2.4). Note the difference with the classical $L^1-L^\infty$ estimates on the Airy group.  We denote $S_\epsilon(t)=e^{itl(\sqrt \epsilon D)\partial_x}.$



\begin{theorem}\label{Meso}
There exists $C>0$ independent of $\epsilon$ such that for any $\phi\in  \mathcal S(\R^2)$ the following estimates hold :
$$1. \;|S_\epsilon(t)\phi)|_\infty \leq  C\left(\frac{1}{\epsilon^{1/4}(1+t/\sqrt \epsilon)^{1/8}}+\frac{1}{(1+t/\sqrt \epsilon)^{1/2}}\right)(|\phi|_{H^1}+|x\partial_x\phi|_2).$$
 $$2. \;|S_\epsilon(t)\phi)|_\infty \leq  C\left(\frac{1}{\epsilon^{3/4}(1+t/\sqrt \epsilon)^{1/3}}|\phi|_{L^1}+\frac{1}{(1+t/\sqrt \epsilon)^{1/2}}(|\phi|_{H^1}+|x\partial _x\phi|_2)\right).$$
 $$3.\;|S_\epsilon(t)\phi)|_\infty \leq  C\left(\frac{1}{\epsilon^{3/4}(1+t/\sqrt \epsilon)^{1/3}}|x\phi|_{L^2}+\frac{1}{(1+t/\sqrt \epsilon)^{1/2}}(|\phi|_{H^1}+|x\partial _x\phi|_2)\right).$$

\end{theorem}
\subsection{\textit{A priori} estimates on $u$ and $v$}

It is well-known that the KdV equation is well-posed in $H^{\infty}(\mathbb R)$. Moreover, by using the complete integrability of KdV and in particular the fact that KdV possesses an infinite number of conserved quantities, one can get global bounds at the $H^j$-level for any $j \ge 0$. We refer for example to Saut \cite{Sa}, Bona and Smith \cite{BoSm} and Bona, Pritchard and Scott \cite{BoPrSc} and the references therein. 

\begin{proposition} \label{aprioriKdV}
Let $\phi \in H^{\infty}(\mathbb R)$. Then there exists a unique solution $u \in C([0,+\infty) : H^{\infty}(\mathbb R))$ to \eqref{KdV} such that $u(\cdot,0)=\phi$. Moreover, the flow map data-solution $\phi \mapsto u$ is continuous from $H^{\infty}(\mathbb R)$ into $C([0,+\infty) : H^{\infty}(\mathbb R))$.

Furthermore, the following bounds hold true. For every $j \ge 0$, there exists $C_j=C_j(\|\phi\|_{H^j})$ (note that the $C_j$ can be chosen to be non-increasing functions of their arguments) such that
\begin{equation} \label{boundKdV}
\|u(t)\|_{H^j} \le C_j(\|\phi\|_{H^j}), \quad \forall \, t \ge 0 \, .
\end{equation}
\end{proposition}

 \medskip
Being a skew-adjoint perturbation of the Burgers equation, the Whitham equation is also trivially well-posed on $H^{\infty}(\mathbb R)$ (and also in $H^s(\mathbb R)$ for $s>\frac32$)\footnote{Similarly to the Burgers equation, (see {\it eg} \cite{LPS}) the Cauchy problem for the Whitham equation is expected to be ill-posed in $H^{\frac32}(\R)$ (see \cite{Hu2} for  weaker results on related fractional KdV equations).} but on a time of interval of length $1/\epsilon$. 
\begin{proposition} \label{aprioriWhitham}
Let $\phi \in H^{\infty}(\mathbb R)$. Then there exist a positive time $T \sim \epsilon^{-1}$, a unique solution $v \in C([0,T] : H^{\infty}(\mathbb R))$ to \eqref{Whit} such that $v(\cdot,0)=\phi$. Moreover, the flow map data-solution $\phi \mapsto v$ is continuous from $H^{\infty}(\mathbb R)$ into $C([0,T] : H^{\infty}(\mathbb R))$.

Furthermore, the following bounds hold true. For every $j \in \mathbb N$, $j \ge 2$
\begin{equation} \label{boundWhitham}
\|u(t)\|_{H^j} \le 2\|\phi\|_{H^j}, \quad \forall \, 0 \le t \lesssim \epsilon^{-1} \, .
\end{equation}
Note that the implicit constant in \eqref{boundWhitham} depends on $\|\phi\|_{H^j}$ as in \eqref{boundWhitham.1}.
\end{proposition}

\begin{proof} We only explain how to prove \eqref{boundWhitham}. Let $v \in C([0,T] : H^{\infty}(\mathbb R))$ be a solution of \eqref{Whit} with initial datum $v(\cdot,0)=\phi$.

Let $J^s$ denote the Bessel potential of order $-s$, \textit{i.e.} 
$$(J^sf)^{\wedge}(\xi)=(1+\xi^2)^{\frac{s}2}\widehat{f}(\xi) \, . $$
Then, it follows from the Kato-Ponce commutator estimates \cite{KaPo} and integrations by parts that 
\begin{displaymath}
\frac{d}{dt}\|J^sv \|_{L^2}^2 \le c\varepsilon \| \partial_xv\|_{L^{\infty}}\|J^sv \|_{L^2}^2 \, ,
\end{displaymath}
for any $s>0$. If $s>\frac32$, we deduce from the Sobolev embedding $H^{s-1}(\mathbb R) \hookrightarrow L^{\infty}(\mathbb R)$ that 
\begin{displaymath}
\frac{d}{dt}\|J^sv \|_{L^2}^2 \le c\varepsilon\|J^sv \|_{L^2}^3 \,.
\end{displaymath}
Hence, we deduce from a classical ODE argument that 
\begin{equation} \label{boundWhitham.1}
\| v(t) \|_{H^s} \le 2\| \phi \|_{H^s}, \quad \text{if} \quad 0 \le t \le (2c\epsilon\|\phi\|_{H^s})^{-1} \, .
\end{equation}
This finishes the proof of Proposition \ref{aprioriWhitham}.
\end{proof}

\medskip
Finally, we also need a bound on $\pi_{\epsilon}(D)u$ where $u$ is the  solution of \eqref{KdV} and $\pi_{\epsilon}(D)$ is the Fourier multiplier of symbol $\pi_{\epsilon}(\xi)$ defined by
\begin{equation} \label{pi}
\pi_{\epsilon}(\xi)= i\Big((\xi-\frac{\epsilon}6 \xi^3)-l(\sqrt{\epsilon}\xi)\xi\Big) \, .
\end{equation}

\begin{proposition} \label{aprioripiu}
Let $\phi \in H^{\infty}(\mathbb R)$ and let $u$ be the solution  of \eqref{KdV} evolving from $\phi$ obtained in Proposition \ref{aprioriKdV}. Then, for all $j \in \mathbb N$, $j \ge 0$, there exists $A_j=A_j(\|\phi\|_{H^{j+8}})>0$ such that
\begin{equation} \label{aprioripiu.1}
\|\partial_x^j\pi_{\epsilon}(D)u(t)\|_{L^2} \le A_j\epsilon^2 \, ,
\end{equation}
for all $0 \le t \le \epsilon^{-1}$.
\end{proposition}

\begin{proof}
We apply the operator $\partial_x^j\pi_{\epsilon}(D)$ to \eqref{KdV}, multiply the equation by $\partial_x^j\pi_{\epsilon}(D)u$  integrate in the space variable over $\mathbb R$ and integrate by parts and use the Cauchy-Schwarz inequality to deduce that
\begin{equation} \label{aprioripiu.2}
\begin{split}
\frac12\frac{d}{dt}\|\partial_x^j\pi_{\epsilon}(D)u\|_{L^2}^2 &=-\epsilon\frac12\int_{\mathbb R} \pi_{\epsilon}(D)\partial_x^{j+1}(u^2)\, \pi_{\epsilon}(D)\partial_x^ju \, dx  \\ &\lesssim \epsilon\|\partial_x^{j+1}\pi_{\epsilon}(D)(u^2) \|_{L^2}\|\partial_x^j\pi_{\epsilon}(D)(u)\|_{L^2} \, .
\end{split}
\end{equation}

Now, we introduce a cut-off function $\eta \in C_0^{\infty}(\mathbb R)$ such that $\text{supp} \, (\eta) \subset [-1,1] $ and $\eta=1$ over $[-1/2, 1/2]$ and we define the Fourier multiplier $P_{< 1/\sqrt{\epsilon}}$ of symbol $\eta(\sqrt{\epsilon}\cdot)$ and $P_{\ge 1/\sqrt{\epsilon}}=1-P_{< 1/\sqrt{\epsilon}}$. Then, the triangle inequality yields
\begin{equation} \label{aprioripiu.3}
\begin{split}
\|\partial_x^{j+1}\pi_{\epsilon}&(D)(u^2) \|_{L^2} \\
& \le \|P_{< 1/\sqrt{\epsilon}}\partial_x^{j+1}\pi_{\epsilon}(D)(u^2) \|_{L^2}
+\|P_{\ge 1/\sqrt{\epsilon}}\partial_x^{j+1}\pi_{\epsilon}(D)(u^2) \|_{L^2} \ ,
\end{split}
\end{equation}
and we need to estimate both terms on the right-hand side of \eqref{aprioripiu.3}.

To control the first one, we use Plancherel's identity, Lemma \ref{symbcomp} and the fact that $H^s(\mathbb R)$ is a Banach algebra for $s>\frac12$ to deduce that
\begin{displaymath}
\|P_{< 1/\sqrt{\epsilon}}\partial_x^{j+1}\pi_{\epsilon}(D)(u^2) \|_{L^2} \lesssim \epsilon^2 \|u^2\|_{H^{j+6}_x} \lesssim \epsilon^2\|u\|_{H^{j+6}}^2 \, ,
\end{displaymath}
which together with Proposition \ref{aprioriKdV} gives that
\begin{equation} \label{aprioripiu.4}
\|P_{< 1/\sqrt{\epsilon}}\partial_x^{j+1}\pi_{\epsilon}(D)(u^2) \|_{L^2} \lesssim \epsilon^2 C_{j+6}(\|\phi\|_{H^{j+6}})^2 \, .
\end{equation}

To control the second one, we use again Plancherel's identity to obtain that
\begin{displaymath}
\||P_{\ge 1/\sqrt{\epsilon}}\partial_x^{j+1}\pi_{\epsilon}(D)(u^2) \|_{L^2}^2 \lesssim \int_{|\xi| \gtrsim 1/\sqrt{\epsilon}} (1+|\xi|^2)^3|\xi|^{2(j+1)}|(u^2)^{\wedge}(\xi)|^2 \, d\xi \, .
\end{displaymath}
Now, observe that $1 \lesssim \epsilon^4|\xi|^8$ on the support of the integral, so that
\begin{equation} \label{aprioripiu.5}
\||P_{\ge 1/\sqrt{\epsilon}}\partial_x^{j+1}\pi_{\epsilon}(D)(u^2) \|_{L^2}^2  \lesssim \epsilon^2 \|u^2\|_{H^{j+8}} \lesssim \epsilon^2 \|u\|_{H^{j+8}}^2 \lesssim \epsilon^2 C_{j+8}(\|\phi\|_{H^{j+8}})^2  \, ,
\end{equation}
since $H^8(\mathbb R)$ is a Banach algebra and where we used Proposition \ref{aprioriKdV} on the last inequality.

Then, we deduce gathering \eqref{aprioripiu.2}--\eqref{aprioripiu.5} that
\begin{equation} \label{aprioripiu.6}
\frac{d}{dt}\|\partial_x^j\pi_{\epsilon}(D)u\|_{L^2}^2 \lesssim \epsilon^3 C_{j+8}(\|\phi\|_{H^{j+8}})^2\|\partial_x^j\pi_{\epsilon}(D)u\|_{L^2} \, .
\end{equation}
Therefore, we deduce integrating between $0$ and $t$ that
\begin{equation} \label{aprioripiu.7}
\|\partial_x^j\pi_{\epsilon}(D)u(t)\|_{L^2} \le \|\partial_x^j\pi_{\epsilon}(D)\phi\|_{L^2}+C_{j+8}(\|\phi\|_{H^{j+8}})^2 \epsilon^3t \, .
\end{equation}

Finally, arguing as above, we get from Lemma \ref{symbcomp} and Plancherel's identity that
\begin{equation} \label{aprioripiu.8}
\begin{split}
\|\partial_x^j\pi_{\epsilon}(D)\phi\|_{L^2} &\le \|P_{< 1/\sqrt{\epsilon}}\partial_x^{j}\pi_{\epsilon}(D)\phi \|_{L^2}+\|P_{\ge 1/\sqrt{\epsilon}}\partial_x^{j}\pi_{\epsilon}(D)\phi \|_{L^2} 
\\ & \lesssim  \epsilon^2\| \phi\|_{H^{j+5}}+\epsilon^2\| \phi\|_{H^{j+7}} \, .
\end{split}
\end{equation}

 Then, we deduce combining \eqref{aprioripiu.7} and \eqref{aprioripiu.8} that
 \begin{equation} \label{aprioripiu.9}
\|\partial_x^j\pi_{\epsilon}(D)u(t)\|_{L^2} \le A_j(\|\phi\|_{H^{j+8}}) \epsilon^2(1+\epsilon t) \le  2A_j(\|\phi\|_{H^{j+8}}) \epsilon^2 \, ,
\end{equation}
as soon as $0 \le t \le t/\epsilon$.
\end{proof}

\subsection{Proof of Theorem \ref{maintheo}}

Let $z=u-v$ and $j \in \mathbb N$, $j \ge 0$. It is deduced from equations \eqref{Whit} and \eqref{KdV} that $z$ solves the initial value problem
\begin{equation} \label{diff}
\begin{cases}
\partial_tz+l(\sqrt{\epsilon}D)\partial_xz+\pi_{\epsilon}(D)u+\epsilon v\partial_xz-\epsilon z\partial_xu=0 \\
w(\cdot,0)=0 \, ,
\end{cases}
\end{equation}
where $\pi_{\epsilon}(D)$ is the Fourier multiplier of symbol $\pi_{\epsilon}(\xi)$ defined in \eqref{pi}.

Note that the equation of $z$ is well defined for $0<t \lesssim 1/\epsilon$.

\medskip
Differentiate $j$ times, multiply the equation \eqref{diff} by $\partial_x^jz$, integrate in the space variable over $\mathbb R$ and integrate by parts to deduce that
\begin{equation} \label{maintheo.3}
\begin{split}
\frac12\frac{d}{dt}\int_{\mathbb R} (\partial_x^jz)^2dx&=
-\int_{\mathbb R}\partial_x^j \pi_{\epsilon}(D)u\partial_x^jz \, dx
-\epsilon\int_{\mathbb R}\partial_x^j(v\partial_xz) \partial_x^jz \, dx\\ & \quad -\epsilon\int_{\mathbb R}\partial_x^j(z\partial_xu)\partial_x^j z \, dx \, .
\end{split}
\end{equation}
Thus, the Cauchy-Schwarz inequality, the Leibniz rule and integration by parts yield
\begin{equation} \label{maintheo.4}
 \frac{d}{dt}\|\partial_x^jz\|_{L^2}^2 \lesssim \|\partial_x^j\pi_{\epsilon}(D)u\|_{L^2}\|\partial_x^jz\|_{L^2}+\epsilon\big( \|\partial_xu\|_{L^{\infty}}+ \|\partial_xv\|_{L^{\infty}})\|\partial_x^jz\|_{L^2}^2
\end{equation}
in the cases $j=0$ and $j=1$ and 
\begin{equation} \label{maintheo.4b}
 \frac{d}{dt}\|\partial_x^jz\|_{L^2}^2 \lesssim \|\partial_x^j\pi_{\epsilon}(D)u\|_{L^2}\|\partial_x^jz\|_{L^2}+\epsilon\sum_{k=1}^j\big( \|\partial_x^ku\|_{L^{\infty}}+ \|\partial_x^kv\|_{L^{\infty}}\big)\|z\|_{H^j}^2
\end{equation}
in the cases where $j \ge 2$.

Now, on the one hand, we get by using the Sobolev embedding and Propositions \ref{aprioriKdV} and \ref{aprioriWhitham} that
\begin{equation} \label{maintheo.5}
\begin{split}
\sum_{k=1}^j\big(\|\partial_x^ku(\cdot,t)\|_{L^{\infty}}&+\|\partial_x^kv(\cdot,t)\|_{L^{\infty}}\big)\\ & \lesssim \|u(\cdot,t)\|_{H^{\kappa}}+\|v(\cdot,t)\|_{H^{\kappa}} \lesssim K_1(\|\phi\|_{H^{\kappa}})=:K_1 \, ,
\end{split}
\end{equation} 
for all $0 \le t \lesssim \epsilon^{-1}$ (where the implicit constant depends on the $\|\cdot\|_{H^{\kappa}}$ norm of $\phi$ as explained in the proof of proposition \ref{aprioriWhitham}). Here, we used the notation $\kappa=2$ for $j=0$ or $1$ and $\kappa=j+1$ for $j \ge 2$. On the other other hand, by applying Proposition \ref{aprioripiu}, it follows that
\begin{equation} \label{maintheo.6}
 \|\partial_x^j\pi_{\epsilon}(D)u\|_{L^2} \le \epsilon^2A_{j+8}(\|\phi\|_{H^{j+8}})=:\epsilon^2K_2 \, ,
\end{equation}
for all $0 \le t \le \epsilon^{-1}$.

Therefore, we deduce gathering \eqref{maintheo.4}, \eqref{maintheo.5}, \eqref{maintheo.6} and using Gronwall's inequality that
\begin{equation} \label{maintheo.7}
\|z(t)\|_{H^j} \le  K_2  \frac{e^{K_1\epsilon t}-1}{K_1\epsilon}\epsilon^2
\le K_2e^{K_1}\epsilon^2 t,
\end{equation}
whenever $0\le t \lesssim \epsilon^{-1}$, which concludes the proof of Theorem \ref{maintheo}.

\begin{remark}
A similar comparison result can be established between the Whitham equation with surface tension \eqref{WhitST} and the KdV equation \eqref{KdVST} by using the expansion 
$$(1+\beta x^2)^{1/2}\left(\frac{\tanh x}{x}\right)^{1/2}=1-\frac{x^2}{6}(1-3\beta)+O(x^4).$$
\end{remark}

\vspace{0.5cm}

\section{The Cauchy problem for the capillary Whitham equation}
As was already noticed the Cauchy problem for the Whitham equations \eqref{Whit} and \eqref{WhitST} is trivially well-posed in $H^s(\R),\; s>\frac{3}{2}$ on time scales of order $\epsilon^{-1}.$

This result can be improved (by enlarging the space of resolution) for the capillary Whitham equation \eqref{WhitST} by using its dispersive properties. Actually one gets, see  \cite{LPS2} where general fractional KdV (fKdV) equations are considered:

\begin{theorem}\label{WST}
Assume that $s >\frac{21}{16}$.
 Then, for every $u_0 \in H^s(\mathbb R)$, there exist a
positive time $T_\epsilon =T_\epsilon(\|u_0\|_{H^s})=0(1/\sqrt \epsilon)$ (which can be chosen as a non-increasing function of its argument), and a unique solution $u$ to
\eqref{WhitST} satisfying $u(\cdot,0)=u_0$ such that
\begin{equation} \label{maintheo1}
u \in  C([0,T_\epsilon]:H^s(\mathbb R)) \quad \text{and} \quad \partial_xu \in L^1([0,T_\epsilon]:L^{\infty}(\mathbb R))  .
\end{equation}
Moreover, for any $0<T'<T_\epsilon$, there exists a neighborhood
$\mathcal{U}$ of $u_0$ in $H^s(\mathbb R)$ such that the flow map
data-solution
\begin{equation} \label{maintheo2}
S^s_{T'}: \mathcal{U} \longrightarrow C([0,T'];H^s(\mathbb R)) , \
u_0 \longmapsto u,
\end{equation}
is continuous.
\end{theorem}
\begin{remark}
1. Since the value $\alpha =\frac{1}{2}$ is $L^2$ critical for the fKdV equation we conjecture that the Cauchy problem for \eqref{WhitST} is globally well posed for initial data in the energy space $H^{\frac14}(\R)$ having a sufficiently small $L^2$ norm.

2. This result was recently improved in \cite{MPV} where the local well-posedness is obtained for $s>\frac{9}{8}.$

3. For \lq\lq large enough\rq\rq \, initial data, one expects a finite time blow-up silmilar to that of the $L^2$-critical generalized KdV equation or to the cubic-Benjamin-Ono equation proven respectively in \cite{MM} and \cite{MP}, as displayed in the simulations of Section 7.
\end{remark}

\begin{remark}
Using the same energy methods one can prove a result similar to that of Theorem \ref{maintheo} between solutions of \eqref{WhitST} and \eqref{KdVST}.
\end{remark}

\section{The Cauchy problem for the Boussinesq-Whitham systems}
In both cases (without or with surface tension) the Boussinesq-Whitham systems are linearly well-posed since the linearized systems write, say in dimension one
\begin{displaymath}
\partial_t\begin{pmatrix} \eta\\u
\end{pmatrix}+\partial_x A\begin{pmatrix} \eta\\u
\end{pmatrix}=0
\end{displaymath}
where the Fourier transform of the matrix operator A has real eigenvalues
$$
\lambda_{\pm}(\xi)=\pm\left(\frac{\tanh \sqrt \epsilon |\xi|}{\sqrt \epsilon|\xi|}\right)^{1/2}
$$
in absence of surface tension and
$$
\lambda_{\pm}(\xi)=\pm (1+\epsilon\beta|\xi|^2)^{1/2}\left(\frac{\tanh \sqrt \epsilon |\xi|}{\sqrt \epsilon|\xi|}\right)^{1/2}
$$
in presence of surface tension.

\vspace{0.3cm}
The theory of the Cauchy problem for the nonlinear system without surface tension is relatively straightforward and we focus on the case of space dimension one. Since the operator $\mathcal L_\epsilon^2 \partial_x$ has order zero,  \eqref{FD1d} is an order zero perturbation of a first order system, more precisely it writes

\begin{equation}\label{hyp1}
\partial_t U+\mathcal A_\epsilon(U)\partial_x U+\mathcal H_\epsilon U=0
\end{equation}
where $U=\begin{pmatrix}
\eta\\
u
\end{pmatrix}$, $\mathcal A_\epsilon(U)= \begin{pmatrix}
\epsilon u&\epsilon\eta\\
1&\epsilon u
\end{pmatrix},$ $\mathcal H_\epsilon U=\begin{pmatrix}
\mathcal L_\epsilon^2 \partial_x u\\
0
\end{pmatrix}.$

The matrix $\mathcal A_\epsilon(U)$ has eigenvalues $\lambda$ satisfying
$$
(\epsilon u-\lambda)^2=\epsilon \eta
$$
so that the system 
$$
\partial_t U+\mathcal A_\epsilon(U)\partial_x U=0
$$ 
is hyperbolic in the regions where 
\begin{equation}\label{hypcond2}
\eta >0,
\end{equation}
and Hadamard ill-posed if this condition is violated.

Actually, under the condition

\begin{equation}\label{hypcond}
\eta \geq C_0>0,
\end{equation}
the system can be symmetrized via the positive definite symmetrizer $S(U)=\begin{pmatrix}
1&0\\
0&\epsilon\eta
\end{pmatrix}.$

The same process can be applied to \eqref{FD1d} yielding a symmetric hyperbolic system perturbed by the order zero operator $\mathcal H_\epsilon$ namely 

\begin{equation}\label{whitSym}
S(U)U_t+\epsilon\begin{pmatrix}
u&\eta\\
\eta&\epsilon u\eta
\end{pmatrix} U_x+\mathcal H_\epsilon U=0.
\end{equation}

Setting $\zeta=\eta+N_0$ where $N_0>0$ is a fixed constraint, the standard theory of symmetrizable hyperbolic systems (see {\it eg} \cite{Ma}) 
imply the local well-posedness of the Cauchy problem for $(\zeta,u) \in H^s(\R)\times H^s(\R),\; s>\frac{3}{2}$ for initial data $\zeta_0$ such that $\eta_0$ is sufficiently small.   Similar arguments in the two-dimensional case yield well-posedness in 
$H^s(\R^2)\times H^s(\R^2),\; s>2. $ 

\vspace{0.3cm}
 Note  that the condition \eqref{hypcond2} implies that the wave is always of elevation which seems to be the case of solitary wave solutions (see Section 7). On the other hand, condition \eqref{hypcond}  implies that the wave  cannot tend to zero at infinity and thus the perturbations of the solitary wave solutions are excluded from the range of well-posed initial data. Those facts are  not physically realistic, invaliding the Boussinesq-Whitham system as a relevant water waves model.

 \begin{remark}
 We do not know of a local well-posedness result for \eqref{FD1d} or \eqref{FD2d} under the assumption \eqref{hypcond2}.
 \end{remark}

To illustrate the (Hadamard) instabilities occurring when condition \eqref{hypcond} is violated, let consider the linearization around $(\eta,u)=(-c,0)$ where $c$ is a positive constant. The linearized system has eigenvalues $\lambda_{\pm}$ where

 $$\lambda_{\pm}=\pm i\xi\left(\frac{\tanh \sqrt \epsilon \xi}{\sqrt \epsilon \xi}-c\epsilon\right )^{1/2}.$$
 
 All modes are unstable when $\epsilon >1/c$ while when $\epsilon<1/c$ all modes corresponding to $|\xi|\geq \frac{x_{c,\epsilon}}{\sqrt \epsilon}$ are unstable where $x_{c,\epsilon}$ is the unique positive solution of  $\frac {\tanh x}{x}=c\epsilon.$ Note that $x_{c,\epsilon}\to \infty$ as 
$\epsilon \to 0.$  Thus the  set of stable modes get larger and larger  when for a fixed $c,$ the small parameter $\epsilon$ tends to $ 0.$

 When the nonlocal term is removed, all modes are unstable, for any $c>0.$ The effect of dispersion is thus here to create a range of linearly stable modes when $c<1,$ getting larger and larger when $\epsilon \to 0.$

\vspace{0.3cm}
Things are  different for Boussinesq-Whitham systems in presence of surface tension since dispersive effects play a significant role here. 

In fact, surface tension prevents the appearance of (Hadamard) unstable modes for the linearized system at $(\eta,u)=(-c,0)$.

Actually the linearized system  eigenvalues are now

$$
\lambda_{\pm}=\pm i\xi\left((1+\beta\epsilon \xi^2)\frac{\tanh \sqrt \epsilon \xi}{\sqrt \epsilon \xi}-c\epsilon\right )^{1/2}.
$$

Whatever the values of $c$ and $\epsilon$ the possible unstable modes are bounded and there are no more Hadamard instabilities.

The effect of surface tension is thus to suppress the Hadamard instabilities. 

\vspace{0.3cm}
A convenient way to get an idea of the nature of the full system  is to 
derive  an equivalent system by diagonalizing the linear part of the system. More precisely, we define

\begin{displaymath}
\widehat{A}(\xi)=i\xi\begin{pmatrix} 0 & (1+\beta \epsilon \xi^2)\frac{\tanh \sqrt \epsilon|\xi|}{\sqrt \epsilon |\xi|}  \\
                 1  & 0  \\

                  \end{pmatrix}
\end{displaymath}
the Fourier transform of the dispersion matrix with eigenvalues 
$$
\pm i\xi(1+\beta \epsilon \xi^2)^{1/2}\left (\frac{\tanh \sqrt \epsilon|\xi|}{\sqrt \epsilon |\xi|}\right )^{1/2} .
$$
Setting
\begin{displaymath}
U=\begin{pmatrix} \eta\\u
\end{pmatrix}
\end{displaymath}
and
\begin{displaymath}
W=\begin{pmatrix} \zeta \\ v
\end{pmatrix}=P^{-1}U, \quad
P^{-1}=\frac{1}{2}\begin{pmatrix} \tilde{\mathcal L_\epsilon}^{-1} & 1 \\
 \tilde{\mathcal L_\epsilon}^{-1}& -1
\end{pmatrix},\; P=
\begin{pmatrix} \tilde{\mathcal L_\epsilon}&\tilde{\mathcal L_\epsilon}\\
1&-1
\end{pmatrix}
\end{displaymath}
the linear part of  the system is diagonalized as
$$
W_t+\partial_xDW=0,
$$
where
\begin{displaymath}
D=\begin{pmatrix}  \tilde{\mathcal L_\epsilon}&0\\
0&- \tilde{\mathcal L_\epsilon}
\end{pmatrix},
\end{displaymath}
that is two dispersive equations of order $\frac{1}{2}.$

\vspace{0.3cm}

In the $W$ variable the complete system writes now
\begin{equation}\label{newFD1d}
\partial _t W+\partial_x DW+\epsilon N_\epsilon(W)=0,
\end{equation}
where
$$
N_\epsilon(W)=\begin{pmatrix} \mathcal L_\epsilon^{-1}[(\zeta -v)\mathcal L_\epsilon\partial _x(\zeta+v))+\mathcal L_\epsilon(\zeta+v)\partial_x(\zeta-v)]+(\zeta-v)\partial_x (\zeta-v)\\
 \mathcal L_\epsilon^{-1}[(\zeta -v)\mathcal L_\epsilon\partial _x(\zeta+v))+\mathcal L_\epsilon(\zeta+v)\partial_x(\zeta-v)]-(\zeta-v)\partial_x(\zeta-v)\end{pmatrix}.
 $$ 

While the linear part is of order $\frac{1}{2},$ the nonlinear terms are formally of order $1.$ However the structure of the  nonlocal nonlinear terms  might prevent the  implementation of  the methods used in \cite{LPS2}  for fractional type KdV equations to obtain the local well-posedness of the Cauchy problem on time scales of order $1/\sqrt \epsilon.$

On the other hand one might think of proving the local well-posedness by using an elementary energy method, reminiscent of the one 
used in \cite{SWX} for the $(-1,0,0,0)$ Boussinesq system (see 
\cite{SWX} Theorem 3.1) or a symmetrization technique as in 
\cite{SWX} Theorem 4.5. However those methods do not seem to work 
since  a control of the $H^s\times H^{s+1/2},\quad s>\frac{1}{2}$ norm on $(\eta,u)$ is not enough to control the nonlinear terms (in case of the $(-1,0,0,0)$ Boussinesq system one gets a control  on the $H^s\times H^{s+1}$ norm).

\section{Finite time blow-up for the Whitham equation}

It has been established in \cite{Hu} that solutions of the Whitham 
equation \eqref{Whit} may blow-up in finite time, the blow-up being "shock-like", that is blow-up of the gradient with bounded $L^\infty$ norm. The numerical simulations in \cite{KS} display the cusp nature of the singularity depending on the sign of the initial data. On the other hand those works consider the Whitham equation \eqref{Whit} with $\epsilon=1$, thus they do not address the long-wave, KdV limit. In particular the crucial dependence of the blow-up time with respect to $\epsilon$ is not obtained. Note that this blow-up time should be larger that $1/\epsilon,$ the time scale on which the KdV and Whitham equations are close and asymptotic models for the propagation of long, weakly nonlinear surface water waves for which no singularities are expected.

Note also that a quite different finite time blow-up is expected for the Whitham equation with surface tension \eqref{WhitST} that cannot be anymore viewed as a "weak" dispersive perturbation of the Burgers equation. In fact,  as was already noticed, it is reminiscent for large frequencies of the $L^2$ critical fractional KdV equation

\begin{equation}\label{cfKdV}
u_t+u_x+uu_x-D^{\frac{1}{2}}u_x=0
\end{equation}
for which a finite time blow-up {\it \`a la Martel-Merle \cite{MM, MP}} is expected (but not yet proven), see the numerical simulations in \cite{KP, KP2, KS}.

\section{Solitary waves}
\subsection{Whitham without surface tension}
We will focus here on non periodic solitary wave solutions of the Whitham equation \eqref{Whit}, that is solutions of the form  $u(x-ct).$

 We refer to \cite{BKN, EK, EK2, HJ, SKCK} for interesting (theoretical and numerical) studies on {\it periodic} traveling waves. In particular the existence of a global bifurcation branch of $2\pi-$ periodic smooth traveling wave solutions is established in \cite{EK2}.

The existence of solitary wave solutions to the Whitham equation \eqref{Whit} decaying to zero to infinity and close to the KdV soliton has been proven in \cite{EGW} by variational methods. Symmetry and decay properties of Whitham solitary waves are established in \cite{BEP}. The analysis in \cite{EGW} is made on \eqref{Whit} with $\epsilon =1$ under the scaling

$$u(x)=\epsilon^\alpha w(\epsilon^\beta x)$$

where $2\alpha-\beta=1$ so that $\frac{1}{2}\int_\R u^2=\epsilon$.


An interesting issue is that of the transverse stability of the Whitham solitary wave in the framework of the full dispersion KP equation \eqref{FDbis}. Considering the similar problem for the usual KP I/II equations (see \cite{RT} and the references therein) one can conjecture that  the Whitham solitary wave is transversally stable when the surface tension parameter $\beta$ is less than $\frac{1}{3}$ and unstable otherwise. We intend to go back to this issue in a subsequent paper.

\vspace{0.5cm}
We do not know of any rigorous result on the existence of solitary wave solutions to  the full dispersion  Boussinesq system.

The following computations give some evidence to the existence of solitary wave solutions "close" to the KdV soliton.

Recall that the system reads

\begin{align}
    \eta_{t}+T^{2}_{\epsilon}u_{x}+\epsilon (\eta u)_{x}& =0,
    \nonumber\\
    u_{t}+\eta_{x}+\frac{\epsilon}{2}(u^{2})_{x} & =0
    \label{sys},
\end{align}
where
$\hat{T}_{\epsilon}=(\tanh(k\sqrt{\epsilon})/(k\sqrt{\epsilon}))^{1/2}$. In
the formal limit $\epsilon\to 0$, the system reduces to the wave equation.

We are here interested in localized travelling wave solutions to the
system (\ref{sys}). Putting $u(x,t)=U(x-ct)$ and $\eta(x,t)=N(x-ct)$,
we find after integration
\begin{align}
    -cN+T^{2}_{\epsilon}U+\epsilon NU& =0,
    \nonumber\\
    N & =cU-\frac{\epsilon}{2}U^{2}
    \label{travelsys}.
\end{align}
Eliminating $N$ from the first equation of (\ref{travelsys}) via  the
second, we find
\begin{equation}
    (T^{2}_{\epsilon}-c^{2})U+\frac{3\epsilon
    c}{2}U^{2}-\frac{\epsilon^{2}}{2}U^{3}=0.
    \label{travel}
\end{equation}

Writing $c = 1+\alpha \epsilon$ with $\alpha>0$, we get, by performing a formal expansion in $\epsilon$ and neglecting terms of  order $\epsilon$ or higher in
(\ref{travel})
\begin{equation}
    \frac{1}{3}U''-2\alpha U+\frac{3}{2}U^{2}=0
    \label{travele},
\end{equation}
which gives after integration
\begin{equation}
    (U')^{2}=6\alpha U^{2}-3U^{3}
    \label{traveleint}
\end{equation}
which has the KdV soliton
\begin{equation}
   U= 2\alpha
    \mbox{sech}^{2}\left(\sqrt{\frac{3\alpha}{2}}(x-ct)\right)
    \label{sol}
\end{equation}
as a solution.
Note that for a solution $(N,U)$ with a given velocity $c$ of the above 
equations,  $(-N, U)$ provides a solution for the same system with velocity
$-c$. 

\vspace{0.5cm}
Similarly, a solitary wave $u(x,t)=U(x-ct)$ of the Whitham equation \eqref{Whit} satisfies the equation
\begin{equation}
    \frac{\epsilon}{2}U^{2}+(T_{\epsilon}-c)U=0.
    \label{wtravel1}
\end{equation}

\begin{remark}
Following the method used in \cite{BEP} one can prove that solutions of \eqref{wtravel1} that tend to $0$ as $|x|\to \infty$ decay to $0$ exponentially
in the sense that for some $\nu >0,$  $e^{\nu | \cdot |}\phi \in L^1(\R)\cap L^\infty(\R).$
\end{remark}

Writing $c=1+\delta \epsilon$ where $\delta>0$ is a constant independent
of $\epsilon$, we get, neglecting terms  of  order $\epsilon$ or higher in the formal expansion in $\epsilon$,
\begin{equation}
    -\frac{1}{2}U^{2}-\frac{1}{6}U''+\delta U=0,
    \label{wtravel2}
\end{equation}
which gives after integration
\begin{equation}
   ( U')^{2}=-2U^{3}+6\delta U^{2}
    \label{wtravel3}.
\end{equation}
Thus we find again the KdV soliton
\begin{equation}
    U=3\delta
    \mbox{sech}^{2}\left(\sqrt{\frac{3\delta}{2}}(x-ct)\right).
    \label{wtravel4}
\end{equation}

\vspace{0.5 cm}
\begin{remark}
The above formal considerations suggest that one could obtain the existence of solitary wave solutions to both the Whitham equation and the Whitham system by perturbations arguments starting from the KdV soliton. In particular in the case of the Whitham equation this would give an alternative proof to the one in \cite{EGW}.
\end{remark}
\subsection {Whitham with surface tension}

The situation is quite different for the Whitham equation with surface tension \eqref{WhitST}.

A solitary wave solution $u(x,t)=U(x-ct)$ of \eqref{WhitST} should satisfy the equation

\begin{equation}\label{SWST}
\frac{\epsilon}{2}U^{2}+(\tilde T_{\epsilon}-c)U=0
\end{equation}

where $\tilde T_\epsilon$ is defined by \eqref{dispST}.

The existence of localized non trivial solutions of \eqref{SWST} results from Theorem 2.1 in \cite{Ar}. Since \eqref{WhitST} (as the fKdV equation with $\alpha=\frac{1}{2}$) is $L^2$ critical one expects the instability of those solitary waves by blow-up, similarly to the generalized KdV equation with $p=4$, see \cite{MM} or to the modified Benjamin-Ono equation (\cite{MP}).

\vspace{0.3cm}
On the other hand, a solitary wave solution $(\eta(x-ct), u(x-ct))$ of the full-dispersion Boussinesq system should satisfy

\begin{equation}
    \label{SWFD1d}
    \left\lbrace
    \begin{array}{l}
    -c\eta+\mathcal P_\epsilon ^2u_x+\epsilon \eta u=0 \\
     -cu+ \eta\frac{\epsilon}{2} u^2=0,
 \end{array}\right.
    \end{equation}
 where  $\mathcal P_\epsilon=(I+\beta \epsilon |D|^2)\left(\frac{\tanh(\sqrt \epsilon |D| )}{\sqrt \epsilon |D|}\right ).$

   Eliminating $\eta$ from the second equation in \eqref{SWST} yields the single equation
   
   \begin{equation}\label {SWFD2}
   -c^2 u+\mathcal P_\epsilon u +\frac{3}{2}c^2 u^2-\frac{\epsilon^2}{2}u^3=0.
   \end{equation}
   
   One cannot apply directly the result of \cite{Ar} to prove the existence of non trivial solutions to \eqref{SWFD2} since the nonlinear term is not homogeneous.We plan to come back to this issue in a next paper.

\vspace{0.5cm}


\section{Numerical simulations for the Whitham equation}

The simulations in this section will illustrate several aspects of the dynamics of the Whitham equation alluded to in the previous sections and will also give evidence for some facts not yet rigorously proven.

\subsection{Numerical construction of solitary waves for the Whitham 
equation}
In this subsection we construct numerically solitary wave solutions to 
the Whitham equations, i.e., localized solutions of (\ref{wtravel1}). 
To this end we use the same technique as in \cite{KS}: since the 
wanted solutions are expected to be rapidly decreasing, it is 
convenient to formulate the problem as an essentially periodic one 
with a sufficiently large period such that the solutions and the 
numerically interesting derivatives vanish with numerical accuracy at 
the domain boundaries $\pm L\pi$ (we typically choose $L=5$ in the 
following). The Fourier transforms are approximated in standard way 
via discrete Fourier transforms conveniently implemented via the \emph{Fast 
Fourier transform} (FFT). 

With this approach equation (\ref{wtravel1}) is approximated via a 
finite dimensional (we use $N$ Fourier modes) nonlinear equation 
system formally written as 
\begin{equation}
    \mathcal{F}(\hat{U})=0
    \label{NK},
\end{equation}
where $\hat{U}$ denotes the discrete Fourier transform of $U$. This 
nonlinear system 
is solved with a Newton-Krylov iteration. This means the 
action of the inverse of the Jacobian of $\mathcal{F}$ on 
$\mathcal{F}$ in the standard Newton iteration is determined 
iteratively via GMRES \cite{GMRES}. As the initial iterate we 
choose the KdV soliton written as  
    $U=3\delta
    \mbox{sech}^{2}\left(\sqrt{\frac{3\delta}{2}}(x-ct)\right)$.
    Note that the reality of $U$ has to be 
enforced during the iteration. 

As an example we study the case $\epsilon=0.01$ and vary the constant 
$\delta$. This is mainly done in order to have for small $\delta$ an 
initial iterate close to the Whitham soliton. For larger $\delta$, we 
use as an initial iterate the solution for a slightly smaller value of 
$\delta$ constructed before. Since we vary $\delta$, 
the choice of $\epsilon$ is, however, no restriction for the general case since the only 
important quantity is the velocity $c$. It is to be expected that 
there will be solitons for $c\sim 1$ because of the analogy to KdV, 
but that for larger values of $c$ the dispersion of the Whitham 
equation becomes in contrast to the KdV equation too weak to 
compensate the nonlinearity. Thus we expect that in contrast to KdV 
there might be an upper limit of $c$ for solitons to the Whitham 
equation. In fact we obtain the situation shown in 
Fig.~\ref{figwhitsol}. For $c=1.01,1.02,1.05,1.1,1.2$ we use 
$N=2^{14}$ Fourier modes, for $c=1.22$ $N=2^{16}$.  
\begin{figure}[htb!]
  \includegraphics[width=0.49\textwidth]{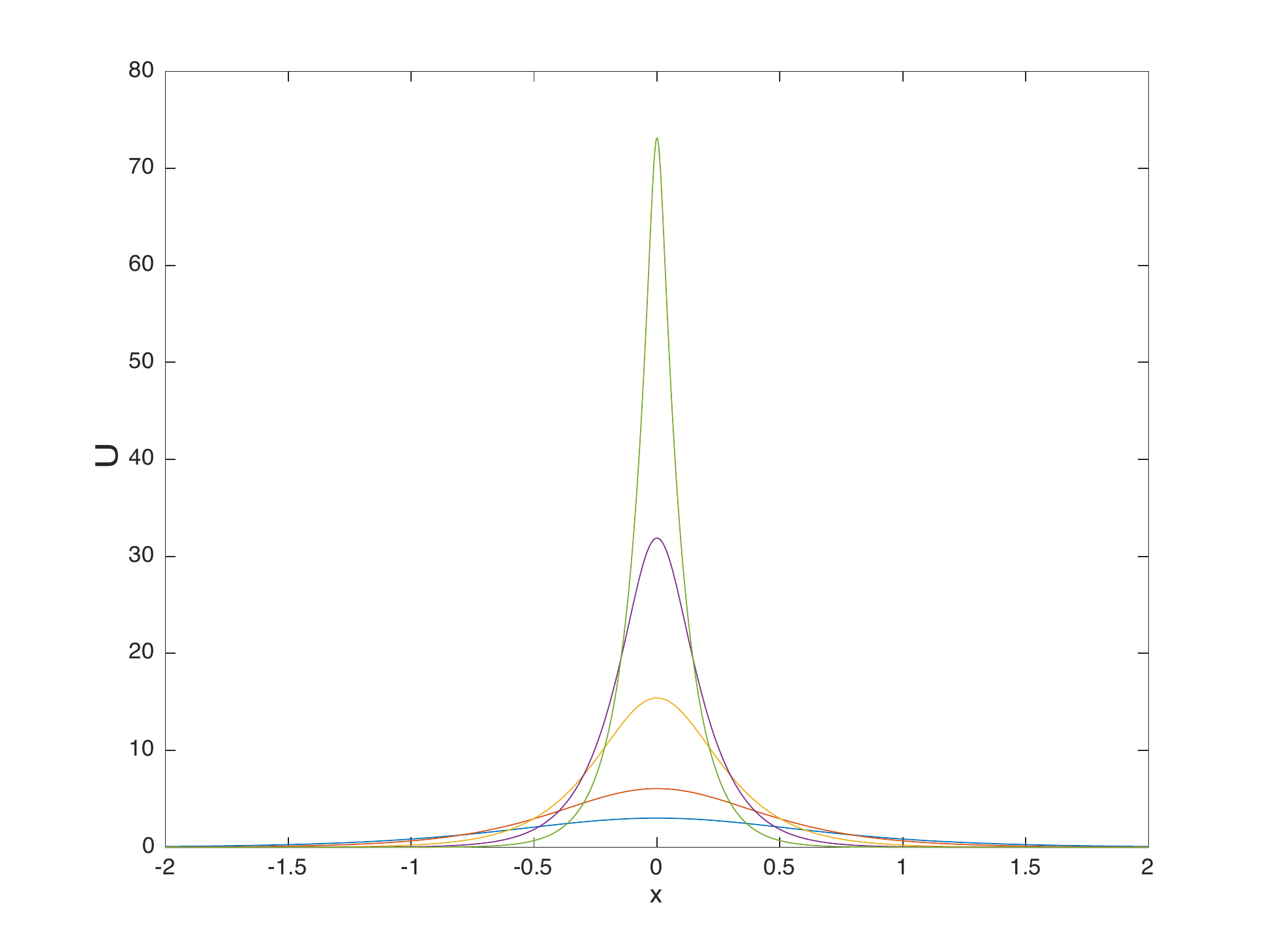}
  \includegraphics[width=0.49\textwidth]{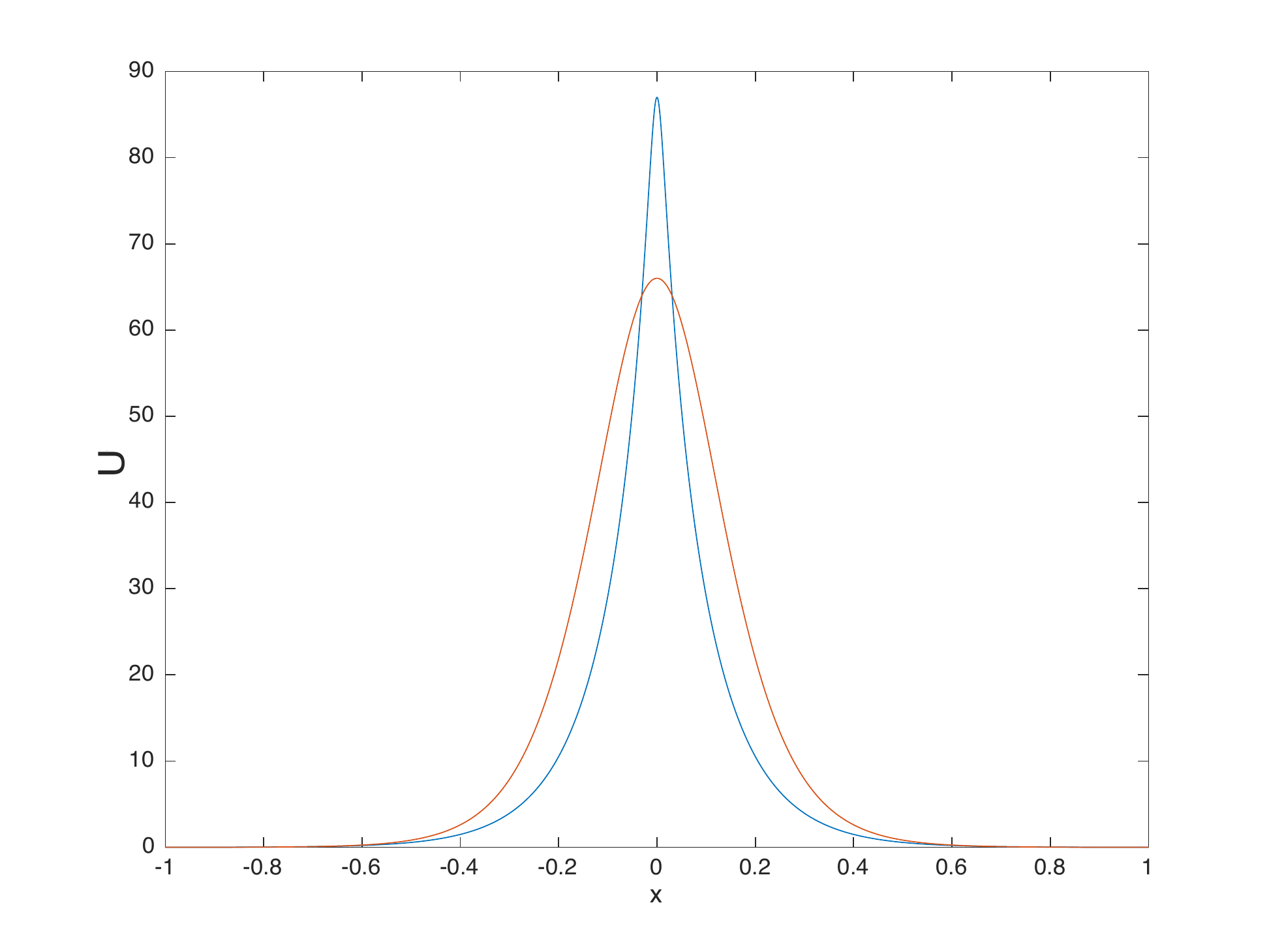}
 \caption{Solitary waves for the Whitham equation (\ref{wtravel1}): 
 on the left the Whitham solitons for $c=1.01,1.02,1.05,1.1,1.2$ (in 
 the order of growing maxima), 
 on the right the Whitham soliton for $c=1.22$ in blue and the 
 corresponding KdV soliton in red.}
 \label{figwhitsol}
\end{figure}

It can be seen 
that maxima of the solution grow as expected, but that the solitons 
become more peaked  and  more compressed compared 
to the corresponding KdV case. In fact we did not succeed to 
construct solitons for $c$ much greater than $1.22$. The failure of an iteration to 
converge obviously does not imply that there will be no longer a 
soliton, it just means that the numerical approach no longer can be 
used. However the fact that the iteration did not converge even for 
larger resolutions, for initial iterates being the numerical solution for 
a slightly smaller $\delta$ and for an iteration with relaxation is a 
strong indication that there might not be a Whitham solitary waves for much larger 
velocities. 

There are additional differences between Whitham and KdV 
solitons. The mass (the $L^{2}$ norm) of the KdV soliton 
(\ref{wtravel4}) is given by $M=4\sqrt{6}\delta^{3/2}$, the mass of 
the Whitham solitons can be seen on the left of 
Fig.~\ref{figwhitsolmass}. As expected the mass is identical for 
small $\delta$ to the one of the KdV soliton. But for larger speeds 
$c$, the mass grows more slowly. The energy of the KdV soliton 
(\ref{wtravel4}) is proportional to  $\delta^{5/2}$, the energy of the 
Whitham soliton can be seen on the right of 
Fig.~\ref{figwhitsolmass}. The energy curve appears to flatten near 
the reachable maximal velocities.
\begin{figure}[htb!]
  \includegraphics[width=0.49\textwidth]{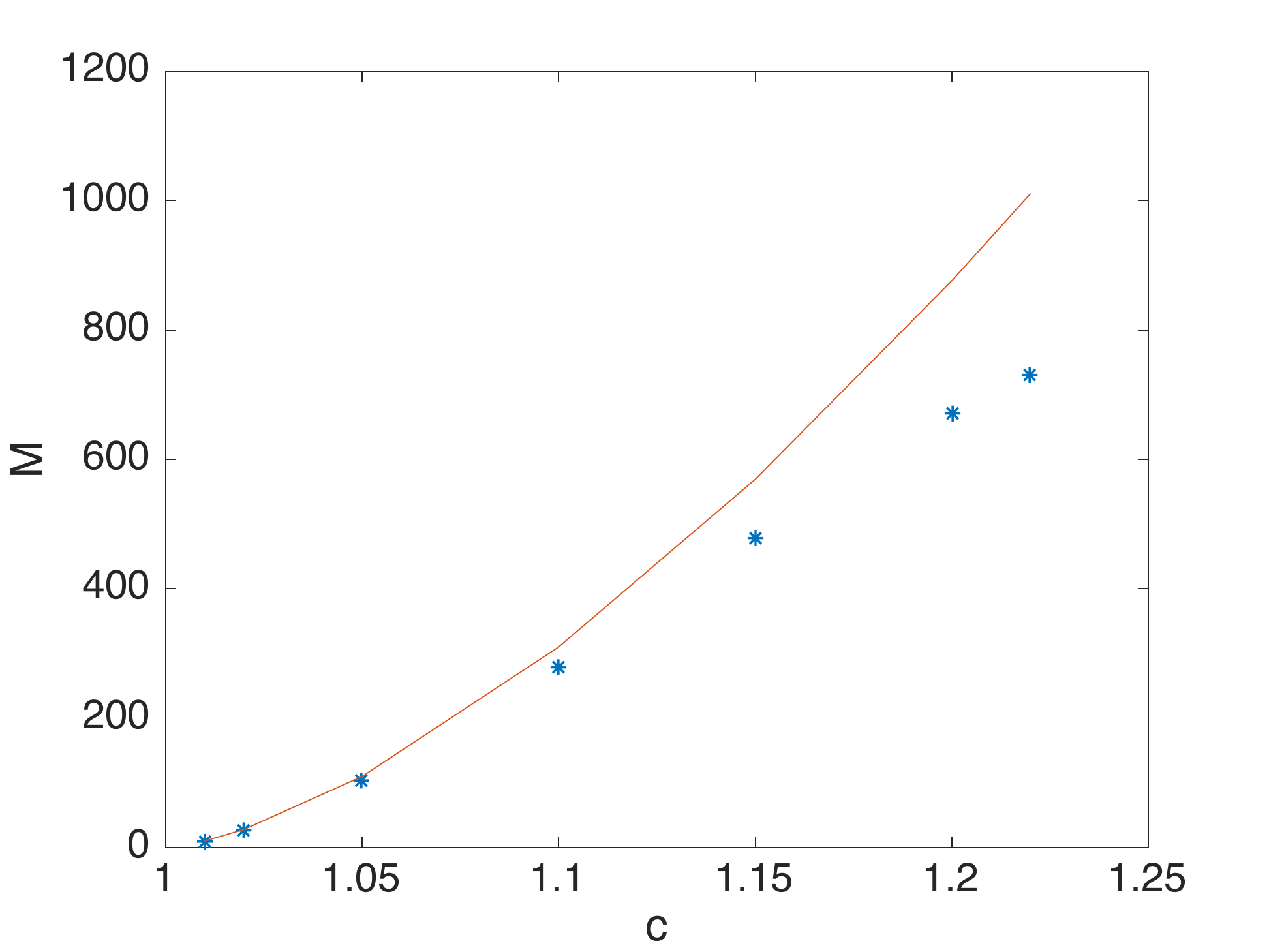}
  \includegraphics[width=0.49\textwidth]{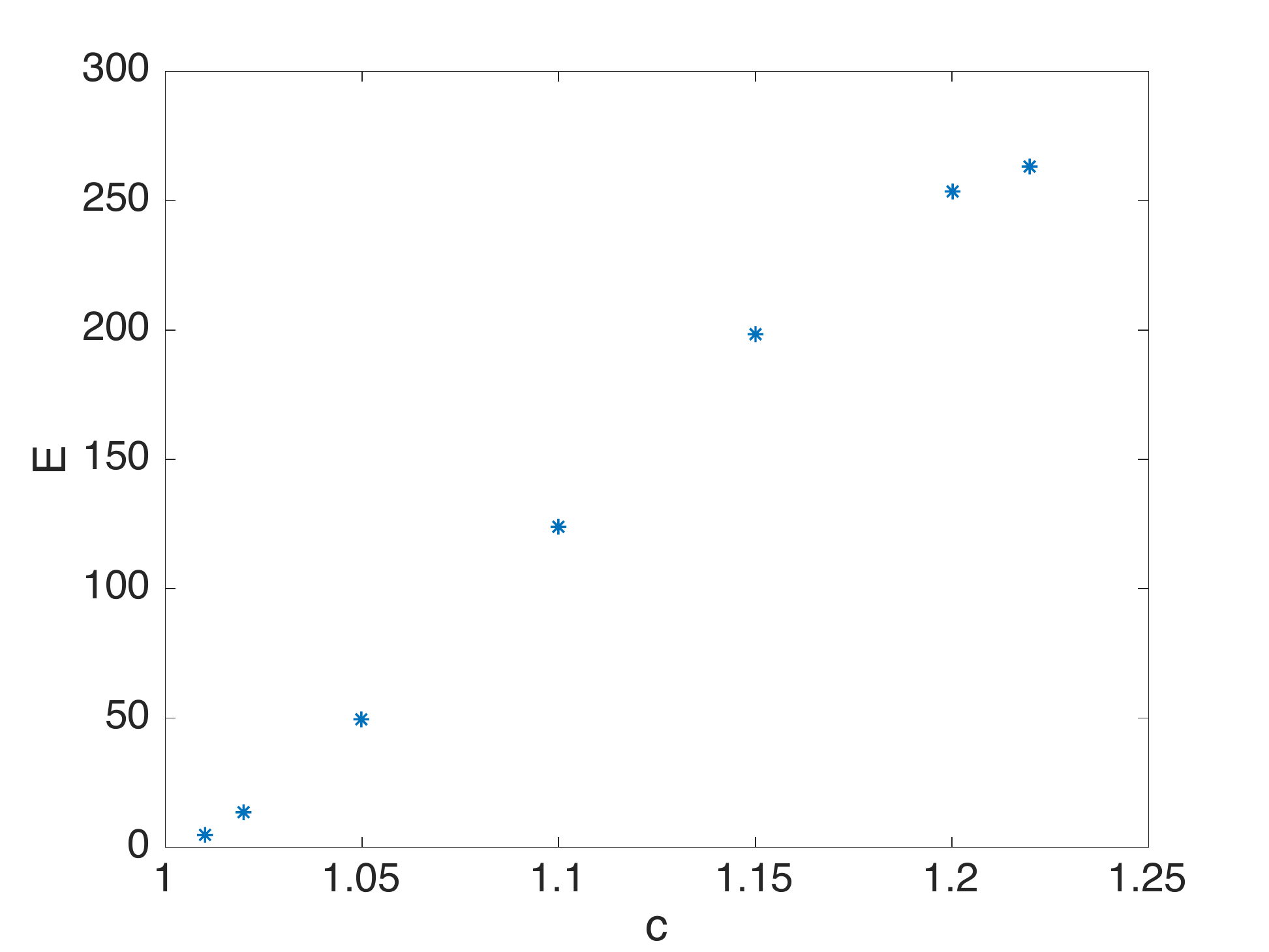}
 \caption{Mass of the solitary waves for the Whitham equation 
 (\ref{wtravel1}) in dependence of $c$ (the continous line gives the mass 
 of the KdV solitons (\ref{wtravel4}))  on the left, and the 
 corresponding energy on the right.}
 \label{figwhitsolmass}
\end{figure}

\subsection{Numerical study of  perturbed Whitham solitons}
In this subsection, we study the time evolution of perturbations of the 
Whitham solitons constructed above. To this end we 
write the Whitham equation for the Fourier transform in a commoving 
frame, i.e.,
\begin{equation}
    \hat{u}_{t}+\frac{\epsilon}{2}ik 
    \widehat{u^{2}}+(\hat{T}_{\epsilon}-c)ik\hat{u}=0
    \label{whithamfourier},
\end{equation}
and approximate the Fourier transform via a discrete Fourier transform as 
before. For the time integration we use as in \cite{KS} an implicit 
Runge-Kutta method of fourth order with a fixed point iteration, see 
\cite{KS} for details. The accuracy of the solution is controlled 
via the Fourier coefficients which should decrease exponentially to 
machine precision (we work here in double precision which allows a 
maximal accuracy of $10^{-16}$) for smooth functions and via the 
conservation of the energy
\begin{equation}
    E = 
    \int_{\mathbb{R}}^{}\left[\frac{1}{2}(\sqrt{T_{\epsilon}-c}u)^{2}-
    \frac{\epsilon}{6}u^{3}\right]dx
    \label{whithamE}.
\end{equation}
Due to unavoidable numerical errors, this energy does numerically 
evolve with time, but thus provides an estimate of the numerical 
error. As shown in \cite{etna}, conserved quantities typically 
overestimate the $L^{\infty}$ error (the maximum of the difference 
between numerical and exact solution) by 1-2 orders of magnitude.

As a first test of the code, we take the numerically constructed 
soliton with $\epsilon=0.01$ and $c=1.2$ of the previous subsection as 
initial data. In the used commoving frame, the solitons should 
correspond to a stationary solution. We use $N_{t}=10^{4}$ time steps 
for $t\in[0,10]$ and find that the difference between the evolved 
solution and the initial data is of the order of $10^{-13}$, the 
order of the error with which the equation (\ref{wtravel1}) had been 
solved in the previous section. The numerically computed relative 
energy $1-E(t)/E(0)$ is of the order of $10^{-15.5}$. 

We first consider as initial data the soliton with $\epsilon=0.01$ 
%
$c=1.2$ plus a Gaussian perturbation 
which is  a perturbation of the order of $1\%$ of the maximum of 
the soliton. We use $N_{t}=2*10^{4}$ time steps for $t\in[0,20]$. The 
Fourier coefficients in this case decrease to the order of $10^{-14}$ 
during the whole computation, and the relative energy is conserved to 
the order better than $10^{-14}$. In this case there appears a 
slightly faster soliton travelling to the right in the commoving 
frame, as can be seen in Fig.~\ref{figwhitsolc12gauss}. But in 
addition to radiation travelling to the left, a slower 
soliton travelling to the left in the frame commoving with the 
unperturbed soliton seems to emanate from the initial data. This is 
even more visible in the plot on the right of 
Fig.~\ref{figwhitsolc12gauss} where a close up of the smaller soliton 
is shown for $t=20$. The soliton appears to be stable, the 
perturbation leads to a slightly larger soliton and possibly a 
smaller soliton plus radiation. 
\begin{figure}[htb!]
  \includegraphics[width=0.49\textwidth]{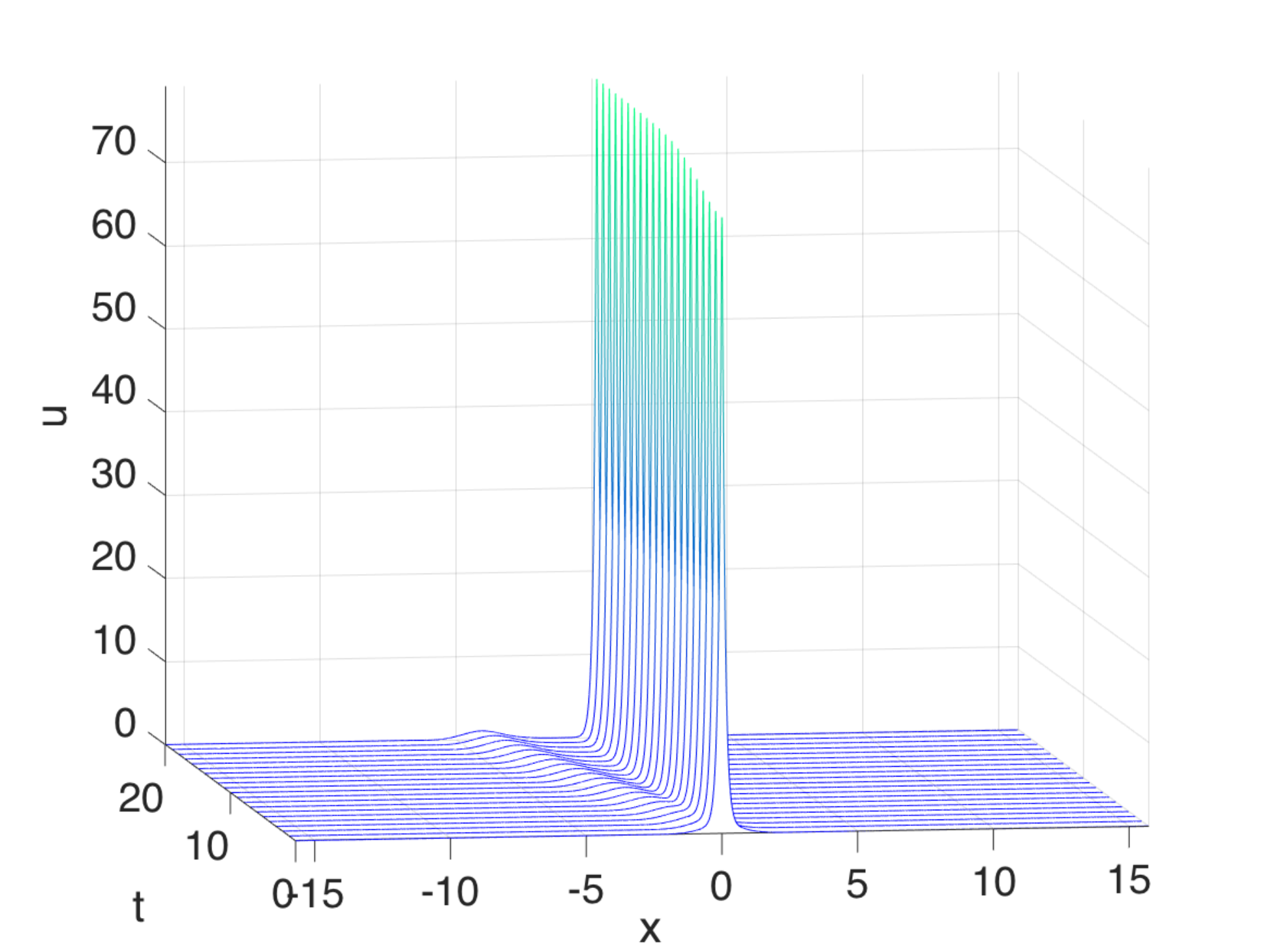}
  \includegraphics[width=0.49\textwidth]{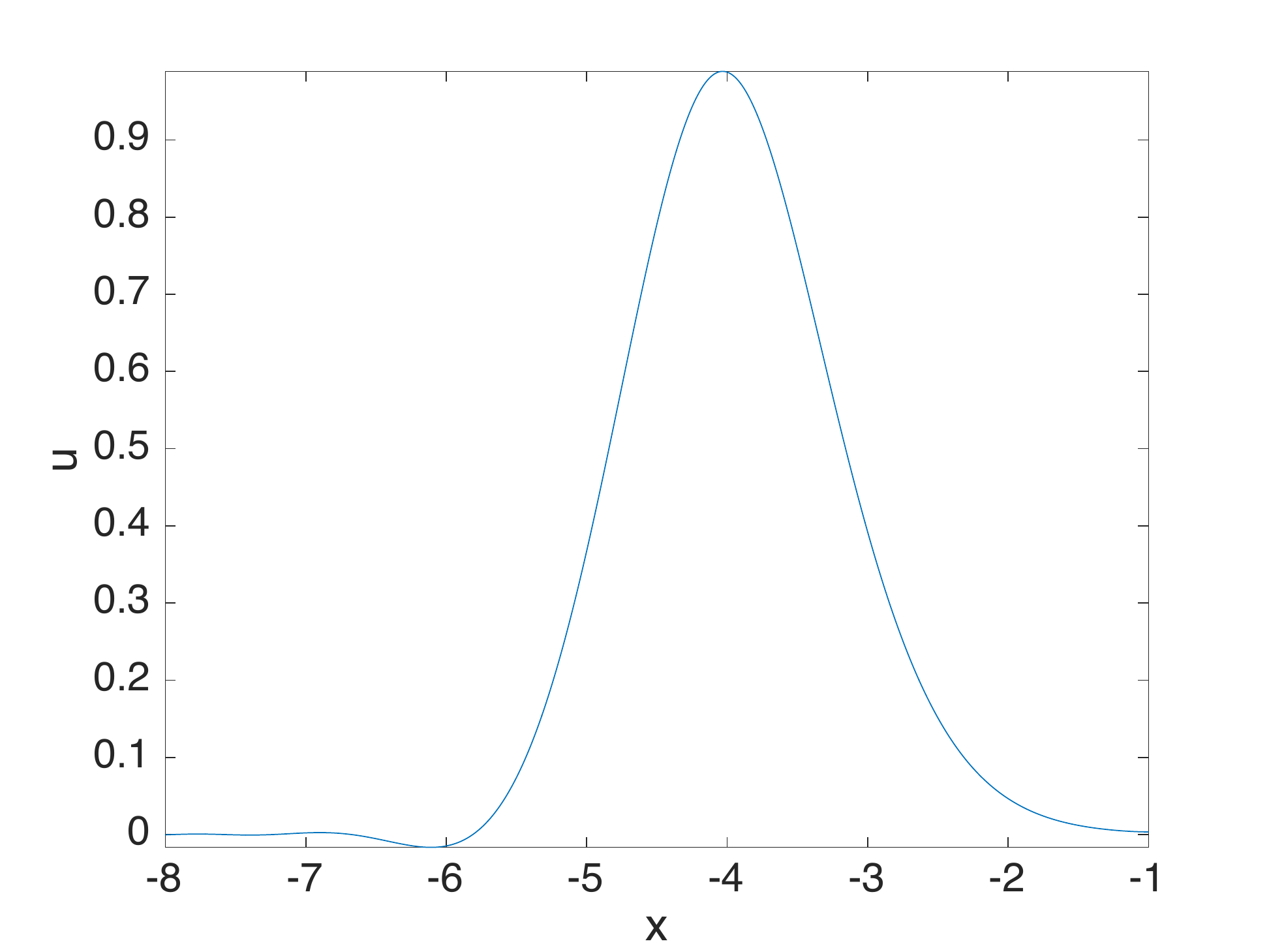}
 \caption{Solution to the Whitham equation with $\epsilon=0.01$ for 
 initial data the soliton with $c=1.2$ plus a Gaussian 
 perturbation; on the left the solution in dependence of time, on the 
 right a close up of the solution at the final time $t=20$.}
 \label{figwhitsolc12gauss}
\end{figure}

The situation changes  if the same soliton as in 
Fig.~\ref{figwhitsolc12gauss} is considered, but this time with a 
perturbation $3\exp(-x^{2})$ corresponding to roughly $4\%$ of the 
maximum of the soliton. We use $N_{t}=10^{4}$ time steps for 
$t\in[0,8]$. In this case the solution appears to develop a cusp. We 
show in Fig.~\ref{figwhitsolc123gauss} the solution and the modulus 
of its Fourier coefficients at $t=7.52$. It can be seen that the 
solution clearly runs out of resolution in Fourier space, already at 
$t\sim 7$, the Fourier coefficients decrease only to $10^{-2}$.
\begin{figure}[htb!]
  \includegraphics[width=0.49\textwidth]{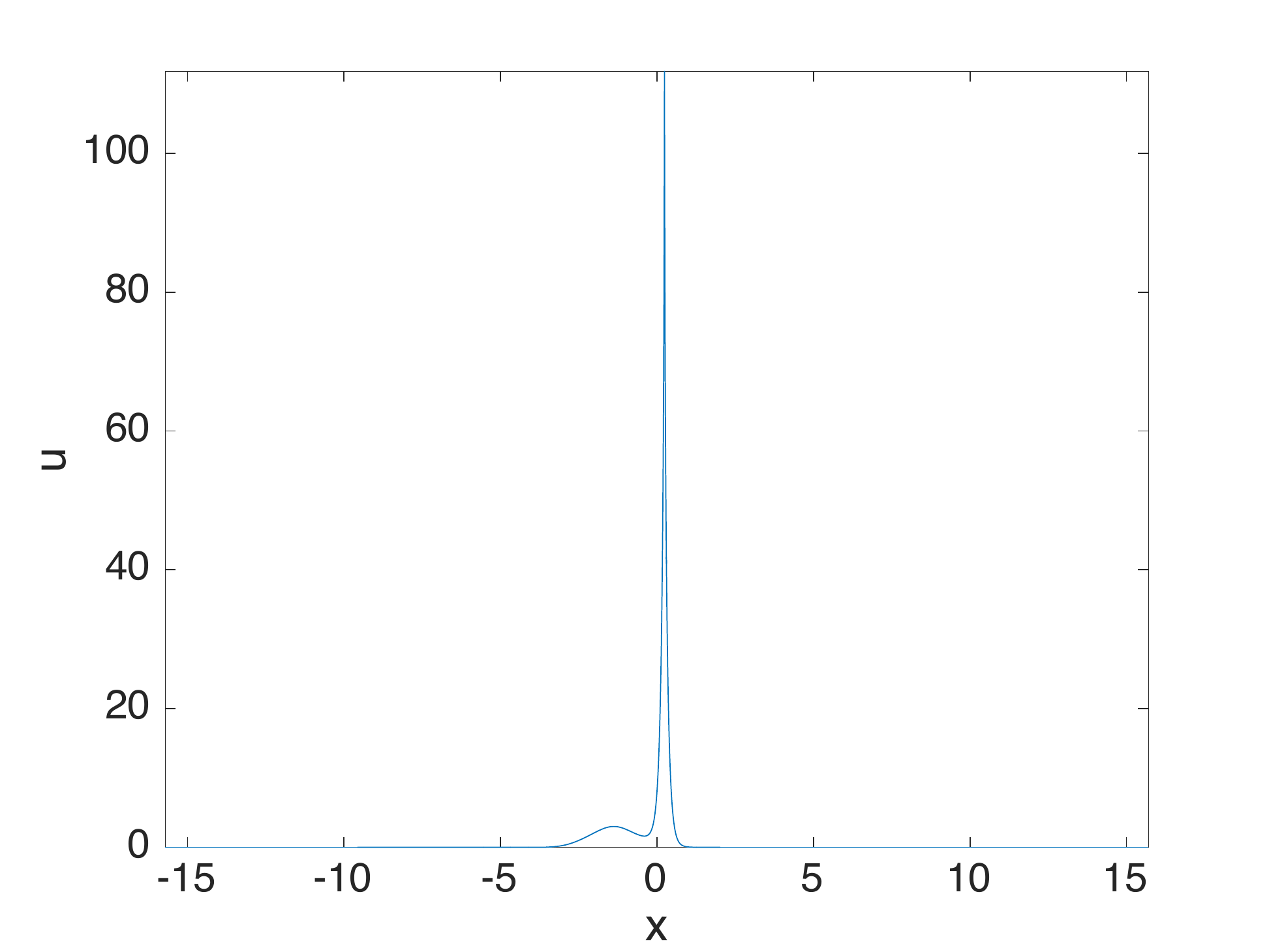}
  \includegraphics[width=0.49\textwidth]{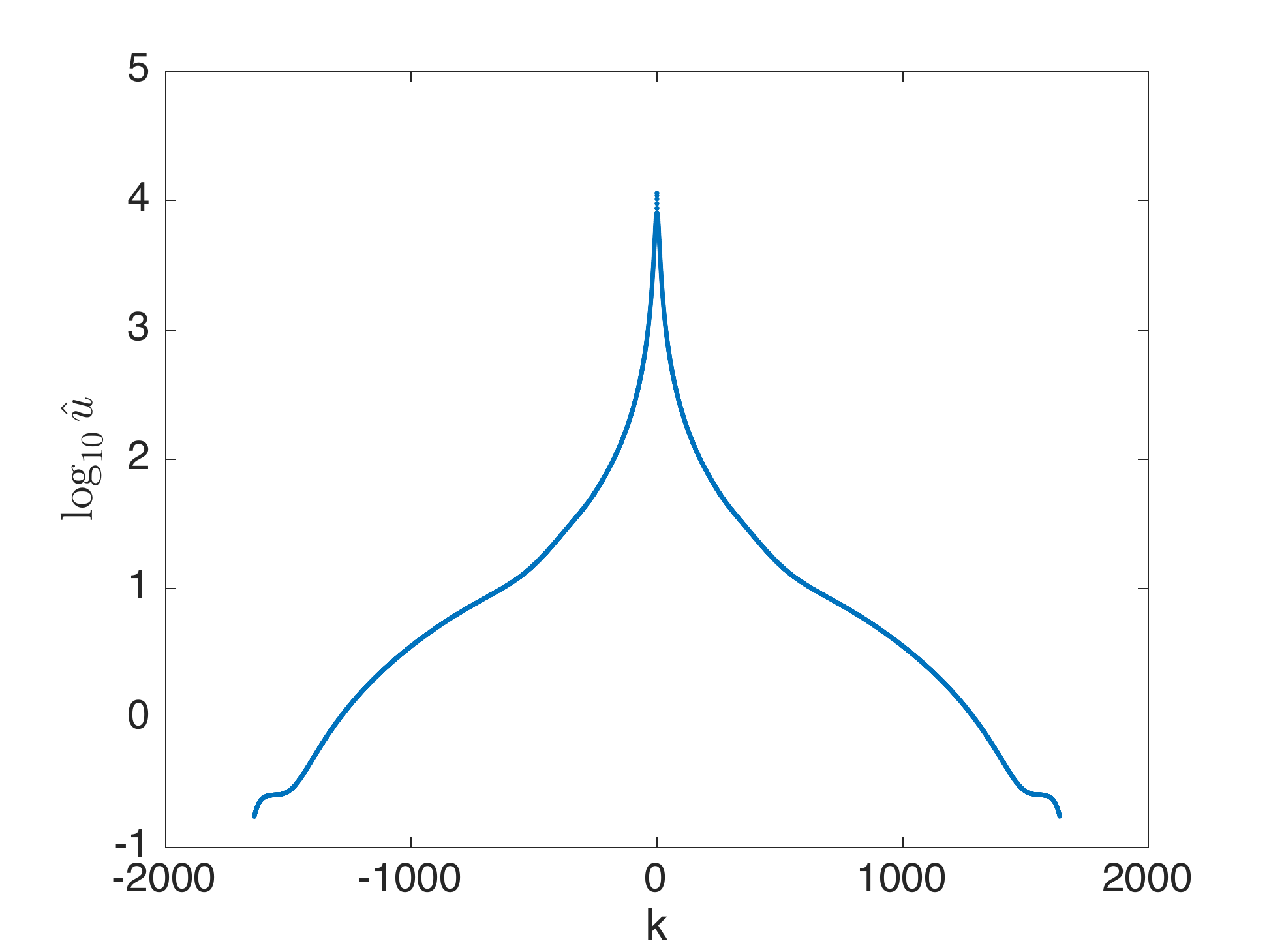}
 \caption{Solution to the Whitham equation with $\epsilon=0.01$ for 
 initial data the soliton with $c=1.2$ plus a  
 perturbation $3\exp(-x^{2})$; on the left the solution for $t=7.52$, 
 on the right the modulus of the Fourier coefficients.}
 \label{figwhitsolc123gauss}
\end{figure}

In Fig.~\ref{figwhitsolc123gaussnorm} it can be seen that the 
$L^{\infty}$ norm of the solution grows clearly beyond the initial 
data, but it is not clear whether it diverges. The same behavior can 
be seen for the $L^{2}$ norm of $u_{x}$ on the right of the same 
figure. A fit 
of the Fourier coefficients as in \cite{SSF,KS} reveals that a 
singularity appears to approach the real axis in the complex 
$x$-plane. The idea of this approach is that an essential singularity 
of the form $(x-x_{s})^{\mu}$, $\mu\notin \mathbb{Z}$, in the complex 
plane for a function $u$ leads to an asymptotic behavior of the 
Fourier transform
\begin{equation}
    |\hat{u}(k)| \propto \frac{1}{k^{\mu+1}} e^{-\delta k},\quad  k \gg 1,
    \label{fit}
\end{equation}
where $\delta=\Im x_{s}$.
For $t=7.547$, this singularity seems to hit the real 
axis ($\delta\to0$) indicating the formation of a cusp (the fitted 
coefficient $\mu$ of 
the essential singularity is positive as in \cite{KS}). Thus there 
seems to be a hyperbolic blow-up in this case. The soliton for 
$c=1.2$ appears therefore to be  unstable against blow-up in the form of a cusp 
for sufficient size of the perturbation. 
\begin{figure}[htb!]
  \includegraphics[width=0.49\textwidth]{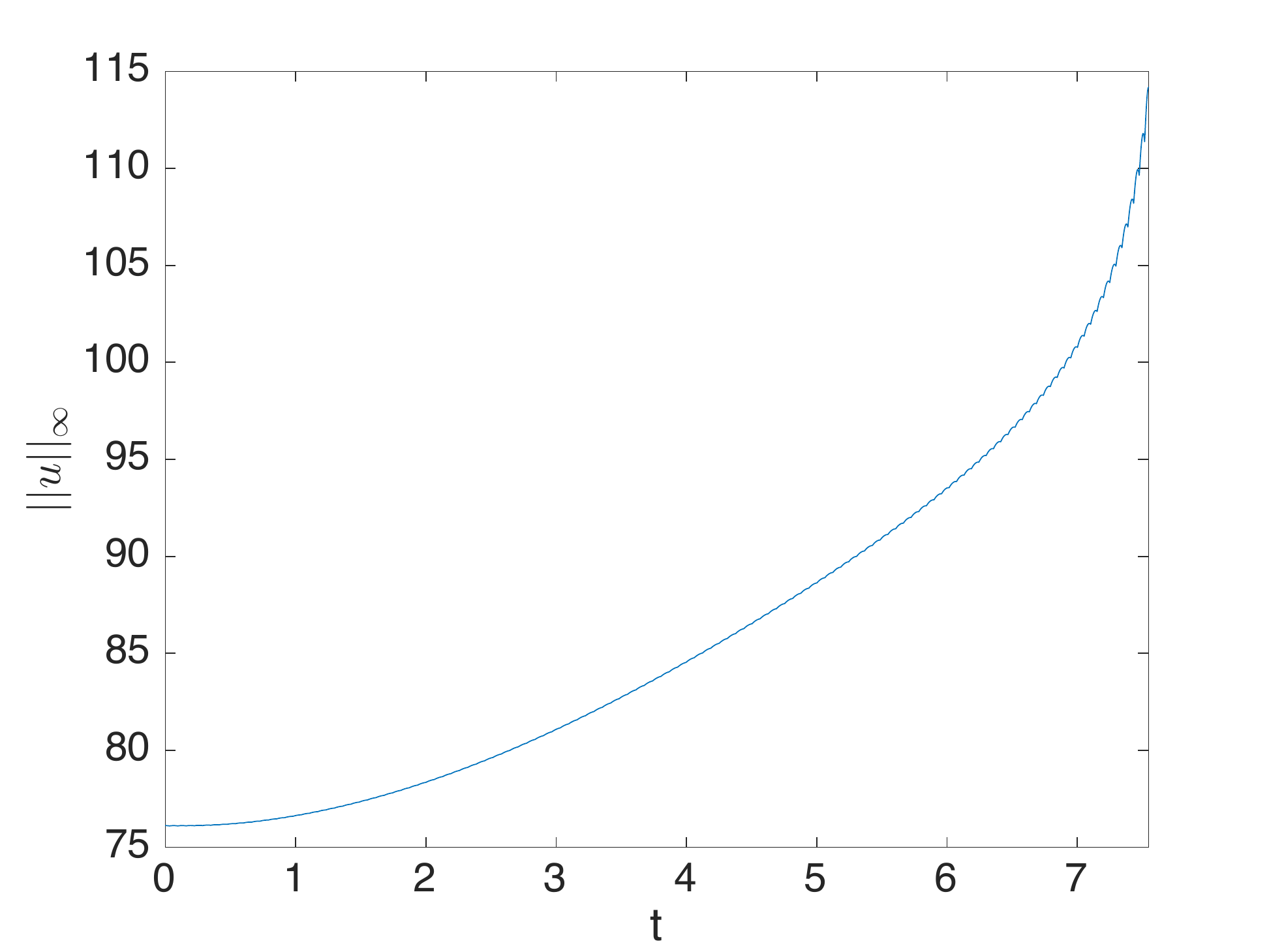}
  \includegraphics[width=0.49\textwidth]{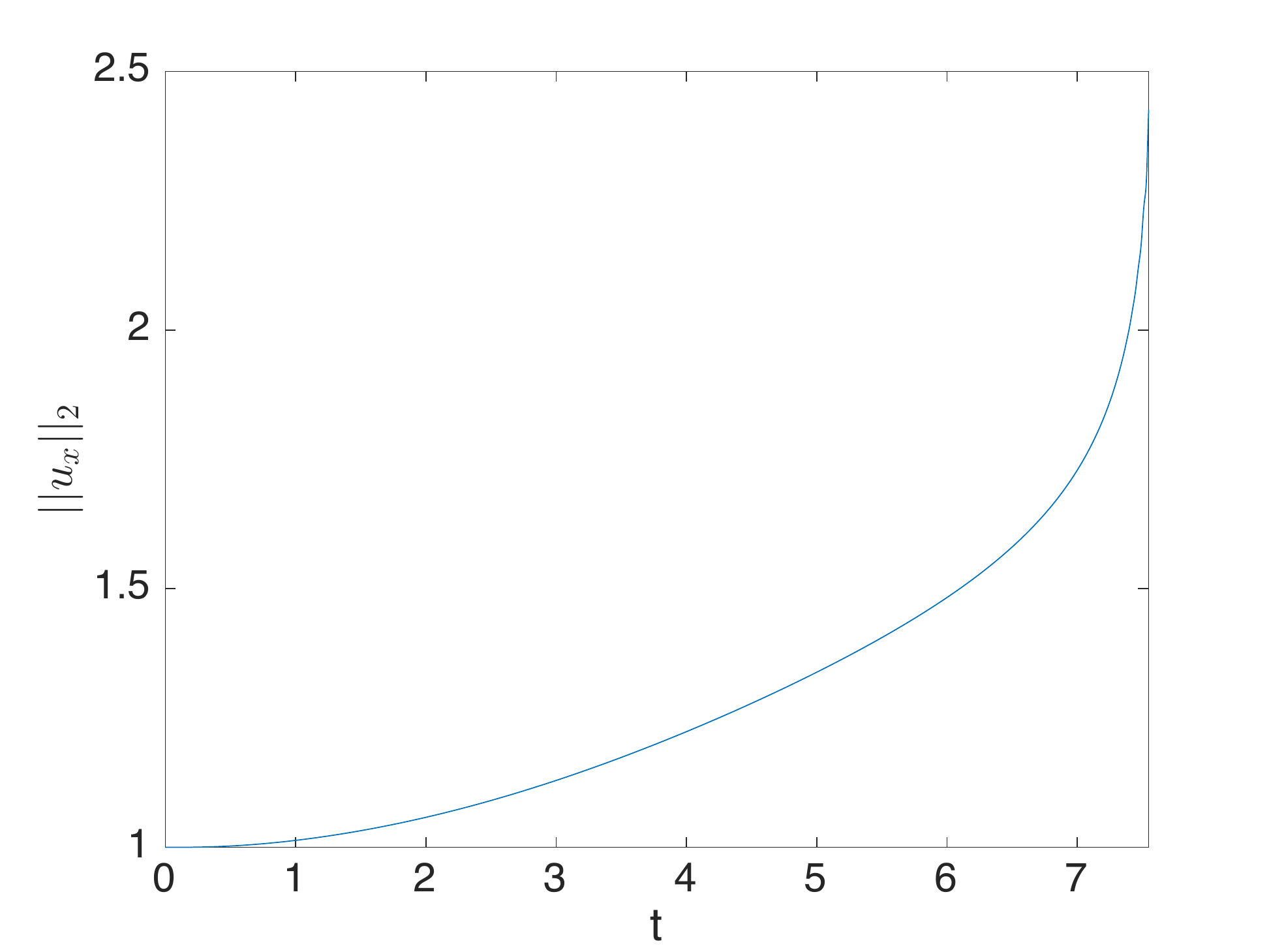}
 \caption{Norms of the solution to the Whitham equation with $\epsilon=0.01$ for 
 initial data the soliton with $c=1.2$ plus a  
 perturbation $3\exp(-x^{2})$; on the left the $L^{\infty}$ norm, 
 on the right the $L^{2}$ norm of $u_{x}$ normalized to 1 for $t=0$.}
 \label{figwhitsolc123gaussnorm}
\end{figure}

Since the solitons of the Whitham equation are rapidly decreasing, it 
is possible to study soliton interactions. In 
Fig.~\ref{figsautwhitham2sol}, we show the solitons with $c=1.05$ and 
$c=1.1$, the latter being centered at $x=-4$ by simply multiplying 
its Fourier coefficients by $e^{4ik}$. The sum of the solutions gives 
two soliton initial data since they both vanish with numerical 
precision where the other takes values above the numerical error. On 
the left of Fig.~\ref{figsautwhitham2sol} we show the solution to the 
Whitham equation for $t\in[0,100]$. The modulus of the Fourier coefficients of the 
solution decreases to $10^{-12}$ during the whole computation. It can 
be seen that the soliton interaction resembles the KdV two-soliton: 
the solitons have almost the same shape after the collision, there is 
just a phase shift (note that the computation is performed in a frame 
commoving with the smaller soliton at $c=1.05$). The close-up of the 
solution at the final time on the right of 
Fig.~\ref{figsautwhitham2sol} reveals, however, that this is not an 
exact two-soliton since there is dispersive radiation propagating to 
the left. This indicates that the Whitham equation is as expected not 
integrable. 
\begin{figure}[htb!]
  \includegraphics[width=0.49\textwidth]{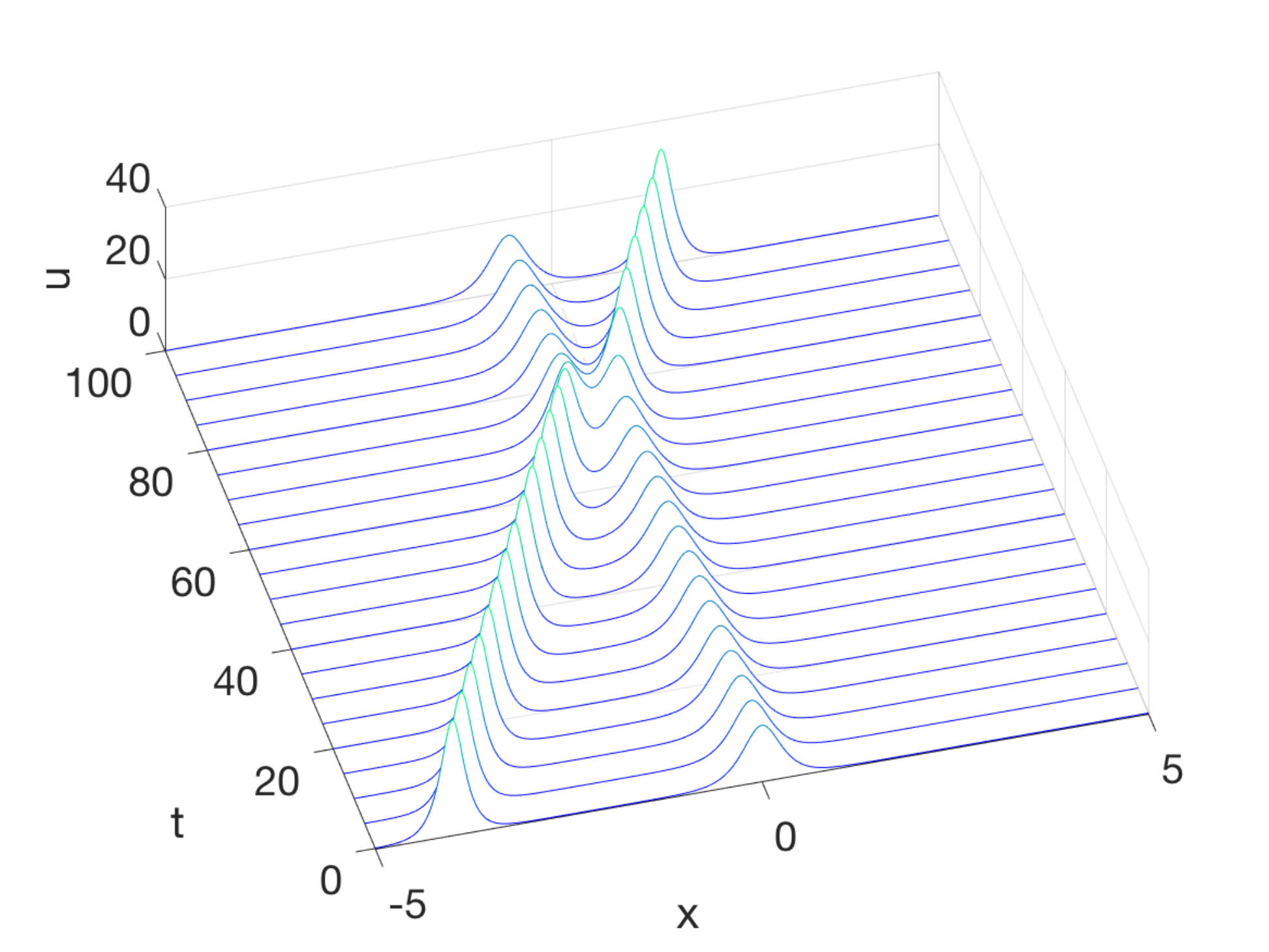}
  \includegraphics[width=0.49\textwidth]{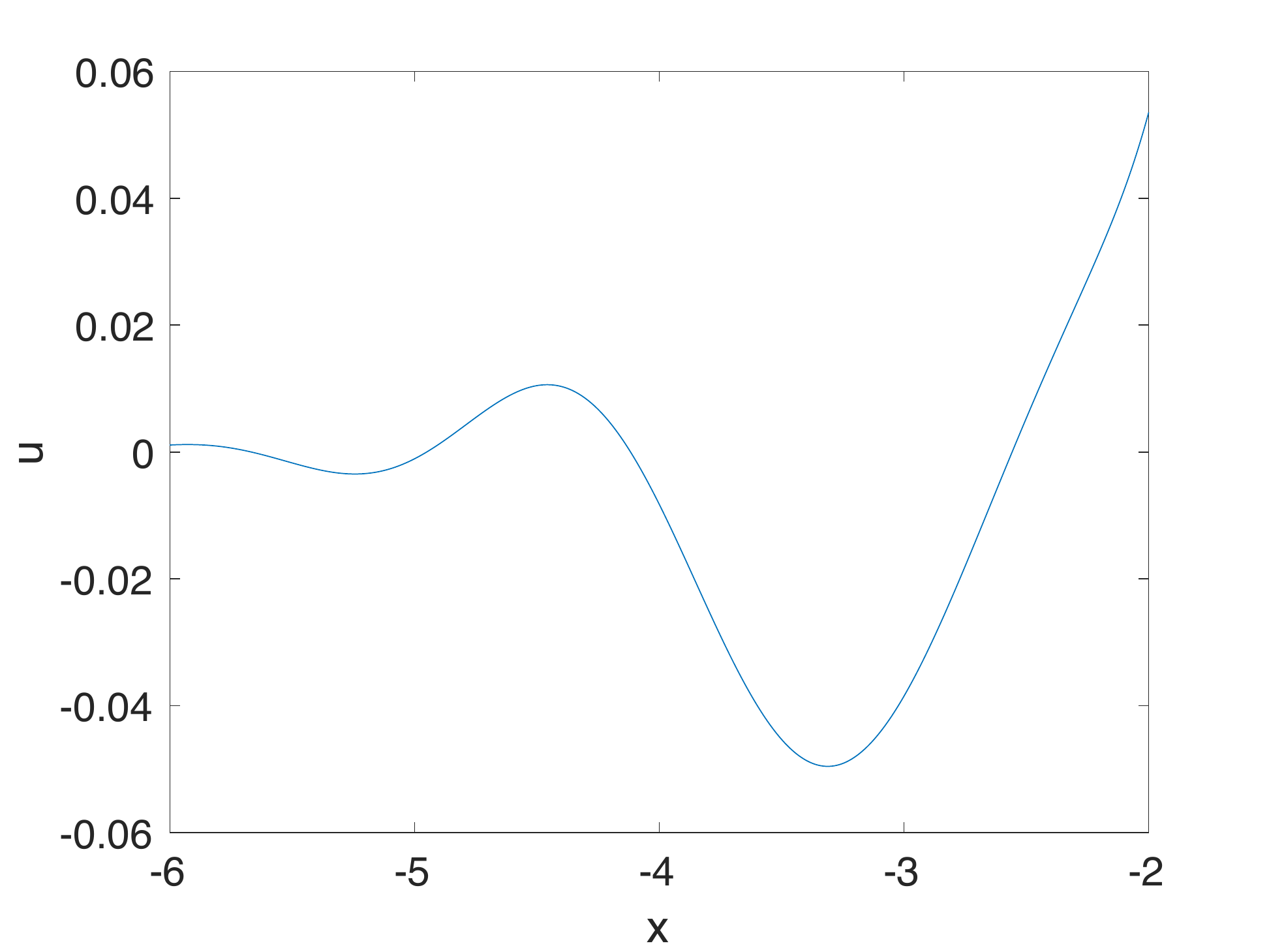}
 \caption{Solution to the Whitham equation for $\epsilon=0.1$ for 
 initial data being the sum of the solitons with $c=1.05$ centered at 
 $x=0$ and for $c=1.1$ centered at $x=-4$ in dependence of time on 
 the left and a close-up of the solution at $t=100$ on the right.}
 \label{figsautwhitham2sol}
\end{figure}

\subsection{Numerical study of Gaussian initial data for different values 
of $\epsilon$}
In this section we study numerically the time evolution of Gaussian 
initial data $u(x,0)=10\exp(-x^{2})$ 
for different values of the parameter $\epsilon$. Since we are 
interested in studying the solutions on time scales of order 
$1/\epsilon$, we 
introduce the rescaled time $\tau=t\epsilon$ in which the Whitham 
equation (\ref{Whit}) reads
\begin{equation}
    u_{\tau}+ uu_{x}+\frac{1}{\epsilon}(T_{\epsilon}-1)u_{x}=0
    \label{whithamtau},
\end{equation}
where we have used  a frame commoving with velocity 1. This 
equation is solved with the same numerical approach as in the 
previous subsection. 

For $\epsilon=1$ we use $N=2^{14}$ Fourier modes and $N_{t}=10^{4}$ 
time steps for $\tau\in[0,0.2]$. The solution appears to develop a cusp 
for $\tau>0.1175$ as can be seen in Fig.~\ref{figwhitgausse1} on the left. 
The Fourier coefficients on the right of the same figure indicate 
a loss of resolution. A fitting of the Fourier coefficients to 
(\ref{fit}) indicates 
that for $\tau_{c}=0.1175$, the parameter $\delta$ indicates the distance 
between a singularity in the complex $x$-plane vanishes, and the 
corresponding parameter $\mu\sim0.3635$ indicating the formation of a 
cusp proportional to $(x-x_{s})^{1/3}$. 
\begin{figure}[htb!]
  \includegraphics[width=0.49\textwidth]{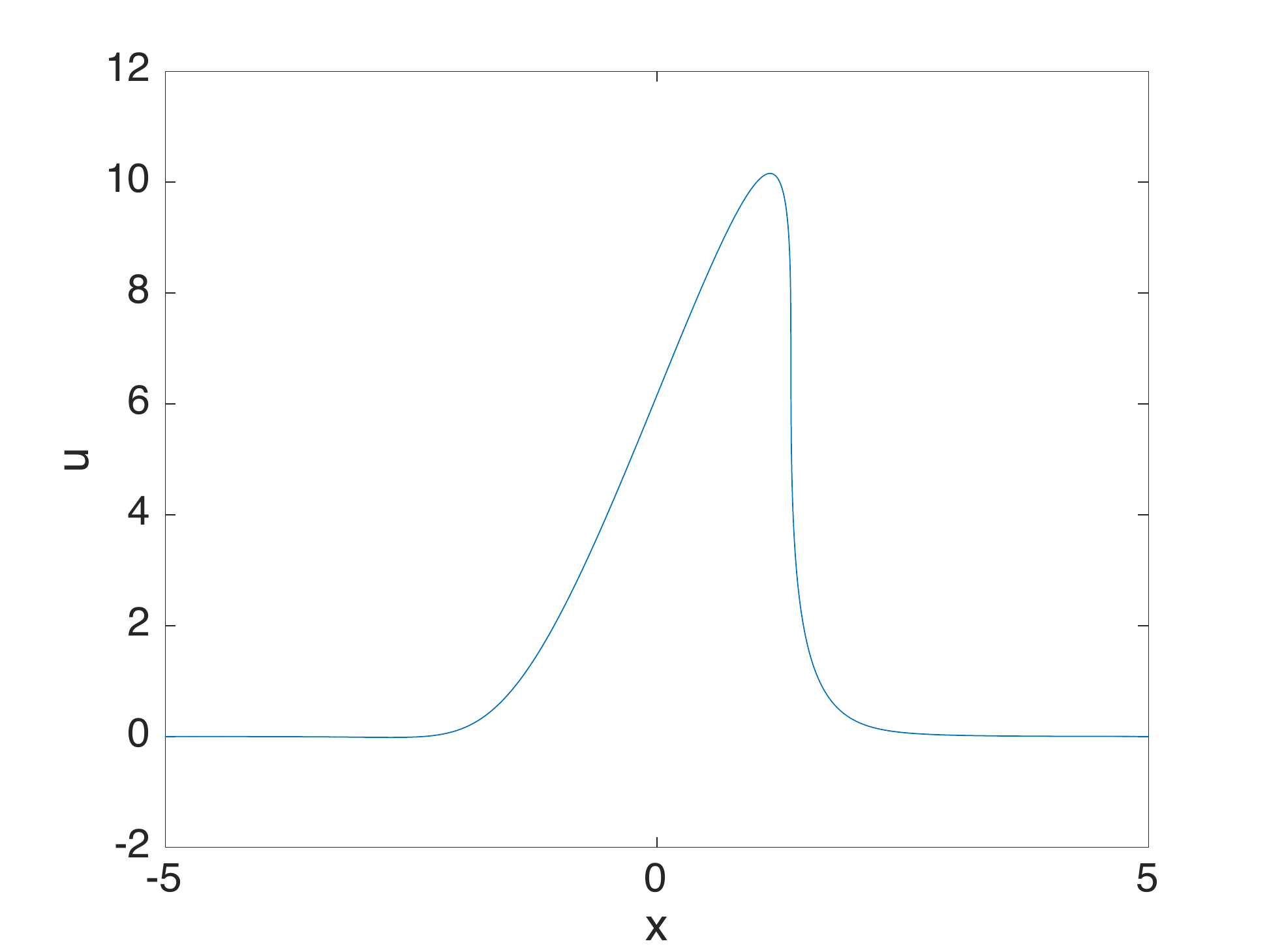}
  \includegraphics[width=0.49\textwidth]{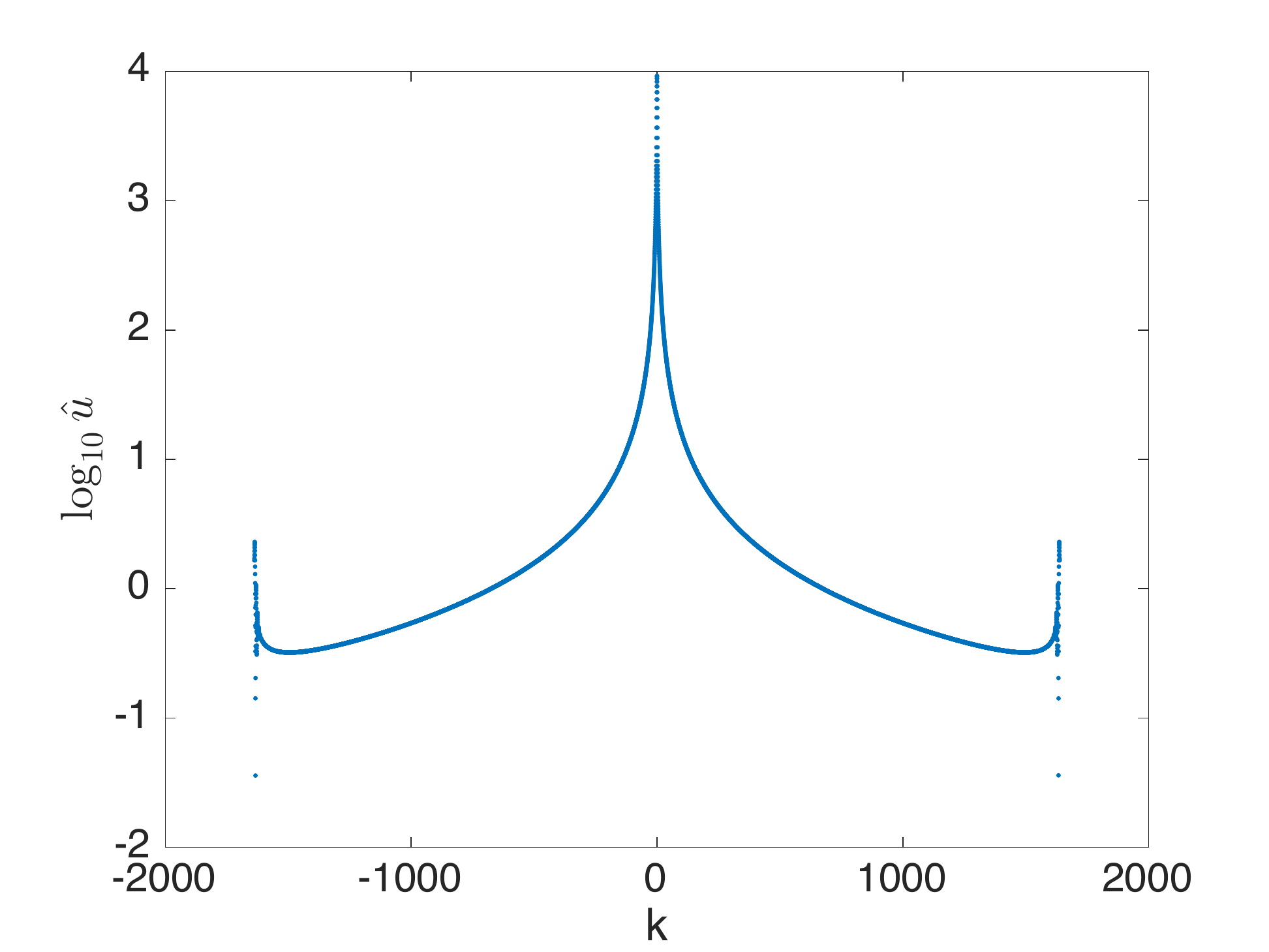}
 \caption{Solution to the Whitham equation with $\epsilon=1$ for 
 initial data  $u(x,0)=10\exp(-x^{2})$; on the left the solution for 
 $\tau=0.1175$, 
 on the right the modulus of the corresponding Fourier coefficients.}
 \label{figwhitgausse1}
\end{figure}

Note that the critical time $\tau_{c}$ above does not indicate the 
exact blow-up time, this just indicates that the singularity in the 
complex plane leading to the asymptotic behavior (\ref{fit}) of the 
Fourier coefficients is too close to the axis to be numerically 
distinguished from 0. If one considers the same initial data as in 
Fig. \ref{figwhitgausse1} for smaller values of $\epsilon$, one 
obtains the critical times and exponents shown in the table in Fig. 
\ref{figwhitgausstable}. It can be seen that the critical times 
    $\tau_{c}$ are always greater than the $\tau_{c}$ for 
    $\epsilon=1$. In fact the critical times appear to grow with 
    decreasing $\epsilon$. This implies, however, that the solutions 
    do not have a blow-up on time scales of order $1/\epsilon$ 
    (recall that $\tau=t\epsilon$). 
    \begin{table}[tbp]
            \centering
\begin{tabular}{|c|c|c|}
    \hline
    $\epsilon$ & $\tau_{c}$ & $\mu$  \\
    \hline
    1 & 0.1175 & 0.3635  \\
    \hline
    0.8 & 0.1178 & 0.3649  \\
    \hline
    0.6 & 0.1184 & 0.3641  \\
    \hline
    0.4 &0.1197 &  0.3619 \\
    \hline
    0.2 &  0.1254 & 0.3499  \\
    \hline
    0.1 &  0.1482 & 0.3855  \\
    \hline
    0.08 &  0.1674 & 0.4133  \\
    \hline
    0.06 &  0.2163 & 0.4274  \\
    \hline
\end{tabular}
 \caption{Critical times $\tau_{c}=\epsilon t_{c}$ 
 and exponents according to (\ref{fit}) for 
 the solutions to the Whitham equation for 
 initial data  $u(x,0)=10\exp(-x^{2})$ for different values of $\epsilon$.}
 \label{figwhitgausstable}
     \end{table}

For values of $\epsilon$ smaller than $0.04$, there does not appear 
to be a blow-up at all. In Fig. \ref{figwhitgausse01} we show the 
solution to the same initial data as in Fig. \ref{figwhitgausse1} for 
$\epsilon=0.01$. We use $N=2^{12}$ Fourier modes for $x\in 
L[-\pi,\pi]$ with $L=10$ and $N_{t}=10^{4}$ time steps for 
$\tau\in[0,2]$. It can be seen that two stable solitons appear which 
gives support to the soliton resolution conjecture in the context of 
the Whitham equation: stable solitons seem to appear in the long time 
behavior of the solution. In addition there is the usual dispersive 
radiation also known from KdV solutions.  The $L^{\infty}$ norm of 
the solution on the right of 
Fig.~\ref{figwhitgausse01} also seems to correspond to a soliton with 
speed greater than 1 (we are in a commoving frame with $c=1$). Neither the 
Fourier coefficients nor the $L^{2}$ norm of  $u_{x}$ indicate the 
formation of a singularity. The latter again indicates 
the appearance of a soliton. 
\begin{figure}[htb!]
  \includegraphics[width=0.49\textwidth]{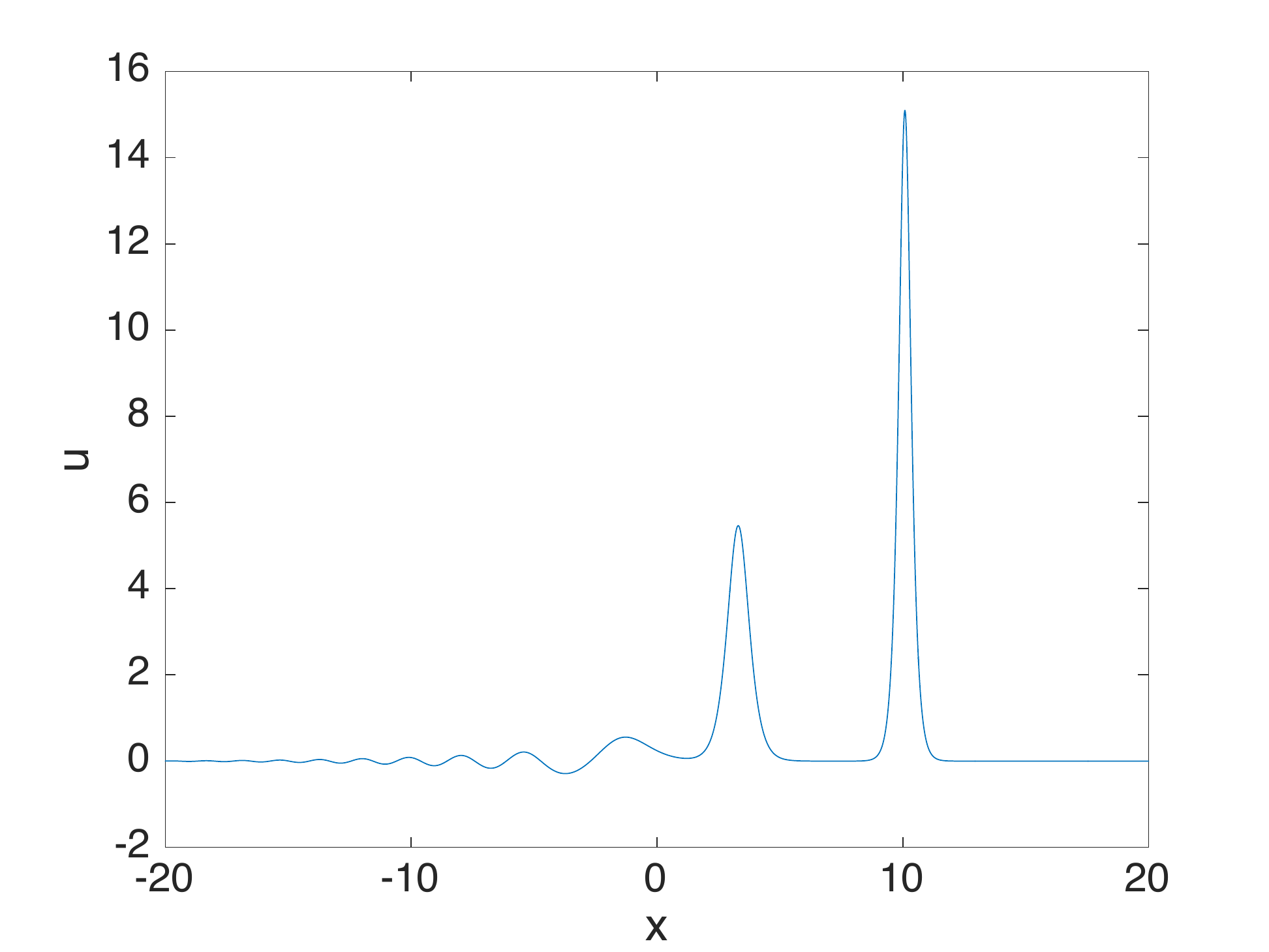}
  \includegraphics[width=0.49\textwidth]{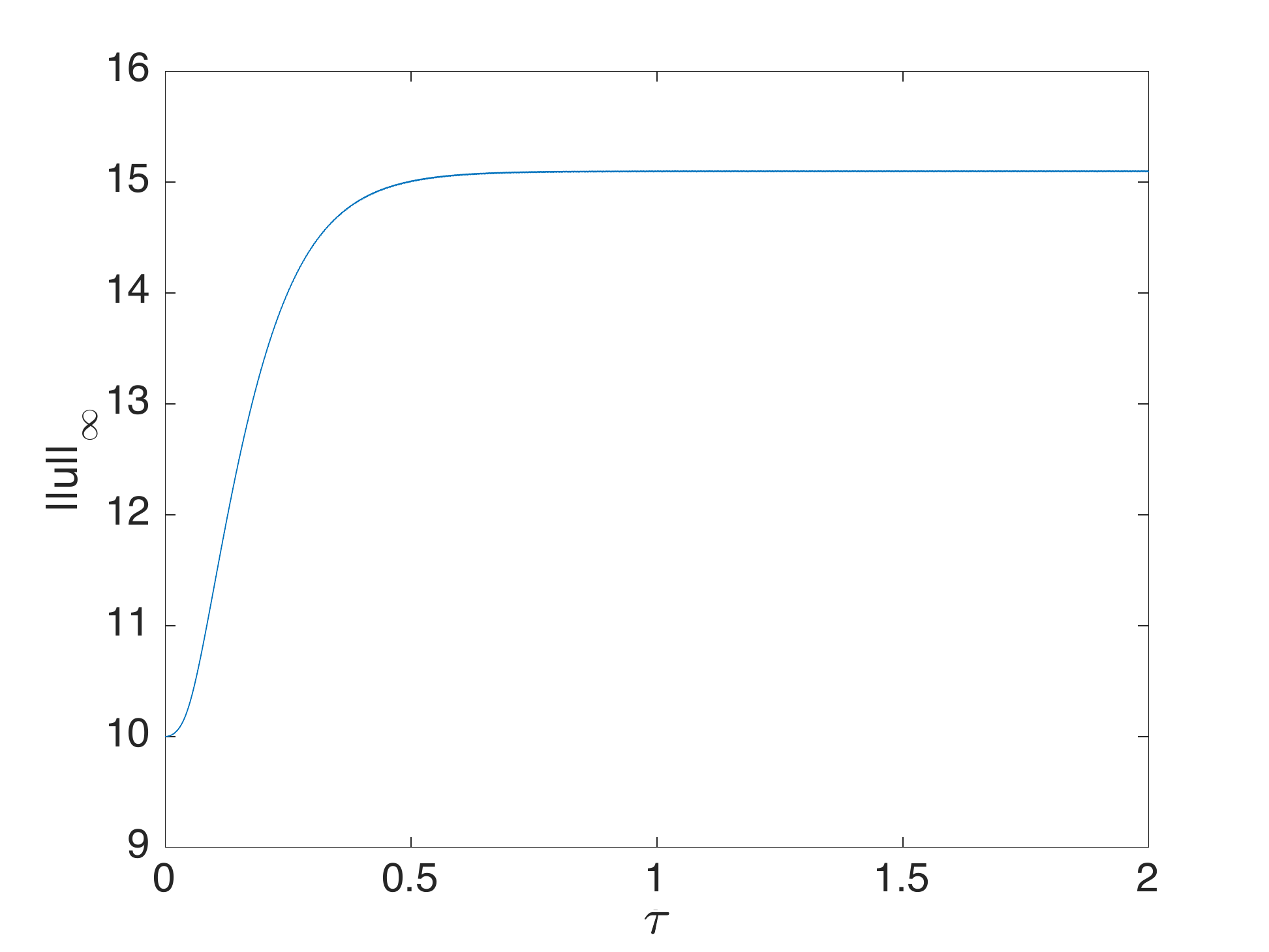}
 \caption{Solution to the Whitham equation with $\epsilon=0.1$ for 
 initial data  $u(x,0)=10\exp(-x^{2})$; on the left the solution for 
 $\tau=2$, 
 on the right the $L^{\infty}$ norm of the solution in dependence of 
 time.}
 \label{figwhitgausse01}
\end{figure}

The solution in Fig.\ref{figwhitgausse01} for small $\epsilon$ 
becomes obviously closer to the solution to the KdV equation 
$u_{\tau}+uu_{x}+u_{xxx}/6$, which can be obtained from the Whitham 
equation in the formal limit $\epsilon\to0$ for the same initial 
data. This KdV solution is shown in 
Fig.\ref{figwhitham10gausskdv}
\begin{figure}[htb!]
  \includegraphics[width=0.7\textwidth]{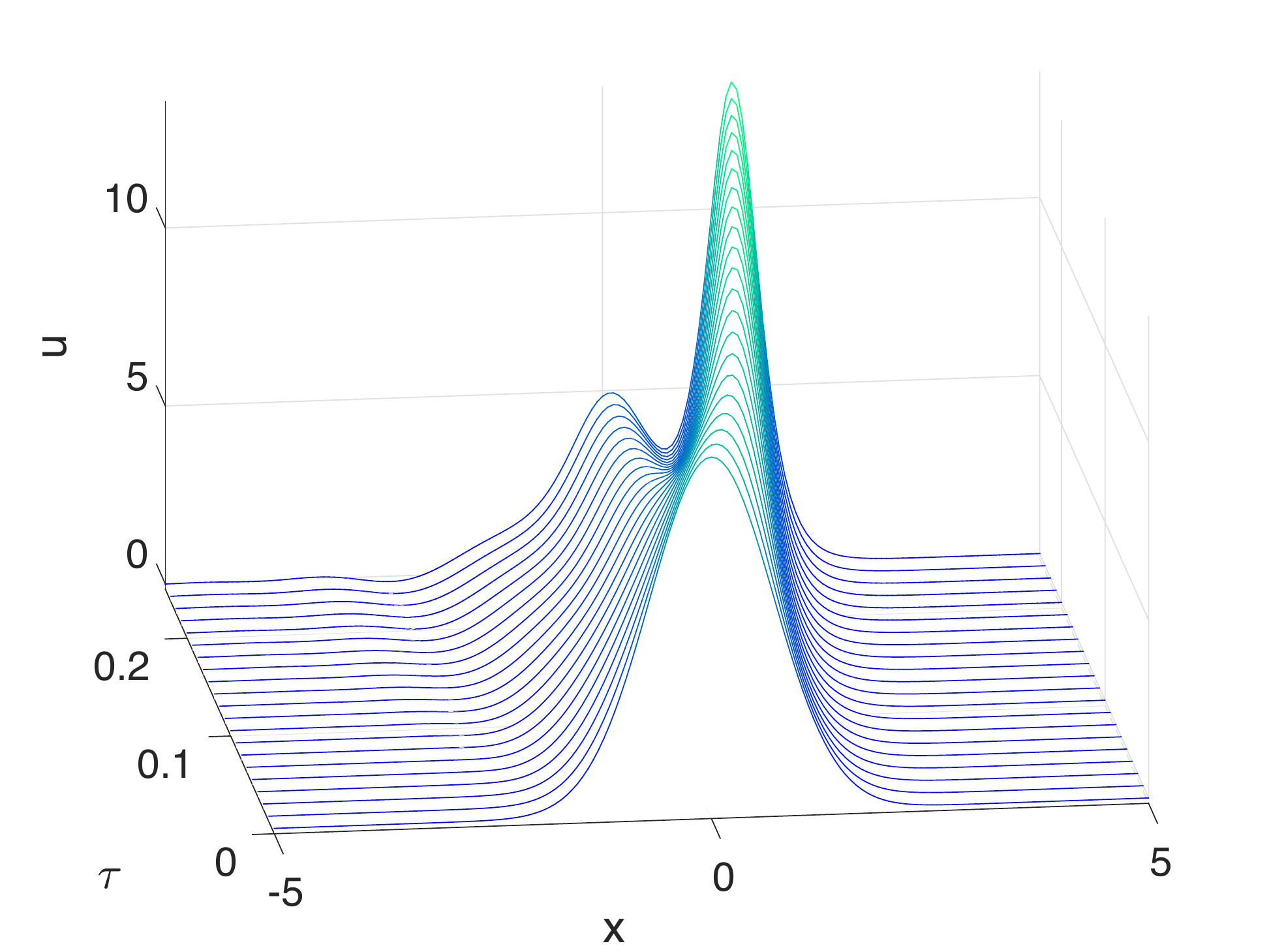}
 \caption{Solution to the KdV equation $u_{\tau}+uu_{x}+u_{xxx}/6$
 for 
 initial data  $u(x,0)=10\exp(-x^{2})$ in dependence of 
 time.}
 \label{figwhitham10gausskdv}
\end{figure}

\subsection{Numerical study of the Whitham equation with surface 
tension}
In this subsection we will study the Whitham equation with surface 
tension (\ref{WhitST}) for similar situations as above for the 
Whitham equation without surface tension: solitons, their stability 
and general initial data in the Schwartz class. The goal is mainly to 
highlight differences with respect to the case without surface 
tension.

Arnesen \cite{Ar} showed that solitons exist in the case of 
non-vanishing surface tension for all values of $\epsilon$. For small 
$\epsilon$, a similar calculation as for the Whitham equation in 
(\ref{wtravel4}) leads to the result that the Whitham soliton should 
be in some sense close to the KdV soliton given by
\begin{equation}
    U=3\delta
    \mbox{sech}^{2}\left(\sqrt{\frac{3\delta}{2(1-3\beta}}(x-ct)\right).
    \label{wtravel4b}
\end{equation}
This means that for $\beta<1/3$, the solitons should have positive 
amplitude, whereas they are depression waves (negative amplitude) for 
$\beta>1/3$. Numerically we construct the solutions as before. Since 
the behavior for the high wave numbers of the operator 
$\hat{\mathcal{L}}_{\epsilon}$ (\ref{dispST}) is different (linear in 
$|k|$) than for 
the case $\beta=0$ (constant in $|k|$), it is useful to divide by 
$\hat{\mathcal{L}}_{\epsilon}$ in the equation for the solitary wave 
in Fourier space. This is straight forward because the 
operator is diagonal in Fourier space. For small $\epsilon$ and 
$\delta$ the solitons are close to the KdV  soliton. All figures are 
created with $N=2^{12}$ Fourier modes. 

We first address the case $\beta=0.1$ and $\delta=2$ which can be seen 
on the left of Fig.~\ref{figwhithamsole01}. The solitary wave for 
the Whitham equation with surface tension is slightly more peaked 
than the KdV soliton in red in this case. If we propagate the exact 
solitary wave in a commoving frame by the time evolution code (as in \cite{KP2} it is 
useful to apply a simplified Newton iteration instead of a fixed 
point iteration), the difference between numerical solution and 
initial data for $N_{t}=10^{4} $ time steps for 
$\tau<1$ is of the order of the numerical error for the solitary wave 
($10^{-12}$). 

If we perturb this solution 
slightly by considering the initial data $u(x,0)=0.99 U(x,c=1.02)$, 
i.e., initial data close to the solitary wave with a slightly smaller 
mass, one can see on the right of Fig.~\ref{figwhithamsole01} that the 
soliton is dispersed to infinity. In contrast to the case without 
surface tension, there are no visible oscillations, the initial pulse 
just gets broader and decreases in height. 
\begin{figure}[htb!]
  \includegraphics[width=0.49\textwidth]{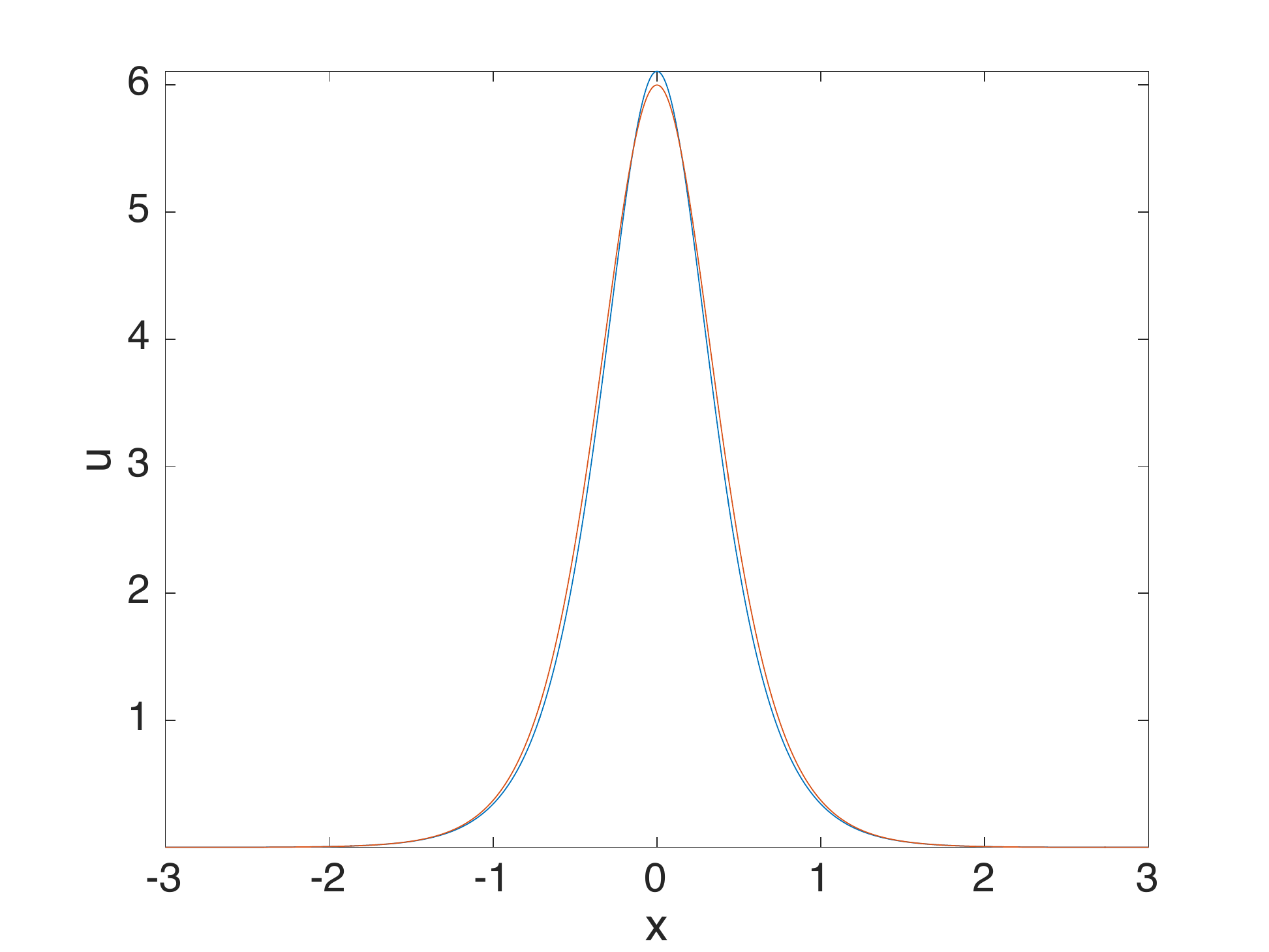}
  \includegraphics[width=0.49\textwidth]{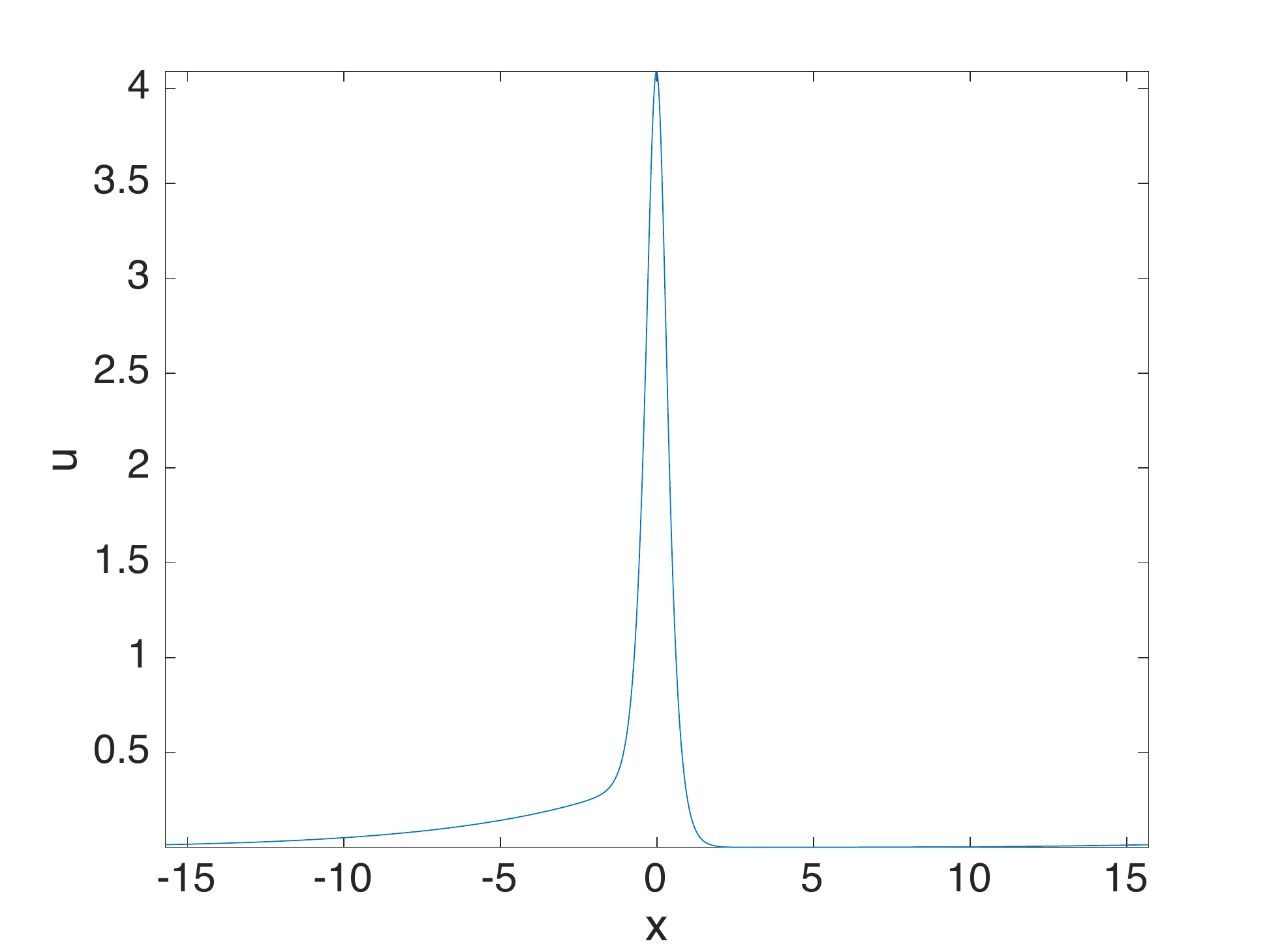}
 \caption{Solitary wave to the Whitham equation with $\epsilon=0.1$ 
 with surface tension $\beta=0.1$ for $c=1.02$ and the corresponding 
 KdV soliton (\ref{wtravel4b}) in red on the left; on the right the 
 solution to the Whitham equation with $\epsilon=0.1$ and surface 
 tension $\beta=0.1$ for the initial data $u(x,0)=0.99 U(x,c=1.02)$ 
 at $\tau=5$.}
 \label{figwhithamsole01}
\end{figure}

If we consider for the same values of $\epsilon$ and $\beta$ as in 
Fig.~\ref{figwhithamsole01} the initial data $u(x,0)=1.01U(x,c=1.02)$, 
i.e., initial data in the vicinity of the solitary wave with slightly 
larger mass, one can see in Fig.~\ref{figwhithamsole01101} that the 
solution has an $L^{\infty}$ blow-up in finite time which is clear 
from the right figure. In fact the divergence of the $L^{\infty}$ 
norm after $\tau=5.96$, for which the solution is shown on the left, is 
so rapid that the fitted distance (\ref{fit}) changes from positive 
to negative values too rapidly to obtain a sensible result for $\mu$ 
(the Fourier coefficients deteriorate too rapidly). Thus one finds the 
same behavior for almost solitary initial data as for the $L^{2}$ 
critical gKdV equation \cite{MMR}: exact solitary initial data are not affected, 
initial data with smaller mass will be dispersed to infinity, and 
initial data with larger mass lead to an $L^{\infty}$ blow-up. 
\begin{figure}[htb!]
  \includegraphics[width=0.49\textwidth]{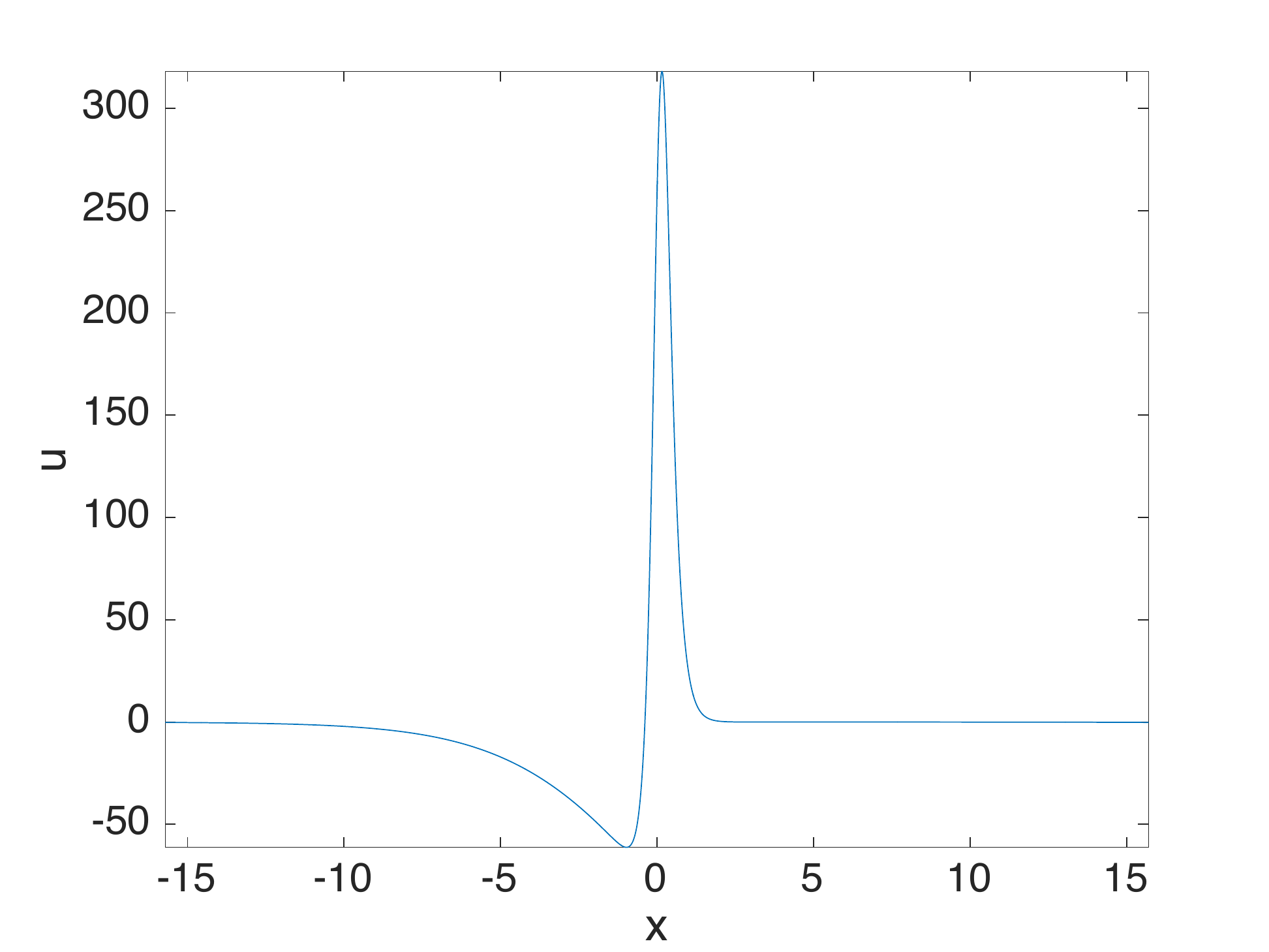}
  \includegraphics[width=0.49\textwidth]{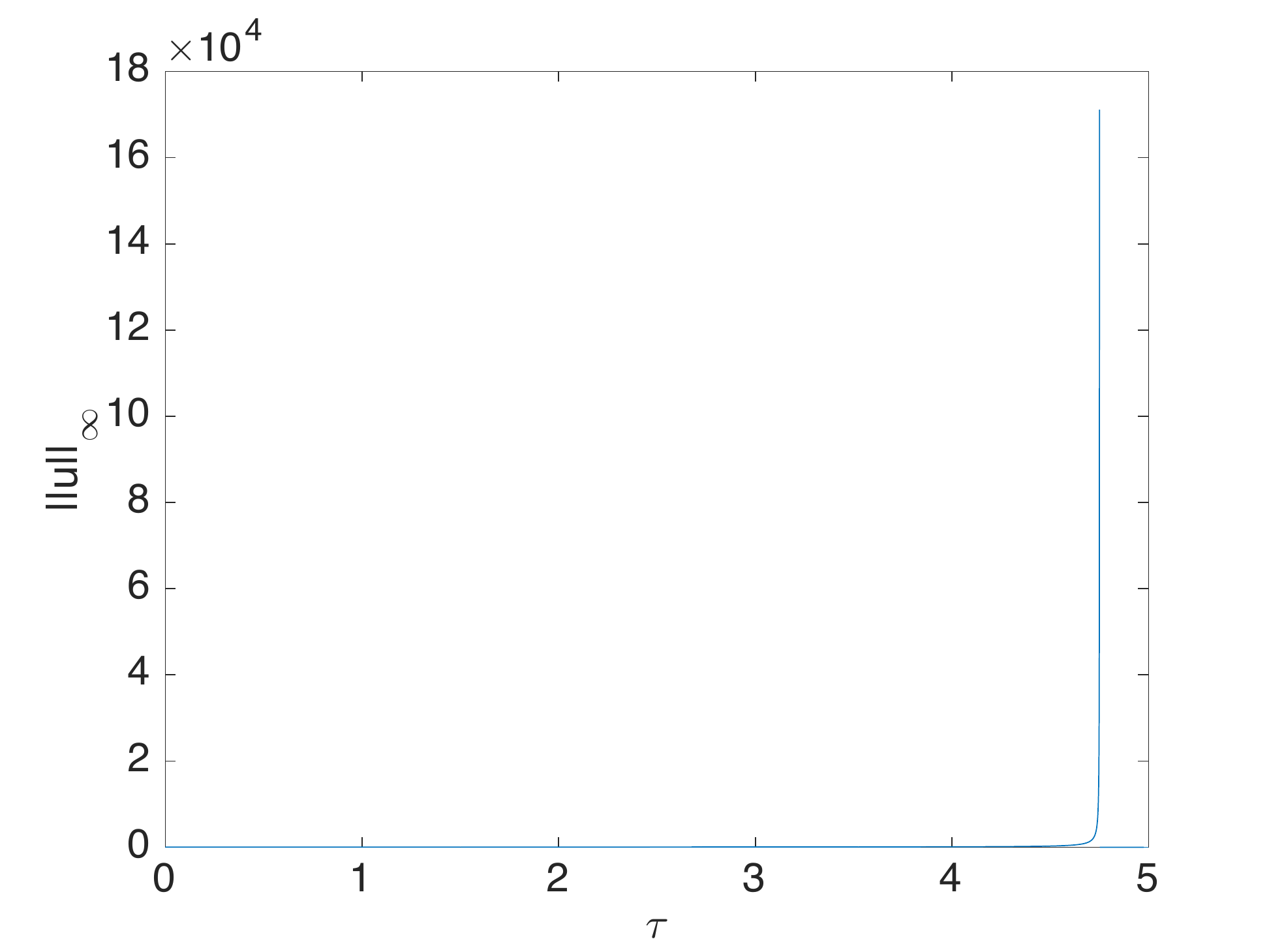}
 \caption{Solution to the Whitham equation with $\epsilon=0.1$ and surface 
 tension $\beta=0.1$ for the initial data $u(x,0)=1.01 U(x,c=1.02)$ 
 at $\tau=5.96$ on the left the $L^{\infty}$ norm of the solution in 
 dependence of time on the right.}
 \label{figwhithamsole01101}
\end{figure}

For $\beta=1$ and $\delta=-20$ the soliton has negative amplitude and 
 can be seen 
on the left of Fig.~\ref{figwhithamsole01b1}. The solitary wave for 
the Whitham equation with surface tension is again slightly more peaked 
than the KdV soliton in red. 
If we perturb this solution 
 by considering the initial data $u(x,0)=0.99 U(x,c=0.8)$, 
one can see on the right of Fig.~\ref{figwhithamsole01b1} that the 
soliton is once more dispersed to infinity, this time to the right, 
but again without visible oscillations. 
\begin{figure}[htb!]
  \includegraphics[width=0.49\textwidth]{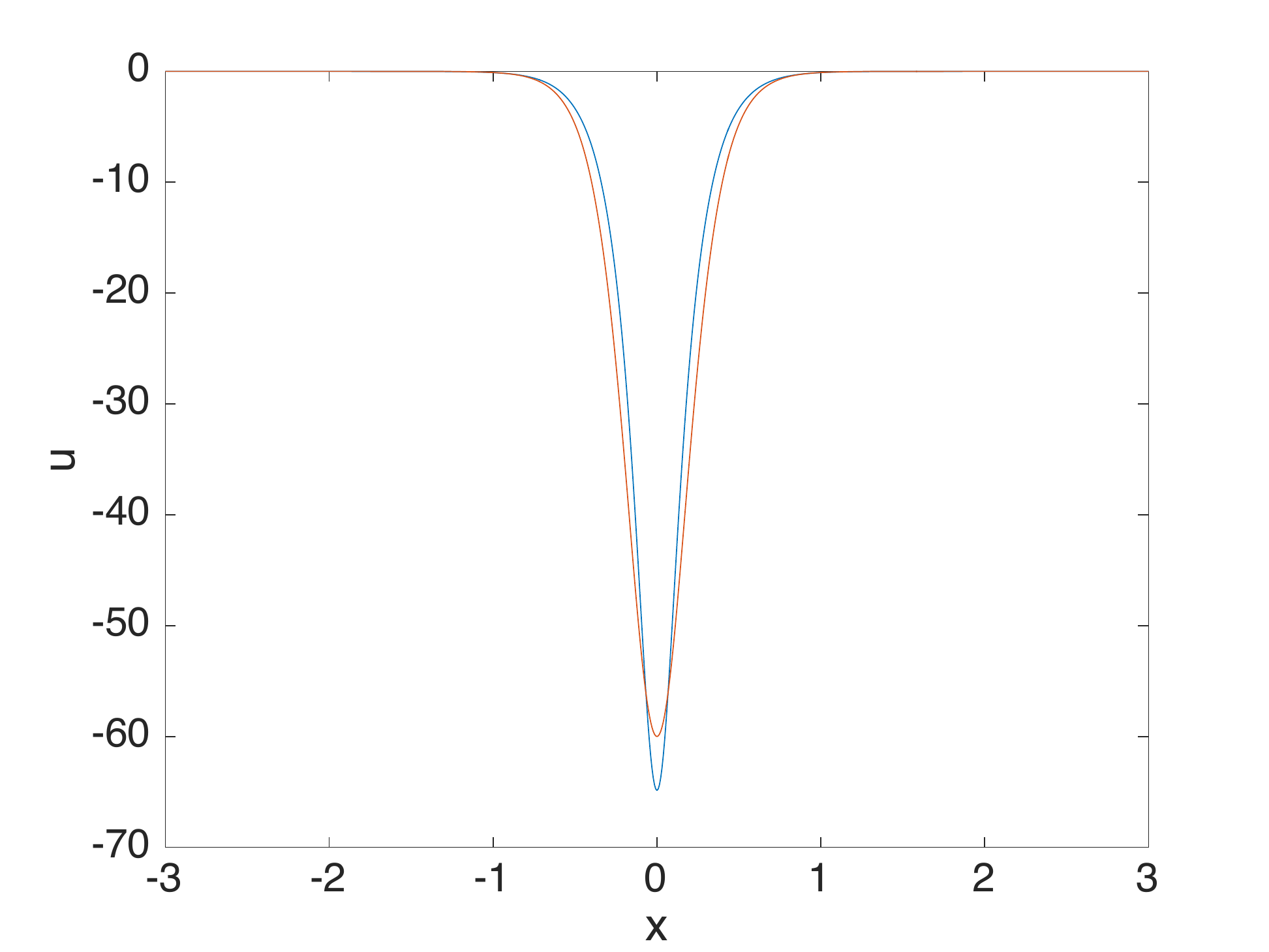}
  \includegraphics[width=0.49\textwidth]{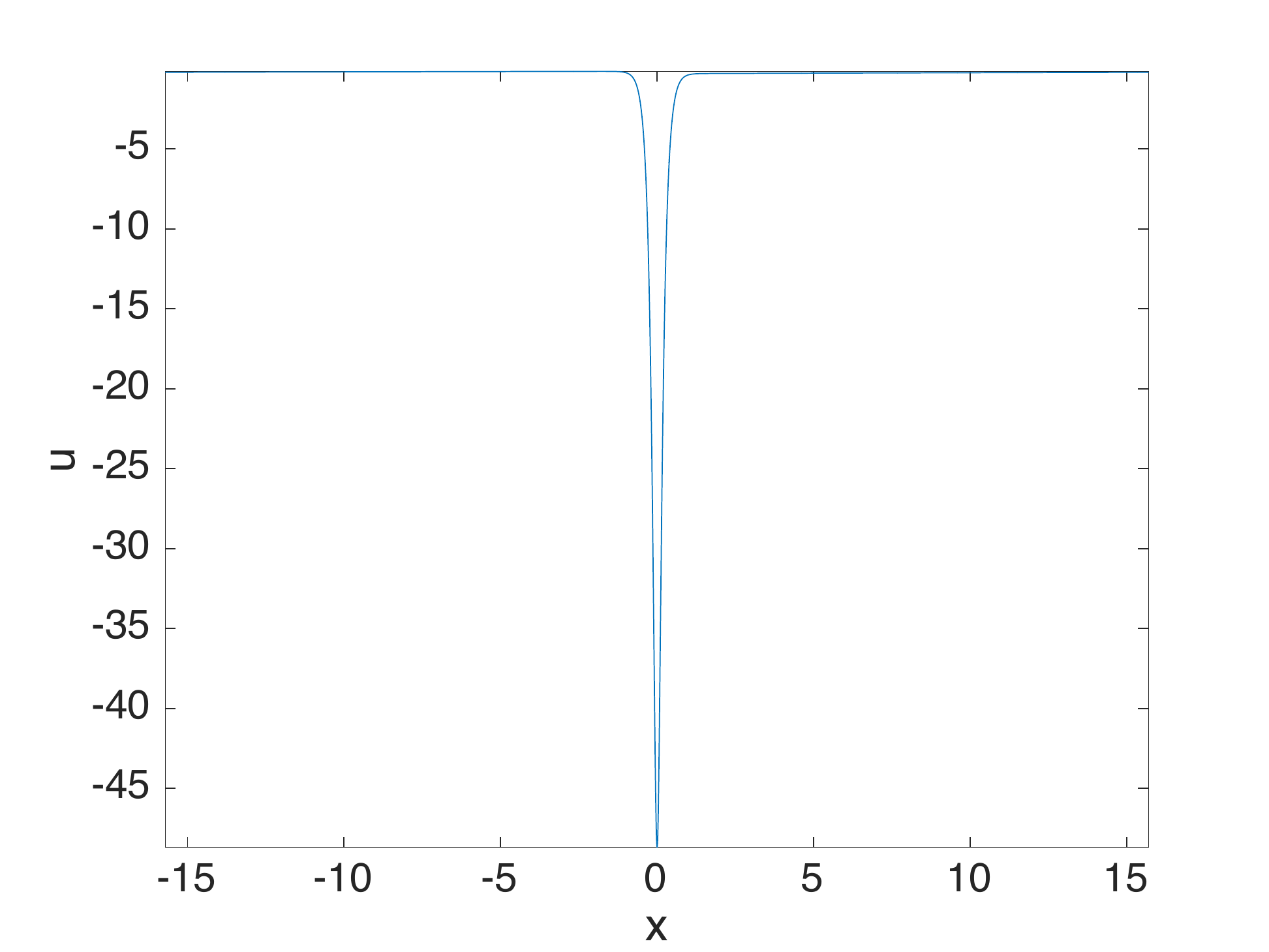}
 \caption{Solitary wave to the Whitham equation with $\epsilon=0.1$ 
 with surface tension $\beta=1$ for $c=0.8$ and the corresponding 
 KdV soliton (\ref{wtravel4b}) in red on the left; on the right the 
 solution to the Whitham equation with $\epsilon=0.1$ and surface 
 tension $\beta=1$ for the initial data $u(x,0)=0.99 U(x,c=0.8)$ at 
 $\tau=3.6$.}
 \label{figwhithamsole01b1}
\end{figure}

If we consider for the same values of $\epsilon$ and $\beta$ as in 
Fig.~\ref{figwhithamsole01b1} the initial data $u(x,0)=1.01U(x,c=0.8)$, 
one can see in Fig.~\ref{figwhithamsole01101b1} that the 
solution has again an $L^{\infty}$ blow-up in finite time which is clear 
from the right figure. The divergence of the $L^{\infty}$ 
norm after $\tau=4.75$, for which the solution is shown on the left, is 
so rapid that the iteration stops converging. Once more one finds the 
same behavior for almost solitary initial data as for the $L^{2}$ 
critical gKdV equation. 
\begin{figure}[htb!]
  \includegraphics[width=0.49\textwidth]{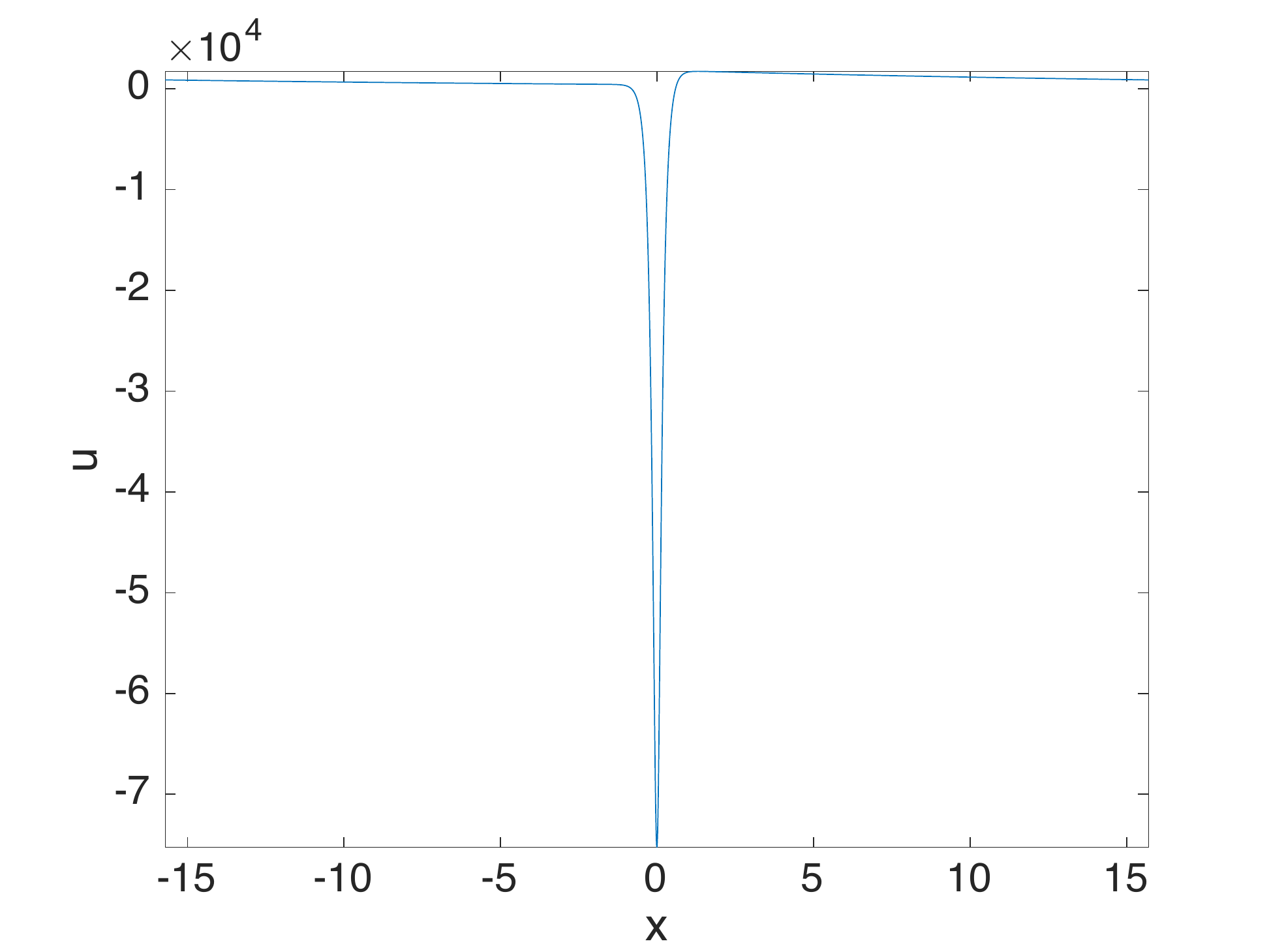}
  \includegraphics[width=0.49\textwidth]{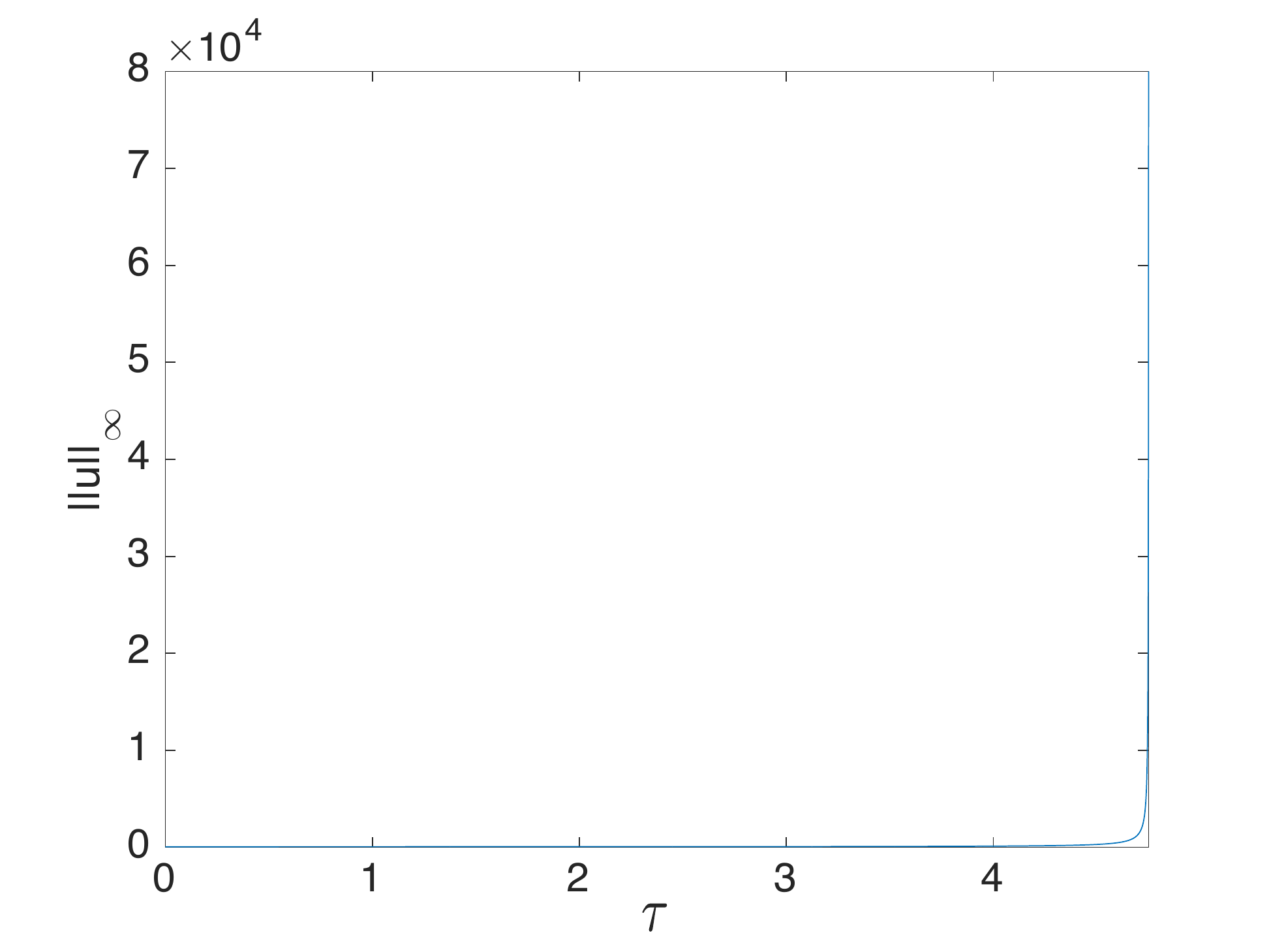}
 \caption{Solution to the Whitham equation with $\epsilon=0.1$ and surface 
 tension $\beta=1$ for the initial data $u(x,0)=1.01 U(x,c=0.8)$ 
 at $\tau=4.75$ on the left and the $L^{\infty}$ norm of the solution in 
 dependence of time on the right.}
 \label{figwhithamsole01101b1}
\end{figure}

To address the case of general initial data in the Schwartz class, we 
consider the same initial data as in the
case $\beta=0$ in Fig.~\ref{figwhitgausse1}. The 
numerical results indicate that there is a hyperbolic blow-up in this 
case.
In Fig.~\ref{figwhithamb110gauss} we show in contrast the corresponding 
situation for $\beta=1$. We use $N=2^{14}$ Fourier modes on 
$x\in[-2\pi,2\pi]$ and $N_{t}=50000$ time steps for $\tau\in[0,0.2]$. 
For small times, the dynamics of 
Burgers' equation is again dominant leading to a steepening of the 
right front towards the formation of a shock. Close to a potential 
gradient catastrophe, the stronger dispersion (compared to the 
Whitham equation with $\beta=0$) takes over and appears to generate a 
dispersive shock as known from the KdV equation. However, as becomes 
clear from the right figure in Fig.~\ref{figwhithamb110gauss}, 
instead of a dispersive shock wave an $L^{\infty}$ blow-up as for 
generalized KdV equations is observed, see e.g. \cite{MM,KP2} and 
references therein.
\begin{figure}[htb!]
  \includegraphics[width=0.49\textwidth]{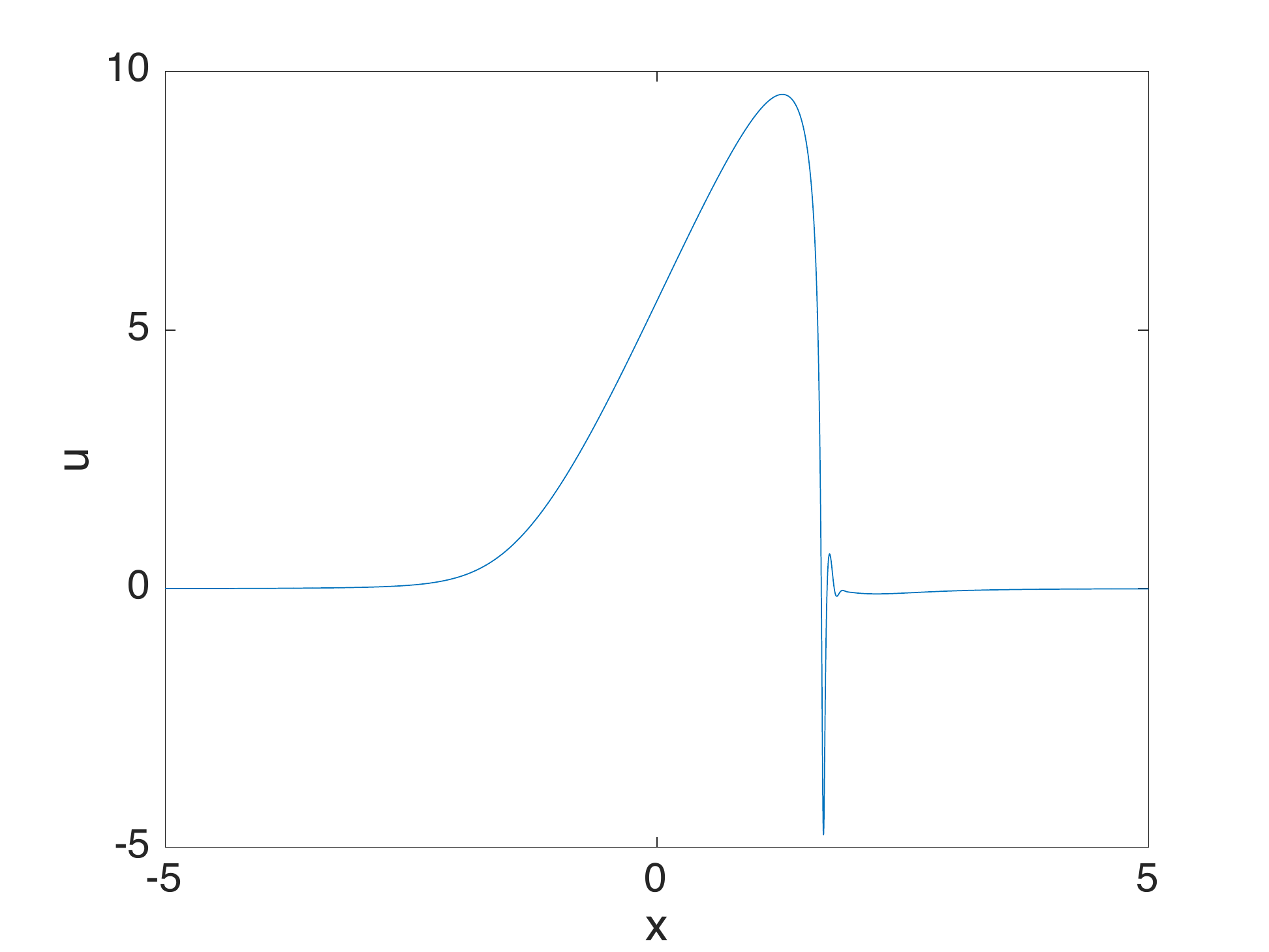}
  \includegraphics[width=0.49\textwidth]{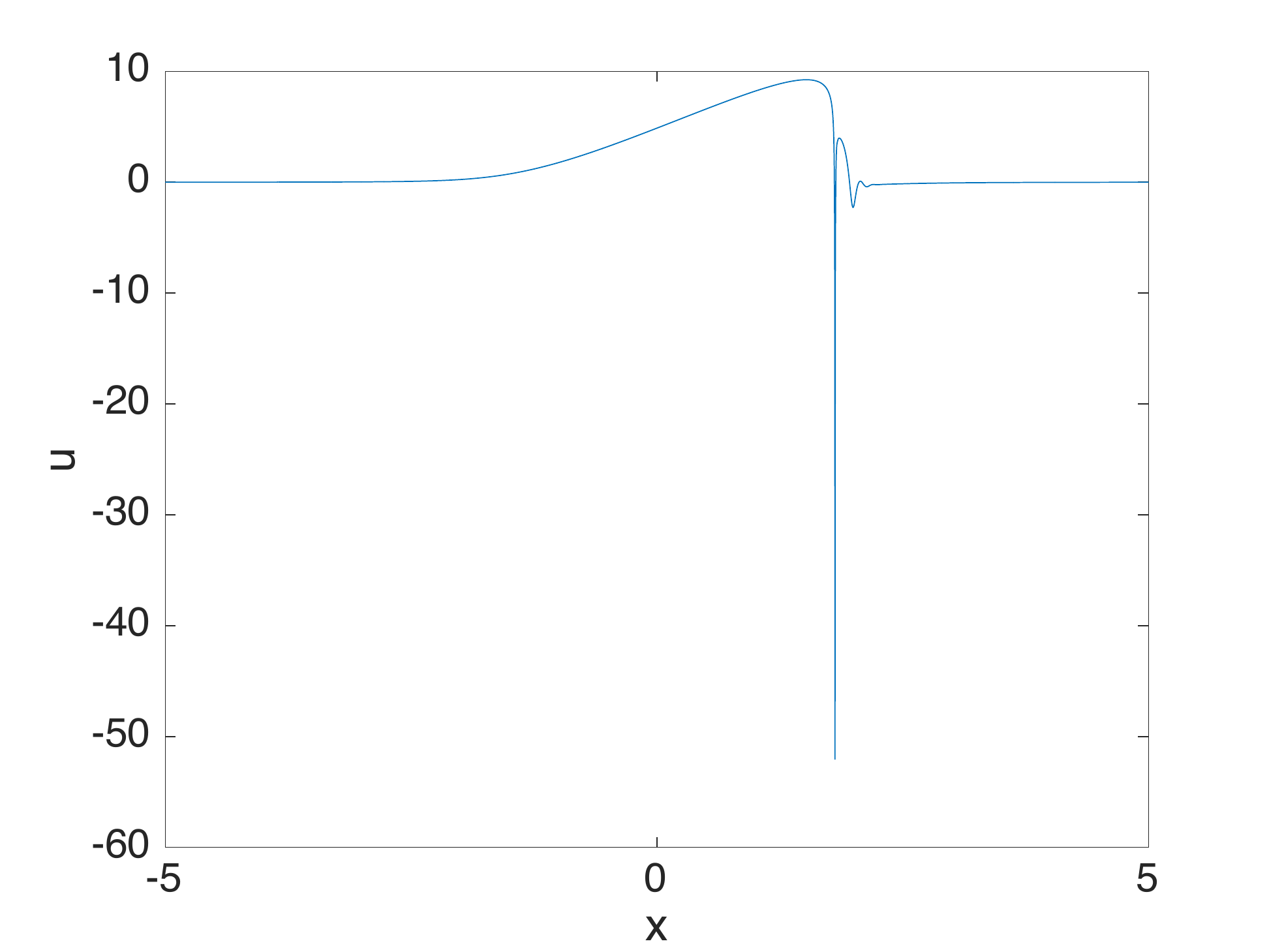}
 \caption{Solution to the Whitham equation with surfaces tension for 
 $\epsilon=1$, $\beta=1$ and the initial data 
 $u(x,0)=10\exp(-x^{2})$; on the left for $\tau=0.13$, on the right for 
 the last recorded time $\tau=0.1648$.}
 \label{figwhithamb110gauss}
\end{figure}

A fitting of the Fourier coefficients in 
Fig.~\ref{figwhithamb110gaussfourier} according to (\ref{fit}) 
indicates indeed that a singularity in the complex plane approaches 
the real axis for $\tau\approx 0.1648$. The critical exponent is found 
to be $\mu=-0.5518$ which confirms an $L^{\infty}$ blow-up. This is in 
accordance with the $L^{\infty}$ norm of the solution shown in 
Fig.~\ref{figwhithamb110gaussfourier} on the right. 
\begin{figure}[htb!]
  \includegraphics[width=0.49\textwidth]{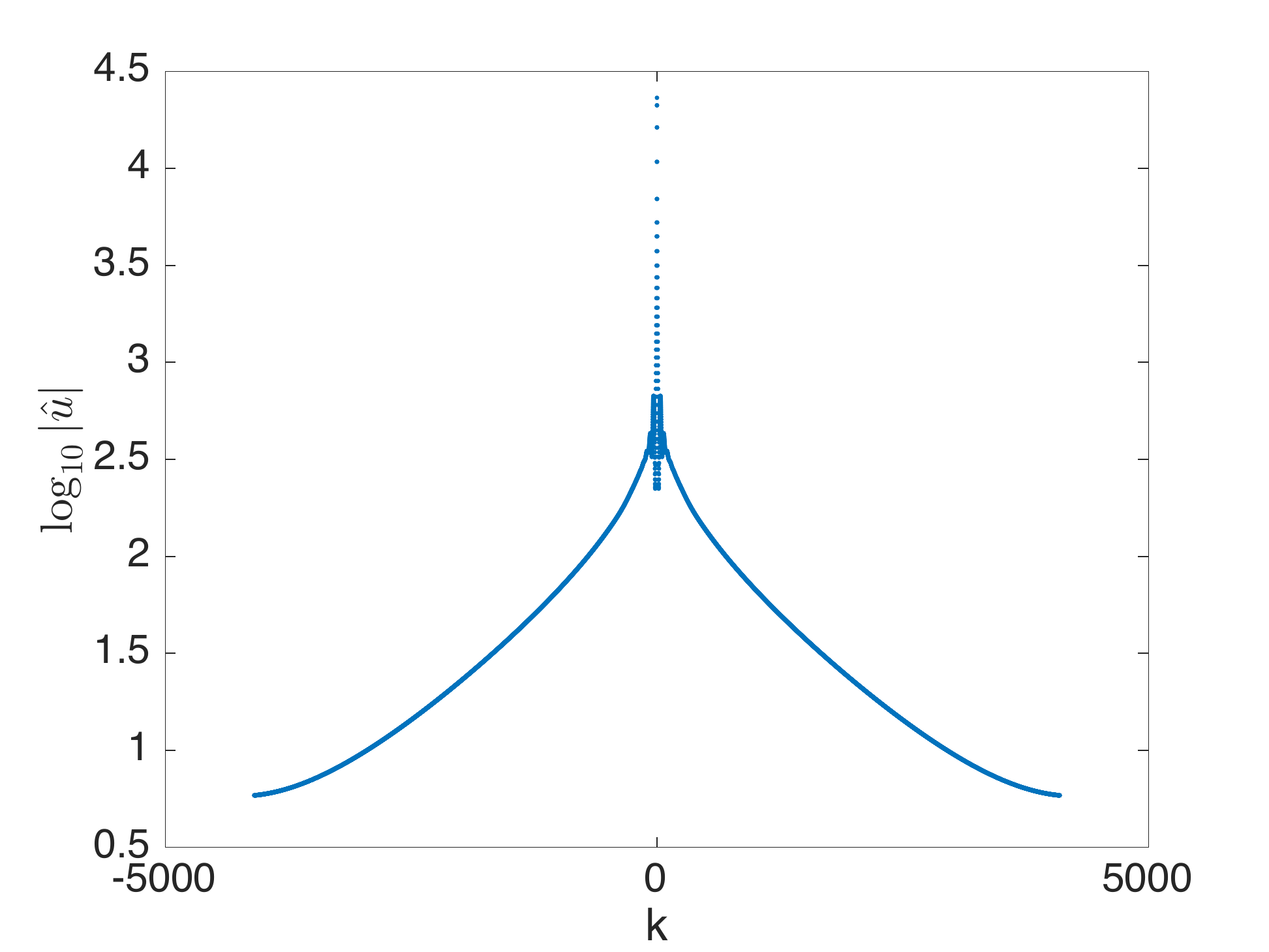}
  \includegraphics[width=0.49\textwidth]{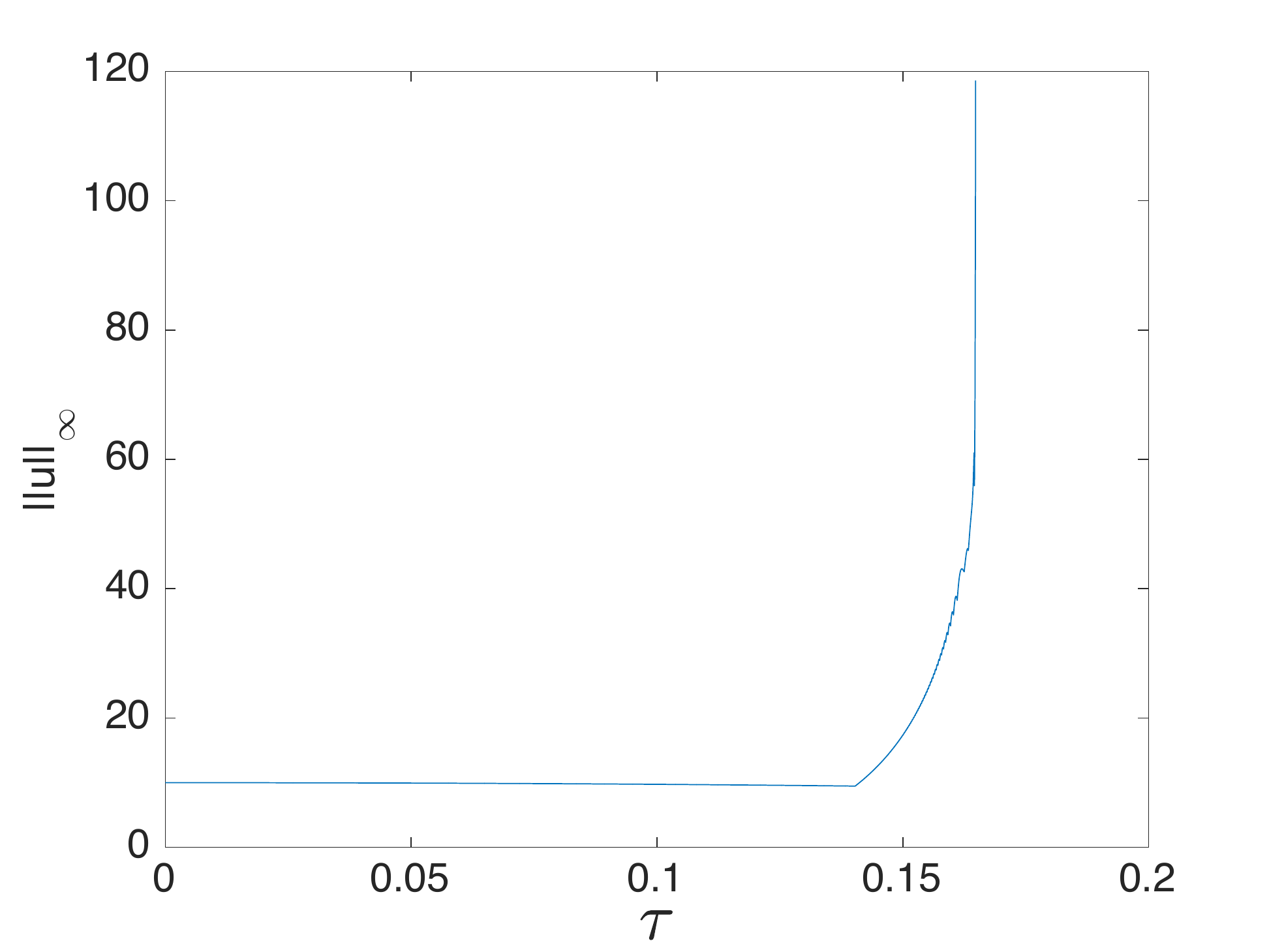}
 \caption{Solution to the Whitham equation with surface tension for 
 $\epsilon=1$, $\beta=1$ and the initial data 
 $u(x,0)=10\exp(-x^{2})$; on the left the modulus of the Fourier 
 coefficients of the solution $t=0.164$, on the right the 
 $L^{\infty}$ norm of the solution in dependence of time.}
 \label{figwhithamb110gaussfourier}
\end{figure}

\section{Numerical simulations for the Boussinesq system}
In this section we numerically construct solitary waves for the 
Boussinesq system and study their stability as well as solutions to 
more general initial 
data.

\subsection{Numerical construction of solitary waves for the 
Boussinesq system}

As in the case of solitary waves for the Whitham equation, we 
construct the solitary waves by solving equation (\ref{travel}) 
with FFT techniques in Fourier space with a Newton-GMRES iteration. 
We choose again $\epsilon=0.01$ and use the KdV soliton (\ref{sol}) as initial 
iterate for small $\alpha$. For larger values of $\alpha$ (which 
implies larger values of the speed $c=1+\alpha \epsilon$), we use the numerical solution to 
(\ref{travel}) for a slightly smaller $\alpha$ as an initial 
iterate. The iterations are carried out on the interval 
$x\in[-5\pi,5\pi]$ with $N=2^{14}$ Fourier modes. For $c=1.16$ we use 
$2^{16}$ Fourier modes, but even  increasing $c$ in small 
increments, we do not observe convergence of the iteration. Again this 
does not prove the non-existence of solitary waves to the system 
(\ref{sys}) at higher speeds, but is an indication that there might 
be an upper limit to the speed of the travelling waves for this 
equation. 

In Fig.~\ref{figsautsyssol}, we show on the left the function $U$ for 
different values of $c$. It can be seen that with increasing speed, 
the solitons become again more peaked and localized. Note that all 
shown solutions are numerically well resolved in the sense that the 
modulus of the Fourier coefficients decreases to machine precision 
for large wave numbers. On the right of Fig.~\ref{figsautsyssol} we show 
the corresponding functions $N$ computed for given $U$ via 
(\ref{travelsys}). Due to the term proportional to $\epsilon$ in 
(\ref{travelsys}), the functions $N$ are less peaked than the 
corresponding function $U$ for large $c$. 
\begin{figure}[htb!]
  \includegraphics[width=0.49\textwidth]{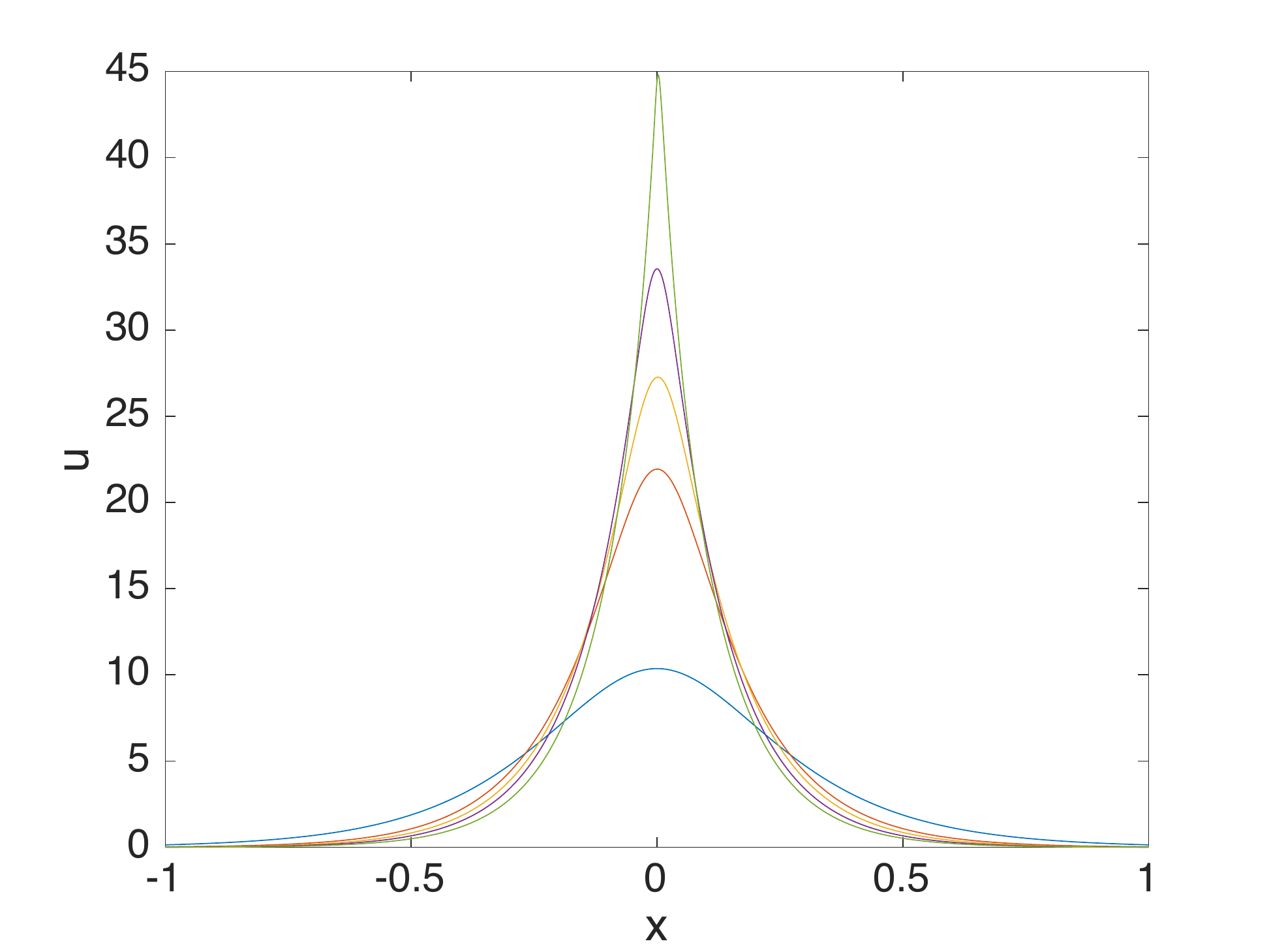}
  \includegraphics[width=0.49\textwidth]{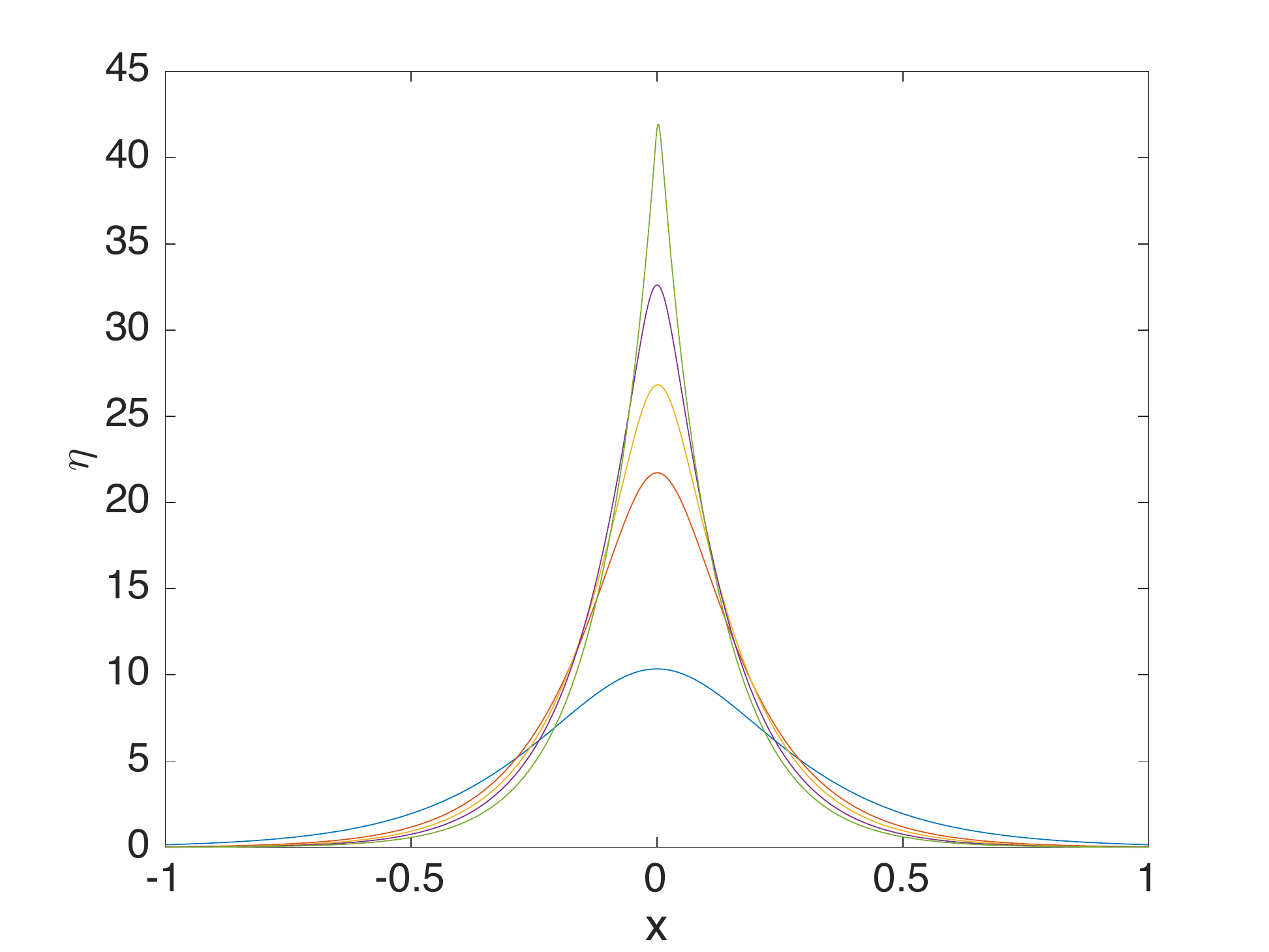}
 \caption{Solitary wave solution to the system (\ref{sys}) for 
 $c=1.05,1.1,1.12,1.14,1.16$ (in order of increasing maxima); on the 
 left the functions $U$, 
 on the right the functions $N$.}
 \label{figsautsyssol}
\end{figure}

\subsection{Numerical study of perturbed solitary waves}
In this section we study numerically solutions to the Cauchy problem for the 
Boussinesq system (\ref{sys}) for several examples. We use the same 
approach as for the Whitham equation, a Fourier spectral method in 
$x$ and an implicit Runge-Kutta method of fourth order with a fix 
point iteration in $t$. Since the system can be seen as a dispersive 
regularization of the ill-posed Kaup system and since it is ill posed 
for negative $\eta$, Krasny filtering \cite{krasny} has to be used 
in order to stabilize the code as in the case of for instance the 
focusing nonlinear Schr\"odinger equation in the semiclassical limit, 
see e.g. \cite{etna,DGK}; this means the Fourier coefficients with a 
modulus smaller than some threshold, typically $10^{-12}$, are put 
equal to zero. The accuracy of the solution is controlled 
via the decrease of the modulus of the Fourier coefficients for large 
wave numbers and the conservation of the numerically computed energy.

If we choose the numerically constructed soliton solution as 
initial data in a commoving frame, the solution is for $c=1.05$ 
numerically evolved for $t\in[0,2]$ with an accuracy of $10^{-14}$ 
($L^{\infty}$ norm of the difference between numerical solution and 
initial data). This shows the accuracy both of the time evolution 
code and the code for the solitons. Note that a rescaling of the time 
as in (\ref{whithamtau}) is not straight forward for the Boussinesq 
system. Therefore we will always consider the time $t$ in this 
section.  

Using initial 
data of the form $u(x,0)=U(x)+\exp(-x^{2})$, $\eta(x,0)=N(x)$, one 
gets the solution shown in Fig.~\ref{figsautsyssolgauss}. It can be 
seen that a slightly larger soliton emerges, and that the remaining 
energy is radiated to the left. 
\begin{figure}[htb!]
  \includegraphics[width=0.49\textwidth]{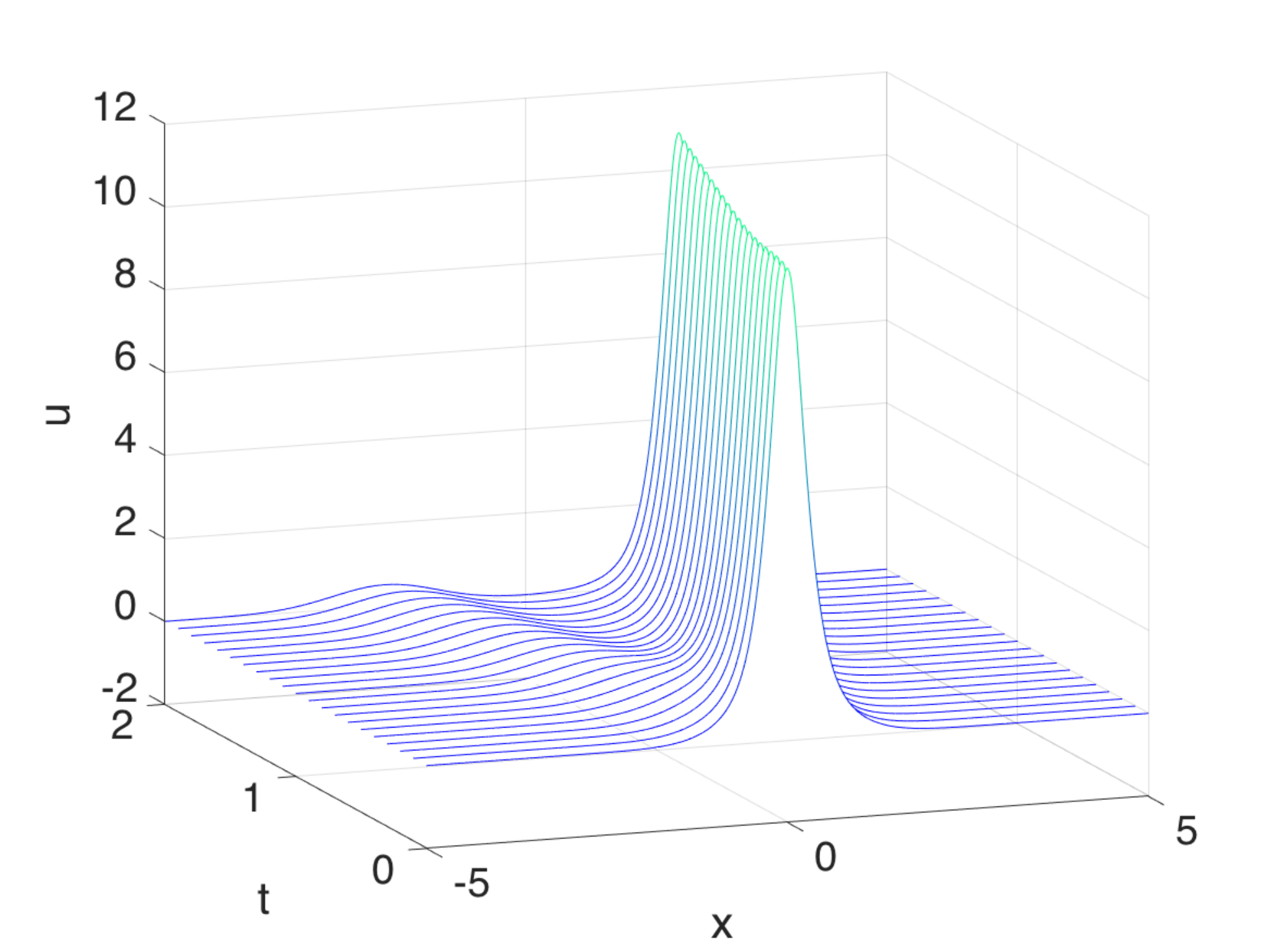}
  \includegraphics[width=0.49\textwidth]{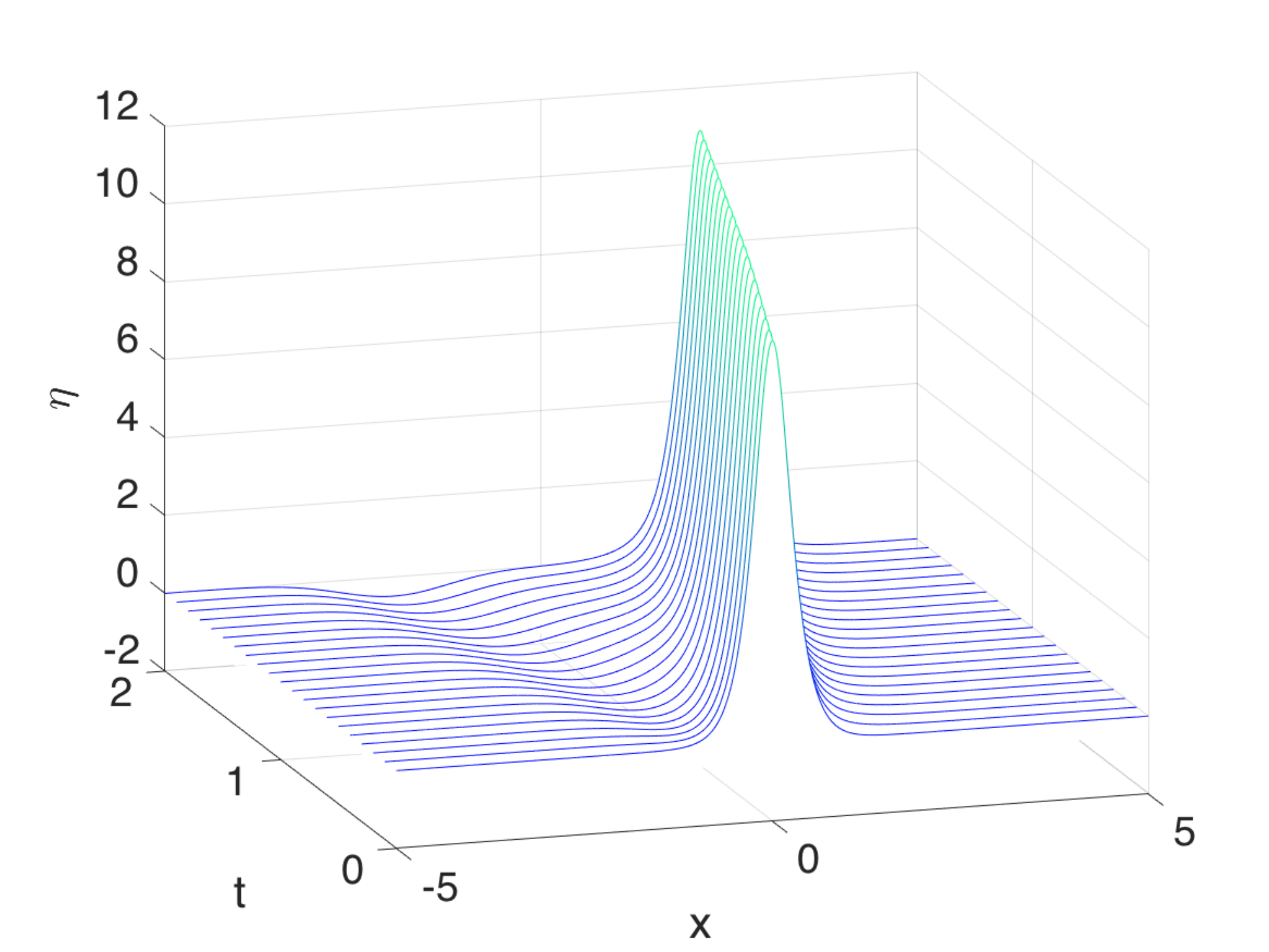}
 \caption{Perturbed soliton of the Boussinesq system (\ref{sys}) for 
 $\epsilon=0.1$, i.e., the solution for the initial 
data  $u(x,0)=U(x)+\exp(-x^{2})$, $\eta(x,0)=N(x)$ in dependence of 
time.}
 \label{figsautsyssolgauss}
\end{figure}

As in the case of the Whitham equation in 
Fig.~\ref{figsautwhitham2sol}, the rapid decrease of the solitons to 
the Boussinesq system allows to study numerically soliton 
interactions. We consider as initial data the solitons with $c=1.05$ 
and $c=1.1$, the latter being centered at $x=-3$. The solution to 
the Boussinesq system for these initial data can be seen in 
Fig.~\ref{figsautsys2sol}. Visibly both $u$ and $\eta$ 
show the behavior of the KdV two-soliton. Note that the solution is 
computed in a frame commoving with $c=1.05$. 
\begin{figure}[htb!]
  \includegraphics[width=0.49\textwidth]{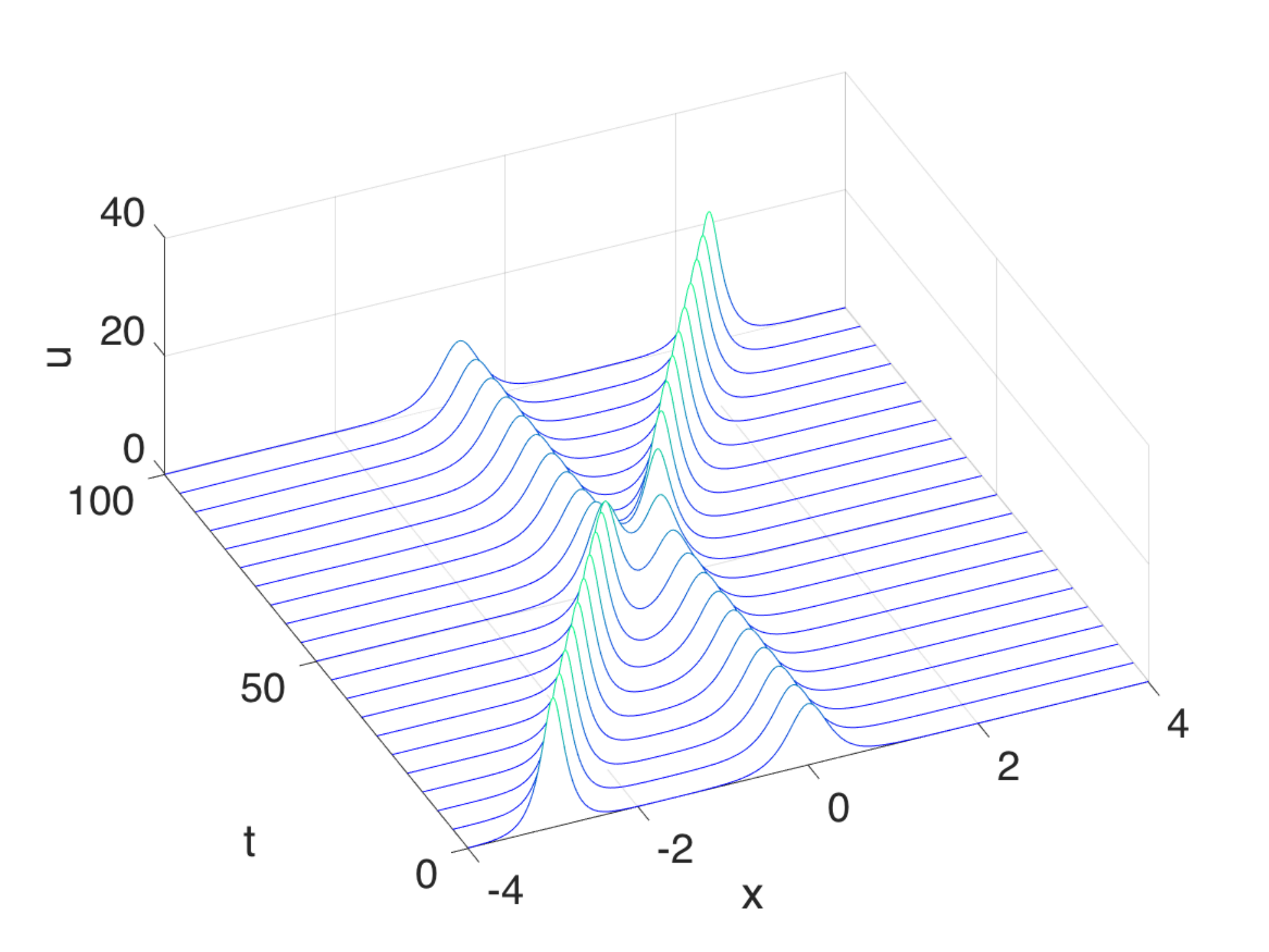}
  \includegraphics[width=0.49\textwidth]{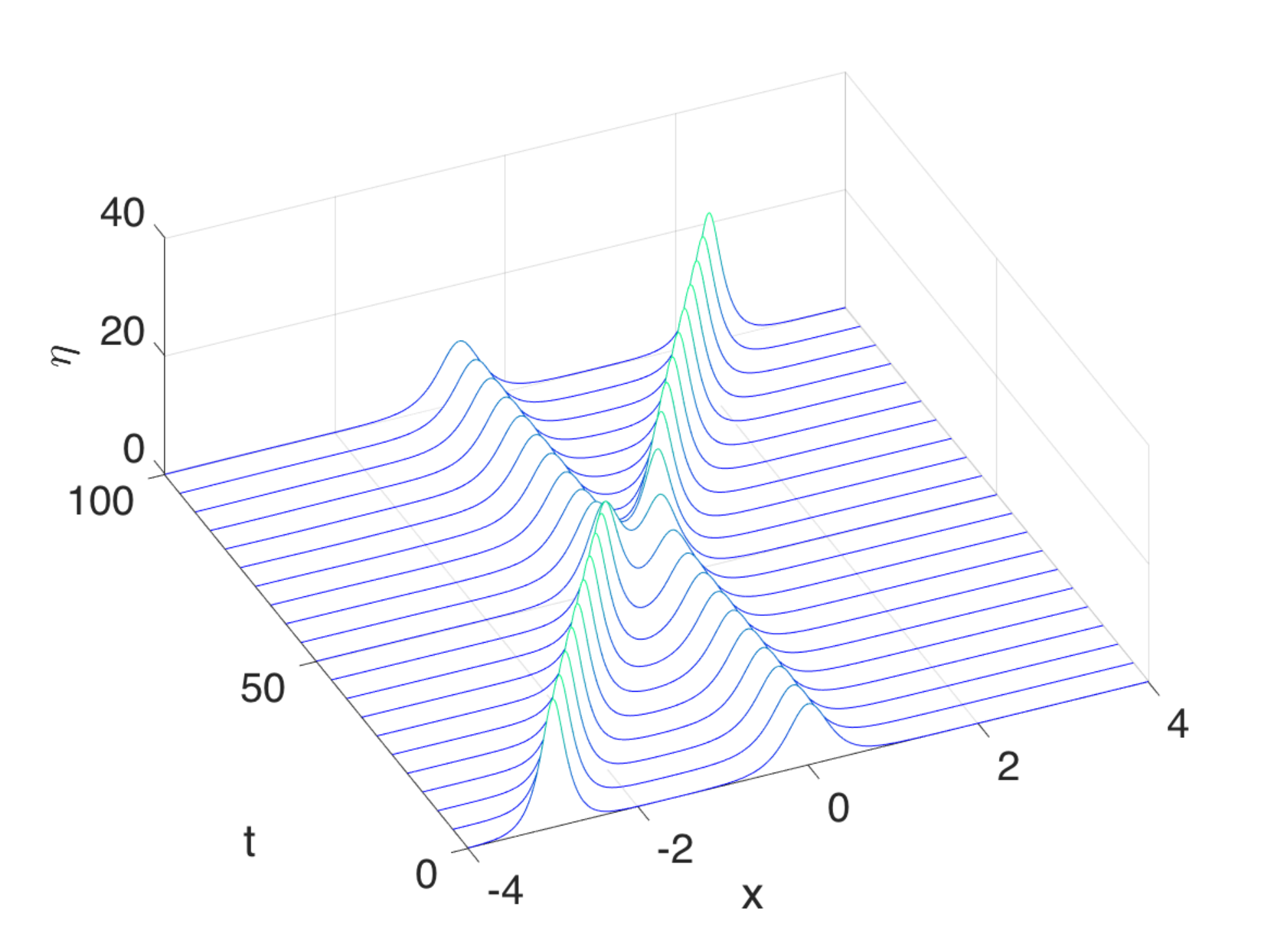}
 \caption{Solution of the Boussinesq system (\ref{sys}) for 
 $\epsilon=0.1$ and initial data being the sum of the solitons with 
 $c=1.05$ centered at $x=0$ and $c=1.1$ centered at $x=-3$ in 
 dependence of time.}
 \label{figsautsys2sol}
\end{figure}

However the solution at $t=100$ in Fig.~\ref{figsautsys2solt100} shows 
upon closer inspection, in particular the close-up on the right, that 
there is dispersive radiation propagating to the left. Thus despite a 
similarity to the KdV two-solitons, the Boussinesq systems does not 
appear to be completely integrable. 
\begin{figure}[htb!]
  \includegraphics[width=0.49\textwidth]{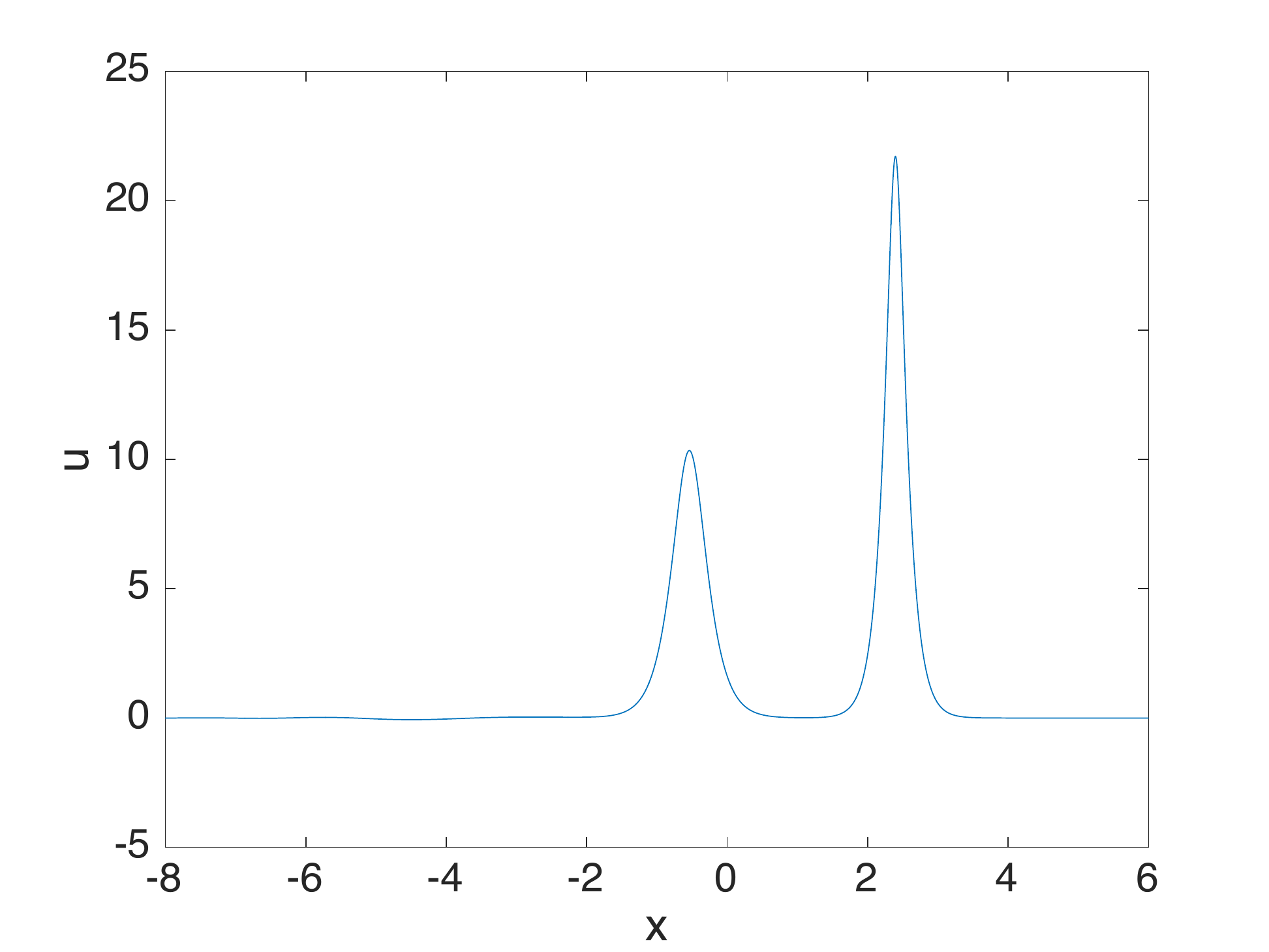}
  \includegraphics[width=0.49\textwidth]{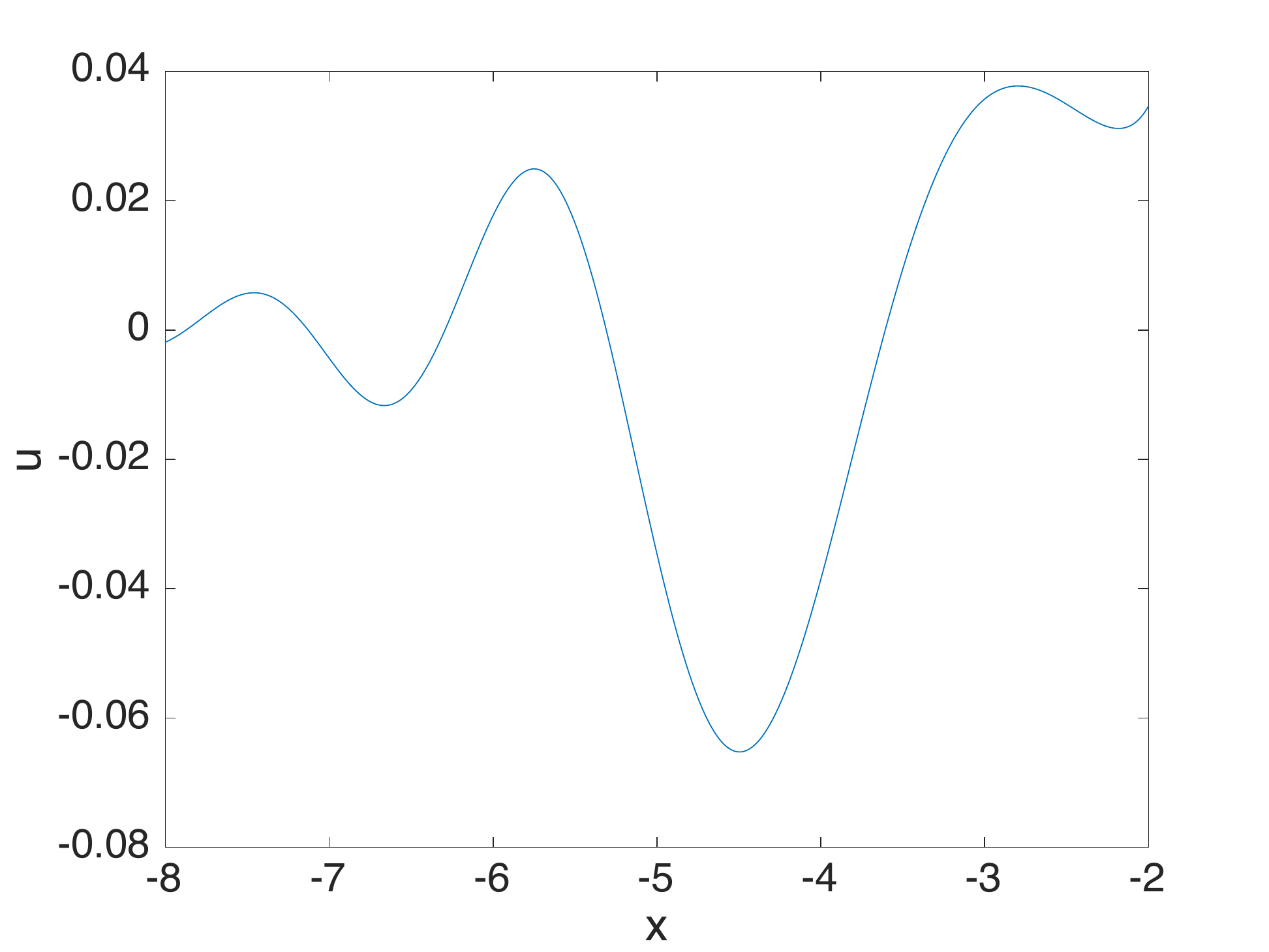}
 \caption{Solution $u$ of the Boussinesq system (\ref{sys}) for 
 $\epsilon=0.1$ and initial data being the sum of the solitons with 
 $c=1.05$ centered at $x=0$ and $c=1.1$ centered at $x=-3$ for 
 $t=100$; on the right a close-up of the situation on the left.}
 \label{figsautsys2solt100}
\end{figure}

\subsection{Numerical study of the Boussinesq system for more general 
initial data} 
In Fig.~\ref{figsautsysgauss} we show the solution to the Boussinesq 
system (\ref{sys} for the initial $u(x,0)=0$ and 
$\eta(x,0)=10\exp(-x^{2})$. The solution breaks at $t=0.4115$ 
since the fit of the Fourier coefficients to (\ref{fit}) indicates 
that a singularity in the complex plane hits the real axis. The 
fitted coefficient $\mu=0.345$ implies the formation of a cusp. The 
solution at the critical time is shown in Fig.~\ref{figsautsysgauss}. 
The behavior is similar to the cusp formation in solutions to the 
semiclassical NLS equation, see e.g.~\cite{DGK}. 
\begin{figure}[htb!]
  \includegraphics[width=0.49\textwidth]{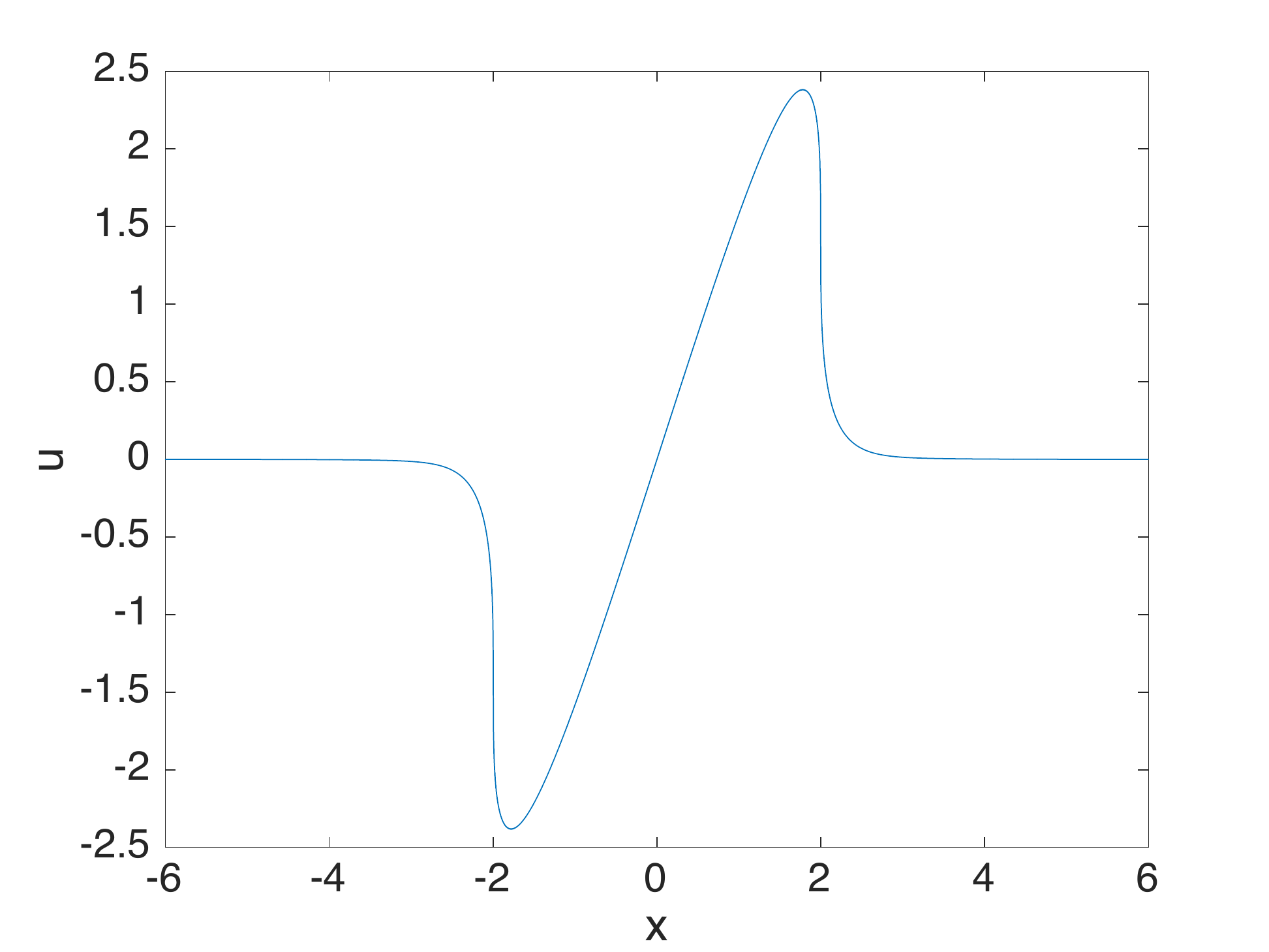}
  \includegraphics[width=0.49\textwidth]{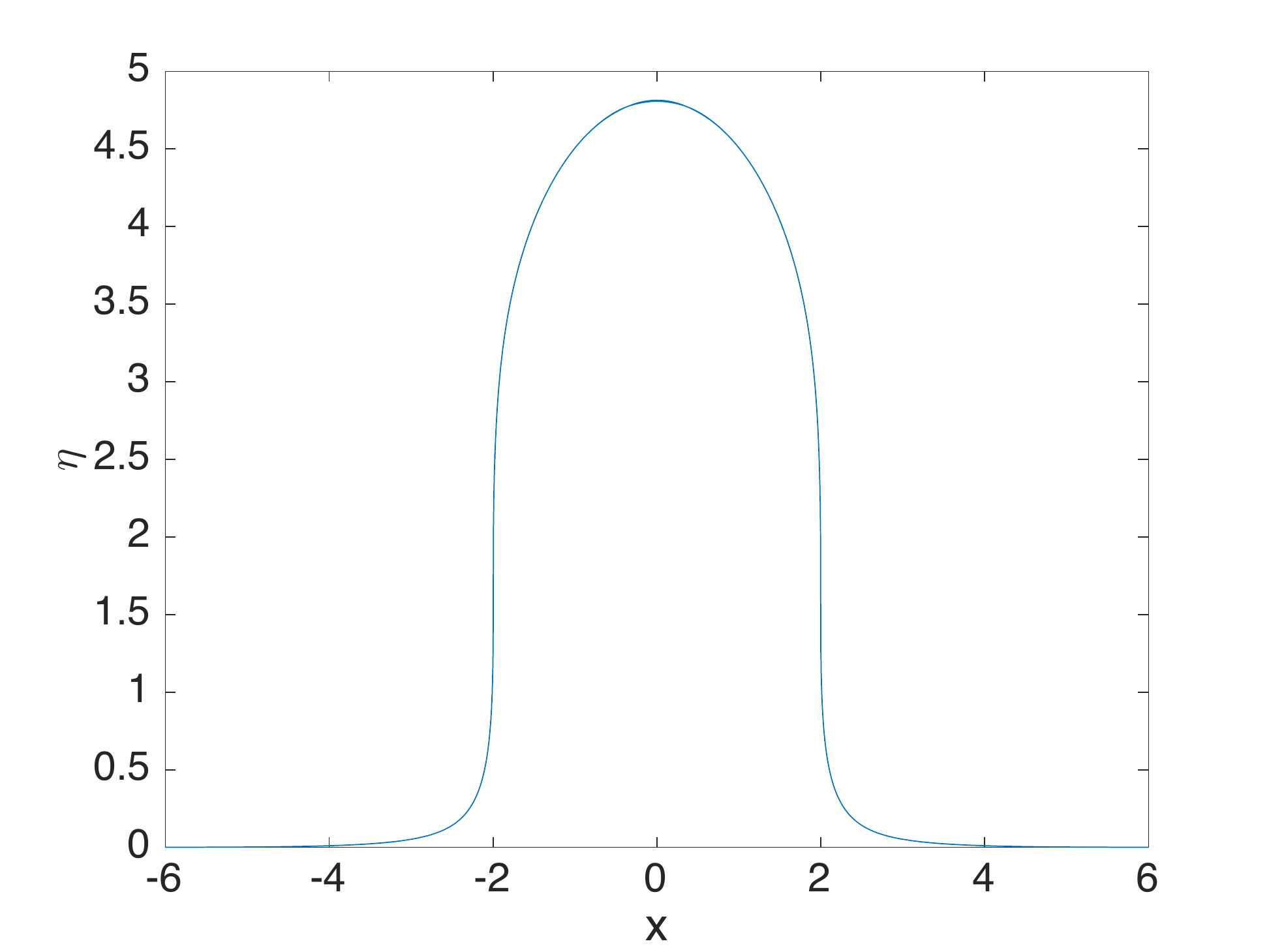}
 \caption{Solution to the Boussinesq system (\ref{sys}) for 
 $\epsilon=1$ and the initial 
data  $u(x,0)=0$, $\eta(x,0)=10\exp(-x^{2})$ at the time $t=0.4115$.}
 \label{figsautsysgauss}
\end{figure}

For $\epsilon=0.1$ and the same initial data, the solution breaks at 
$t=2.2262$ with $\mu=.423$, i.e., again a cusp as shown in 
Fig.~\ref{figsautsysgausse01}. But this time the behavior is 
different from Fig.~\ref{figsautsysgauss}: the solution follows for a 
certain time the underlying wave equations and two humps are forming 
from the initial hump. But dispersion is too weak to overcome the 
nonlinearity and cusps form in finite time. Note that there is no 
clear scaling for the blow-up time in dependence of $\epsilon$ here 
since the mechanisms for the blow-up appear to be different. 

\begin{figure}[htb!]
  \includegraphics[width=0.49\textwidth]{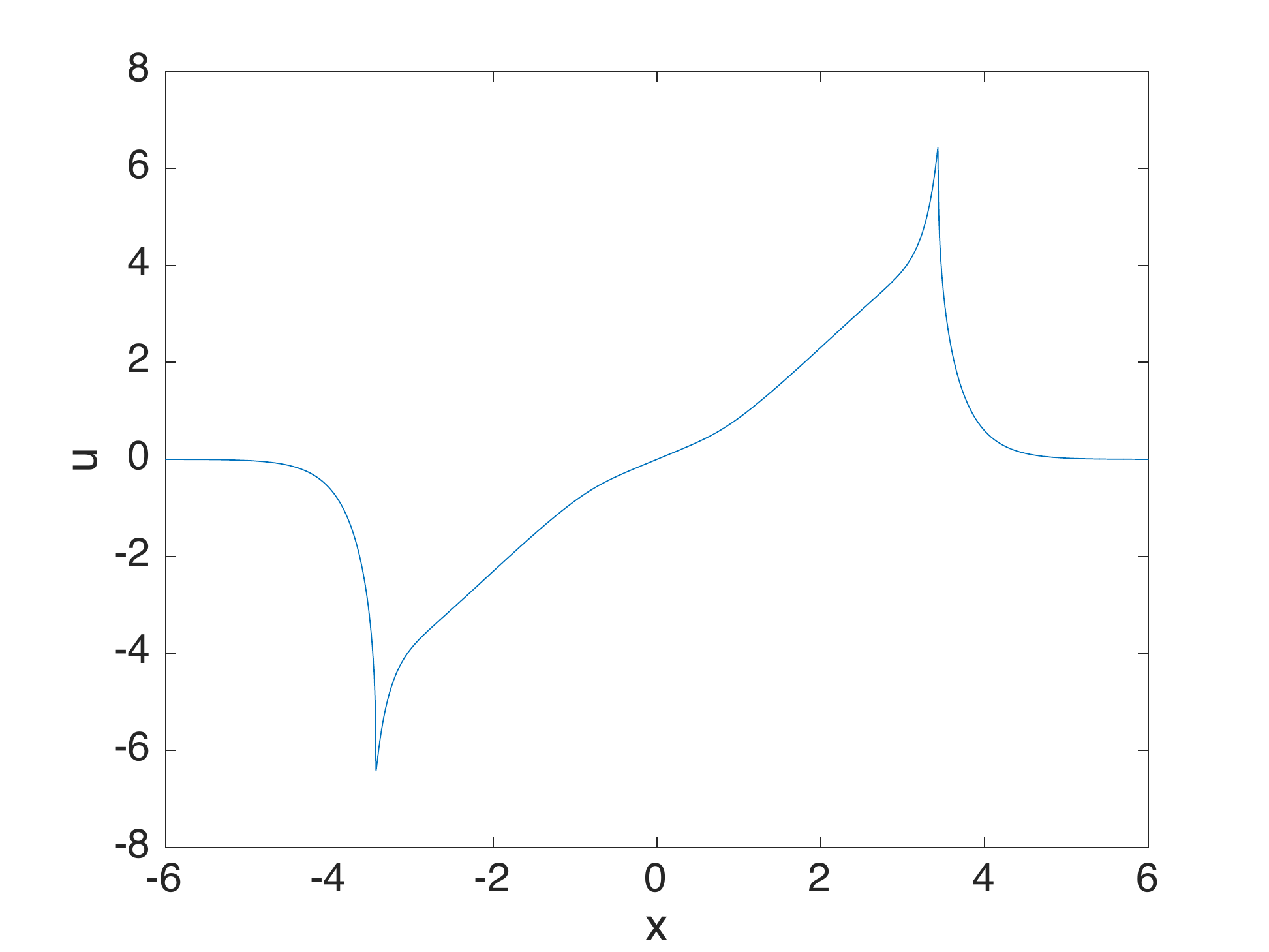}
  \includegraphics[width=0.49\textwidth]{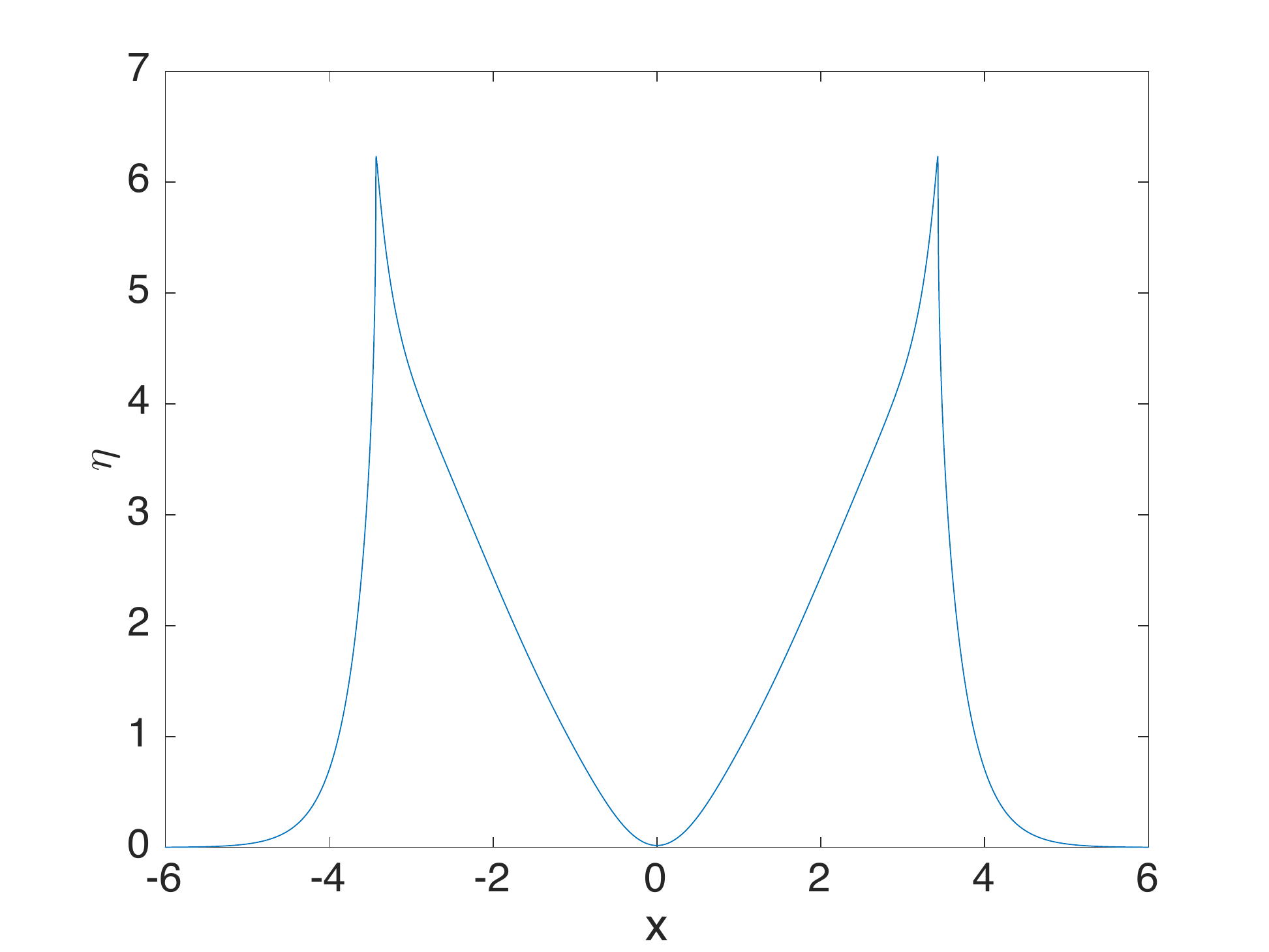}
 \caption{Solution to the Boussinesq system (\ref{sys}) for 
 $\epsilon=0.1$ and the initial 
data  $u(x,0)=0$, $\eta(x,0)=10\exp(-x^{2})$ at the time $t=2.2262$.}
 \label{figsautsysgausse01}
\end{figure}

For even smaller values of $\epsilon$, the solution for the same 
initial data appears to be global in time. As can be seen in 
Fig.\ref{figsautsysgausse001} for $\epsilon=0.01$, just two `solitons' 
appear to emerge from the initial hump. For sufficiently small $\epsilon$, the 
system appears to be close to the wave equation.
\begin{figure}[htb!]
  \includegraphics[width=0.49\textwidth]{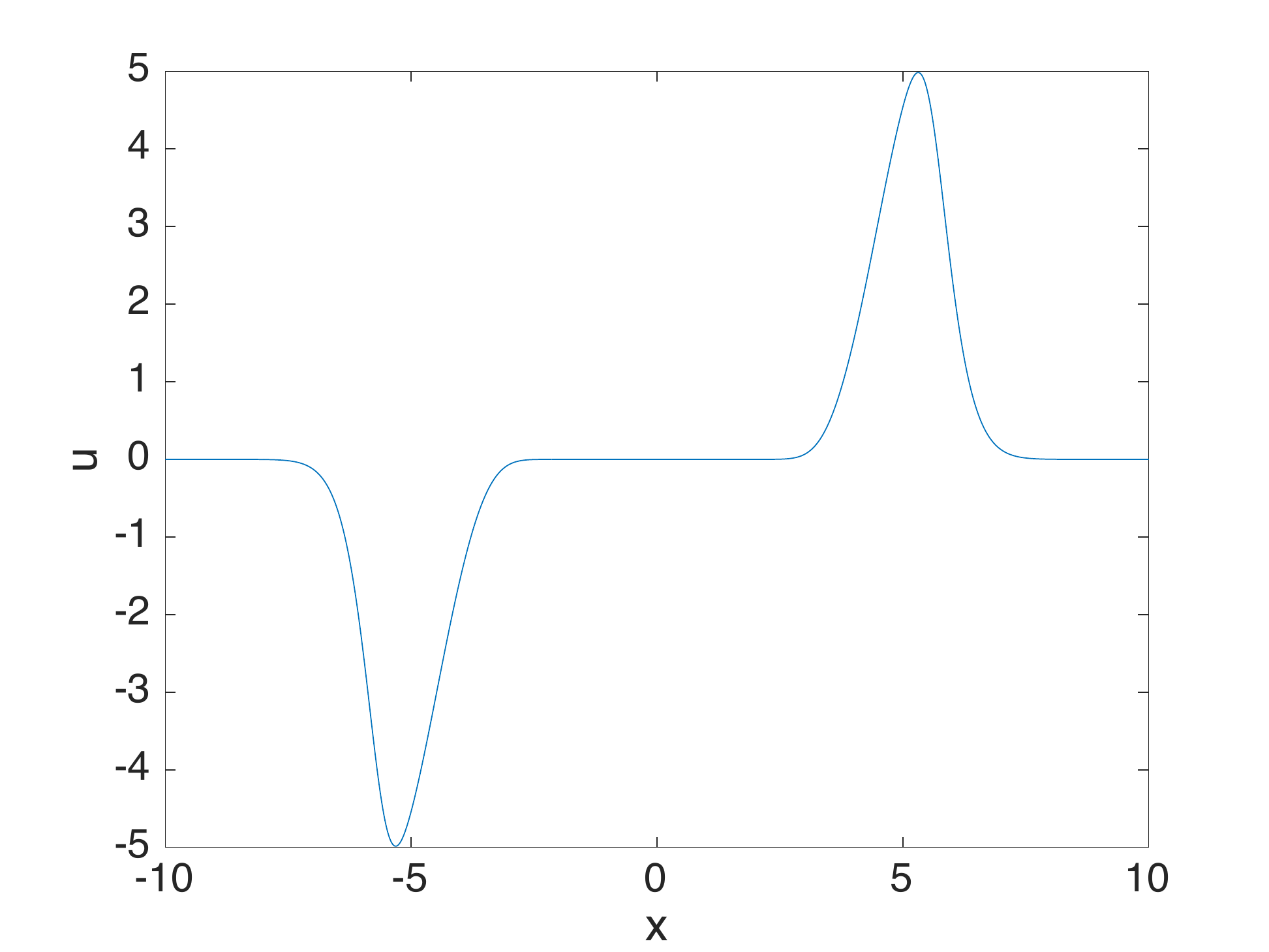}
  \includegraphics[width=0.49\textwidth]{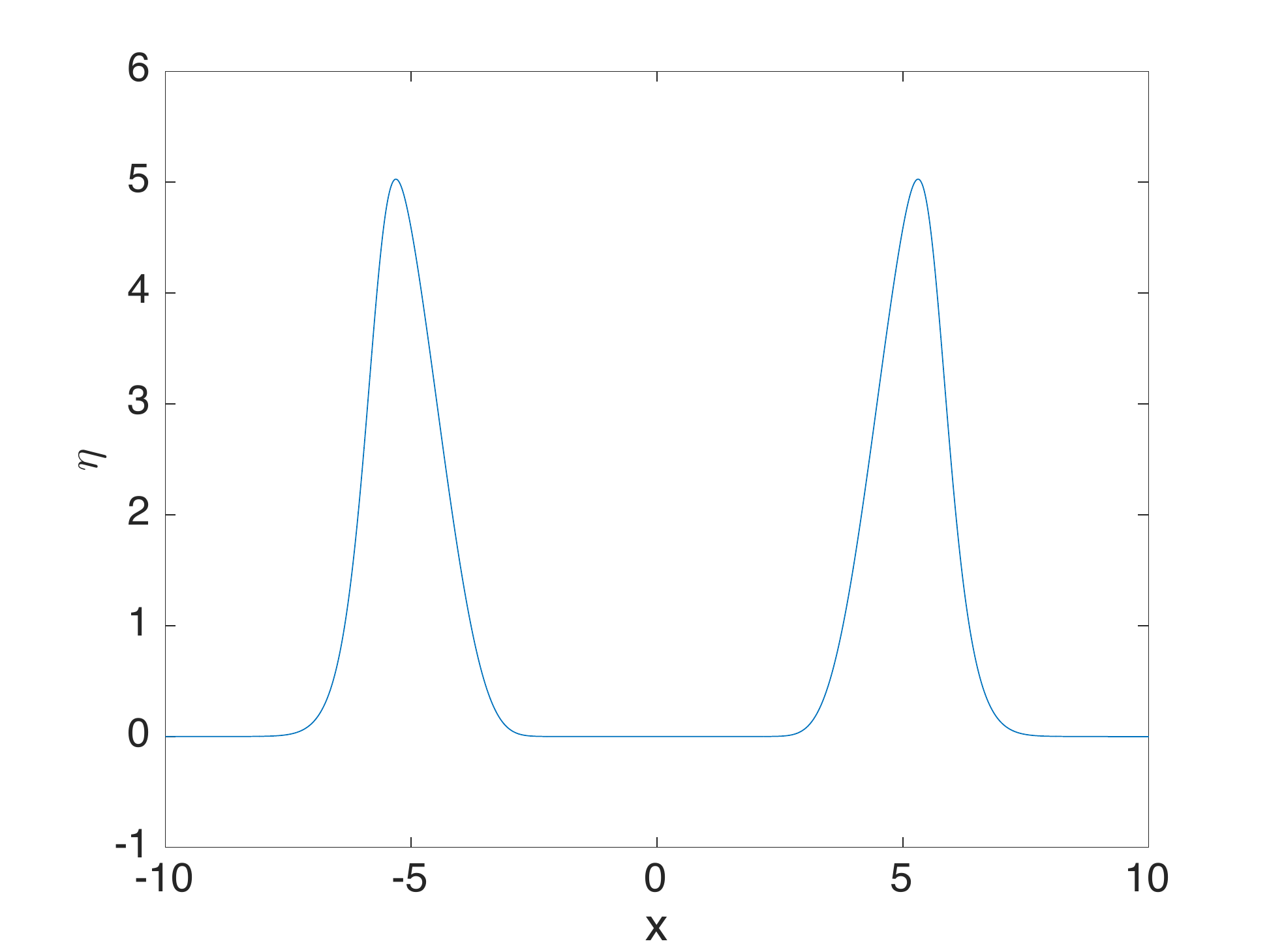}
 \caption{Solution to the Boussinesq system (\ref{sys}) for 
 $\epsilon=0.01$ and the initial 
data  $u(x,0)=0$, $\eta(x,0)=10\exp(-x^{2})$ at the time $t=5$.}
 \label{figsautsysgausse001}
\end{figure}

For negative $w$, the system (\ref{sys}) is ill-posed. For the 
numerical studies, this implies that the Krasny filter must be used 
at a higher level (we choose $10^{-10}$) in order to control 
numerical instabilities. If we consider the initial data $u(x,0)=0$, 
$\eta(x,0)=-10\exp(-x^{2})$, i.e., the same data which led to 
Fig.~\ref{figsautsysgauss} except for the sign of $w$, we find that 
the solution has a blow-up for $t=0.17$ as shown in 
Fig.~\ref{figsautsysmgauss}. The fitting of the Fourier coefficients 
according to (\ref{fit}) yields $\mu=0.2029$, i.e., again a cusp. 
\begin{figure}[htb!]
  \includegraphics[width=0.49\textwidth]{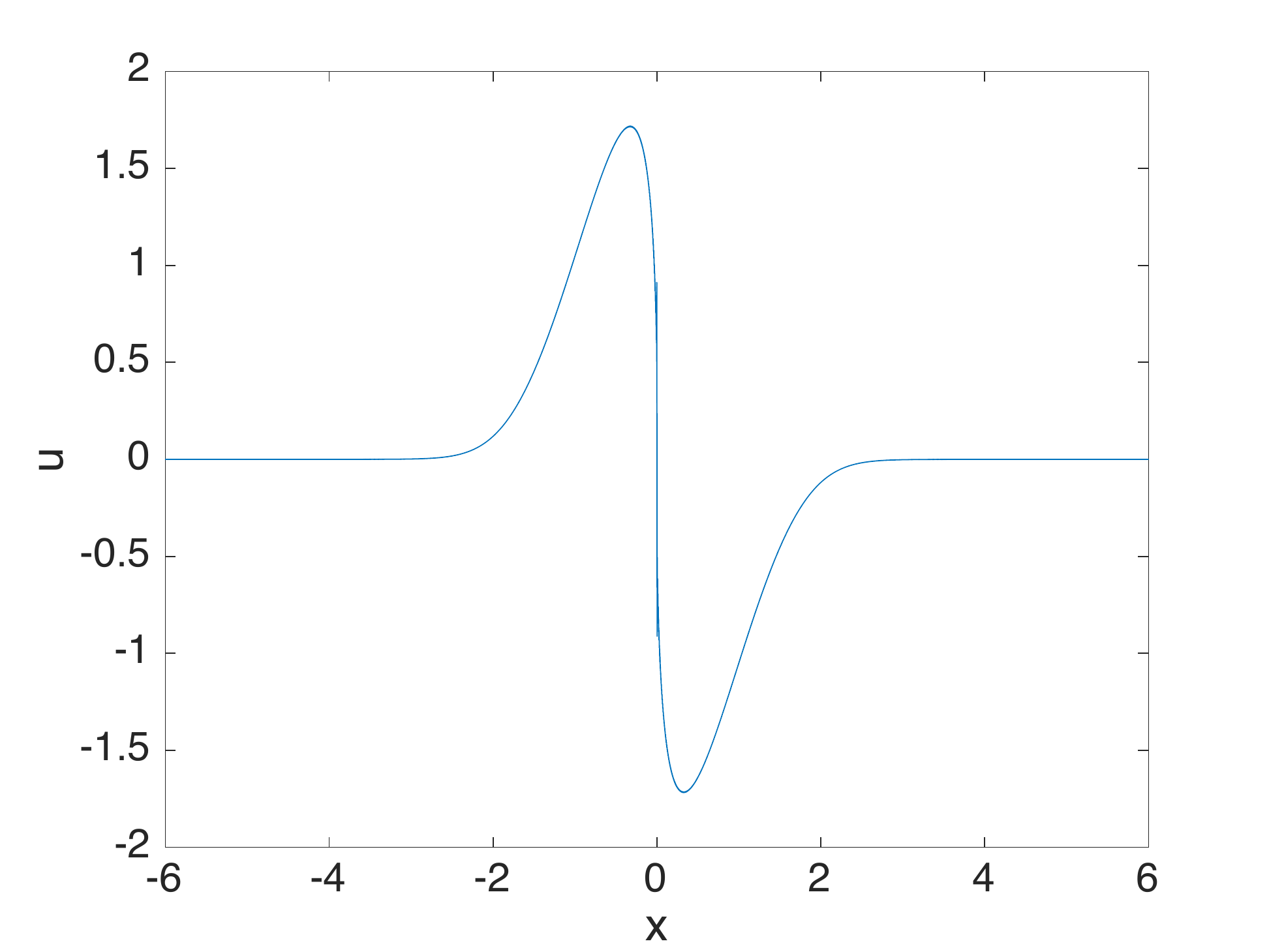}
  \includegraphics[width=0.49\textwidth]{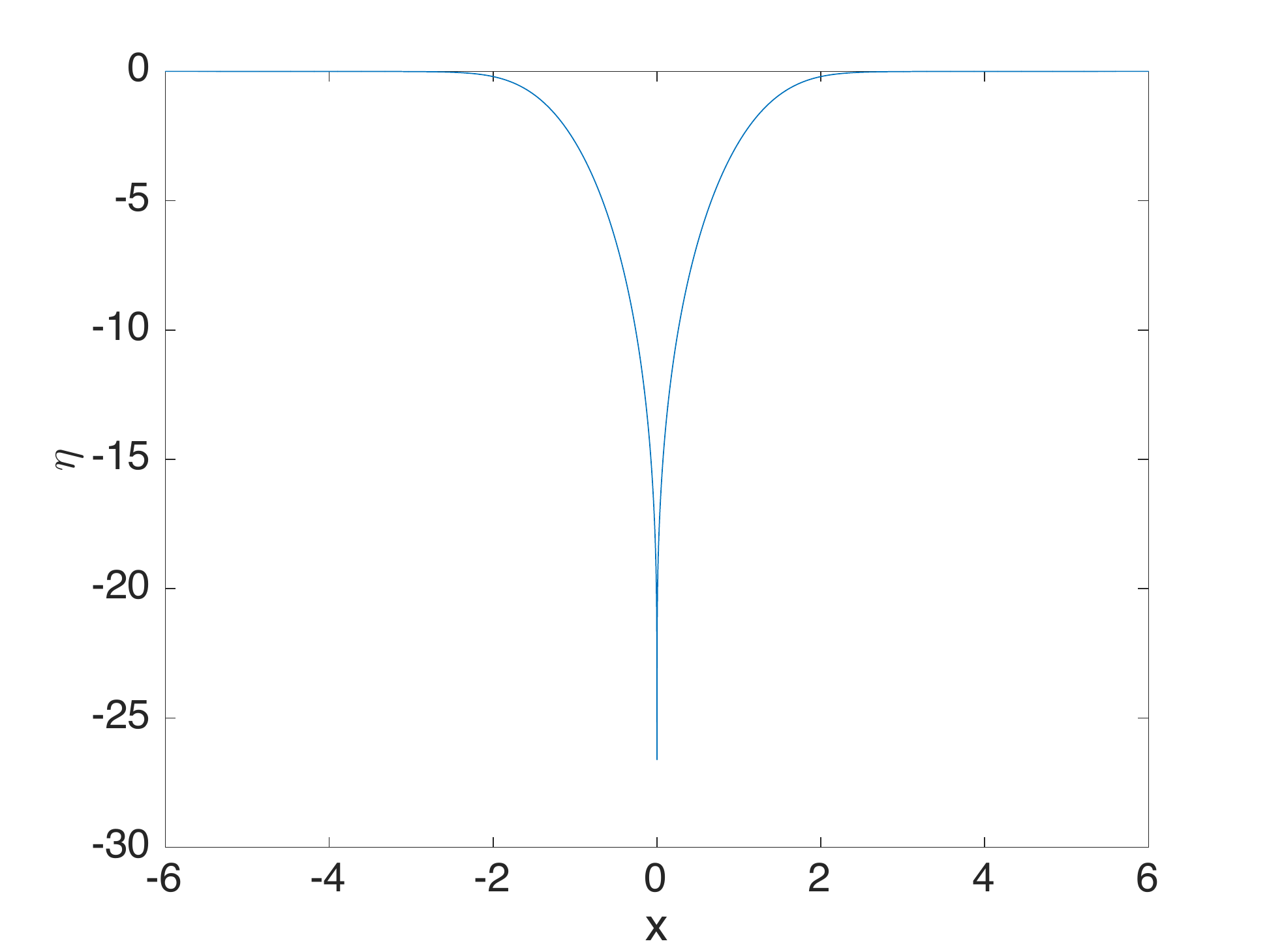}
 \caption{Solution to the Boussinesq system (\ref{sys}) for 
 $\epsilon=1$ and the initial 
data  $u(x,0)=0$, $\eta(x,0)=-10\exp(-x^{2})$ at the time $t=0.17$.}
 \label{figsautsysmgauss}
\end{figure}

For the same initial data, but $\epsilon=0.1$, a blow-up very similar 
to Fig.~\ref{figsautsysmgauss} is observed 
for $t=0.579$ with $\mu=1.15$. In the case of even smaller 
$\epsilon=0.01$, the solution appears to be again close to the 
solution of the wave equation as in Fig.~\ref{figsautsysgausse001}.

The effects of the ill-posedness of the system for negative initial 
data are more visible in the presence of oscillations. In 
Fig.~\ref{figsautsyssint00895u} we show the solution to the system 
for the initial data $u(x,0)=0$ and $\eta=\sin(10x)e^{-x^{2}}$. The code 
breaks at $t=0.0898$ since a singularity in the complex plane appears 
to hit the real axis. The fitting of the Fourier coefficients 
according to (\ref{fit}) is not conclusive ($\mu=0.0619$) whether 
this is a cusp or a pole. 
\begin{figure}[htb!]
  \includegraphics[width=0.49\textwidth]{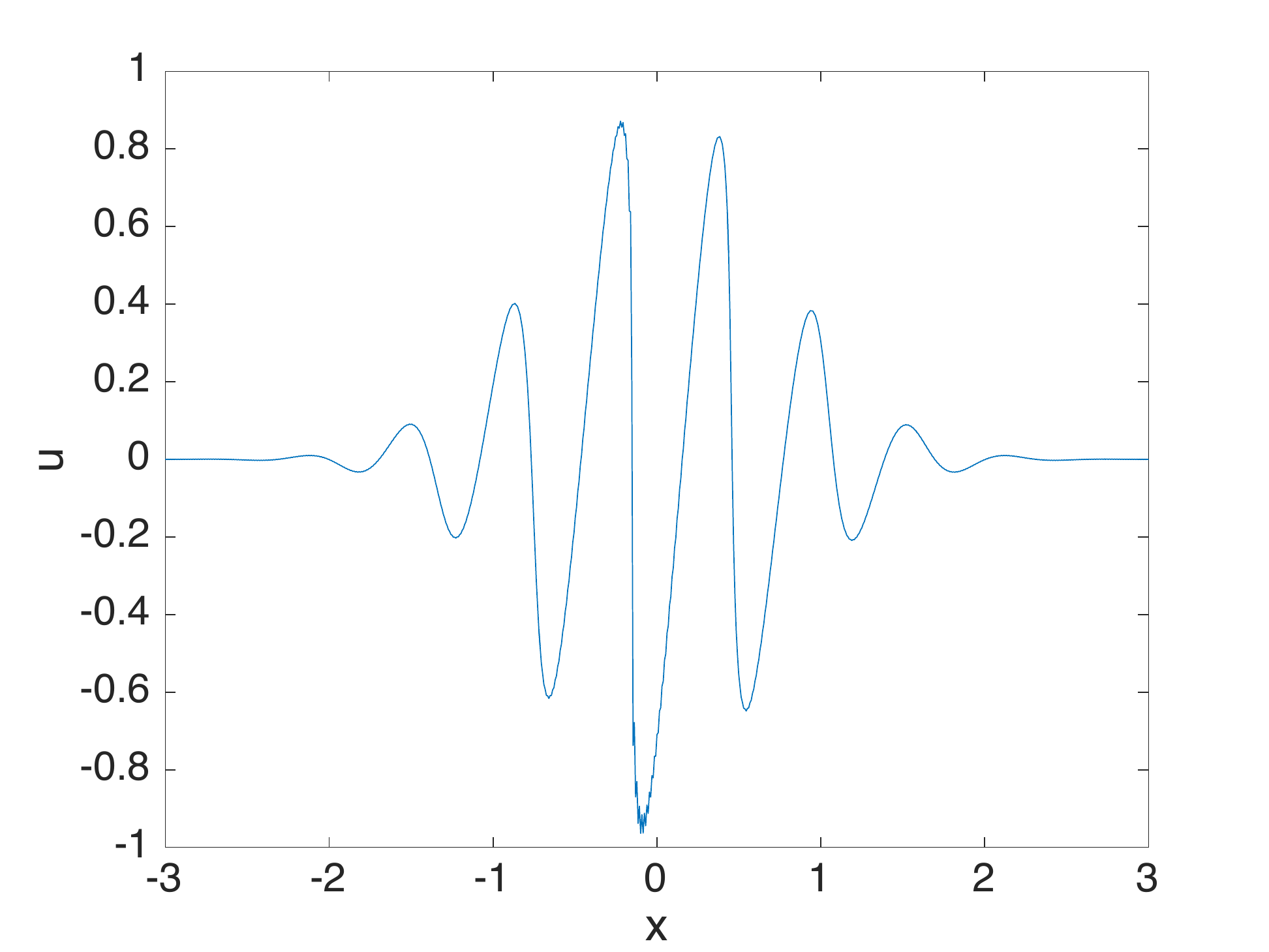}
  \includegraphics[width=0.49\textwidth]{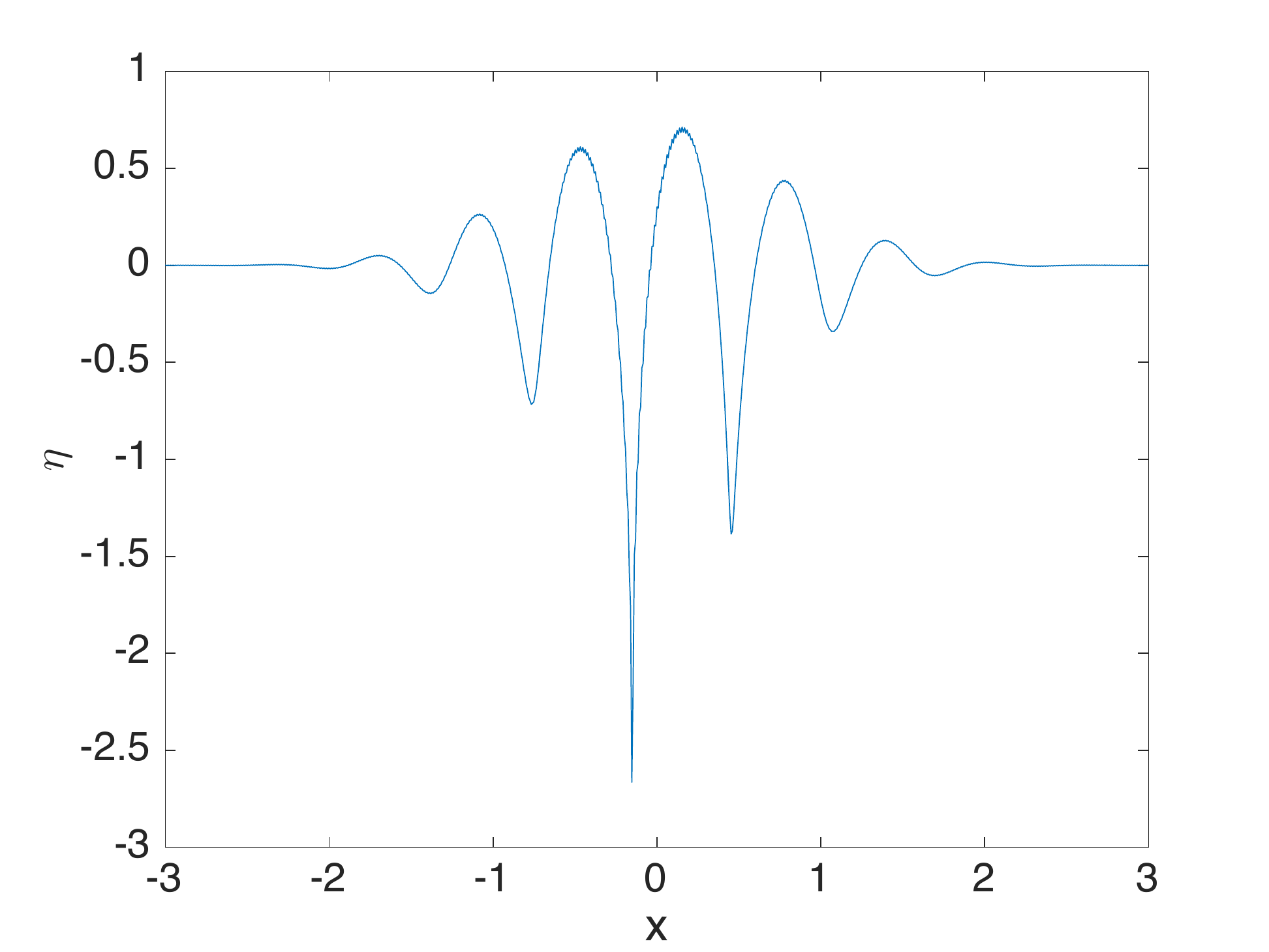}
 \caption{Solution to the Boussinesq system (\ref{sys}) for 
 $\epsilon=1$ and the initial 
data  $u(x,0)=0$, $\eta(x,0)=\sin(10x)\exp(-x^{2})$ at the time 
$t=0.085$.}
 \label{figsautsyssint00895u}
\end{figure}

\subsection{Numerical study of the Boussinesq system with surface 
tension}
If we study the same initial data as in the previous subsection for 
the Boussinesq system with surface tension, the effects of the 
stronger dispersion are clearly visible. 

The initial data of Fig.~\ref{figsautsysmgauss} lead in the 
presence of surface tension with $\beta=1$ to the solution in 
Fig.~\ref{figsautsysgaussb1}. The solution appears to exist for all 
times, instead of a shock a dispersive shock wave is observed. 
\begin{figure}[htb!]
  \includegraphics[width=0.49\textwidth]{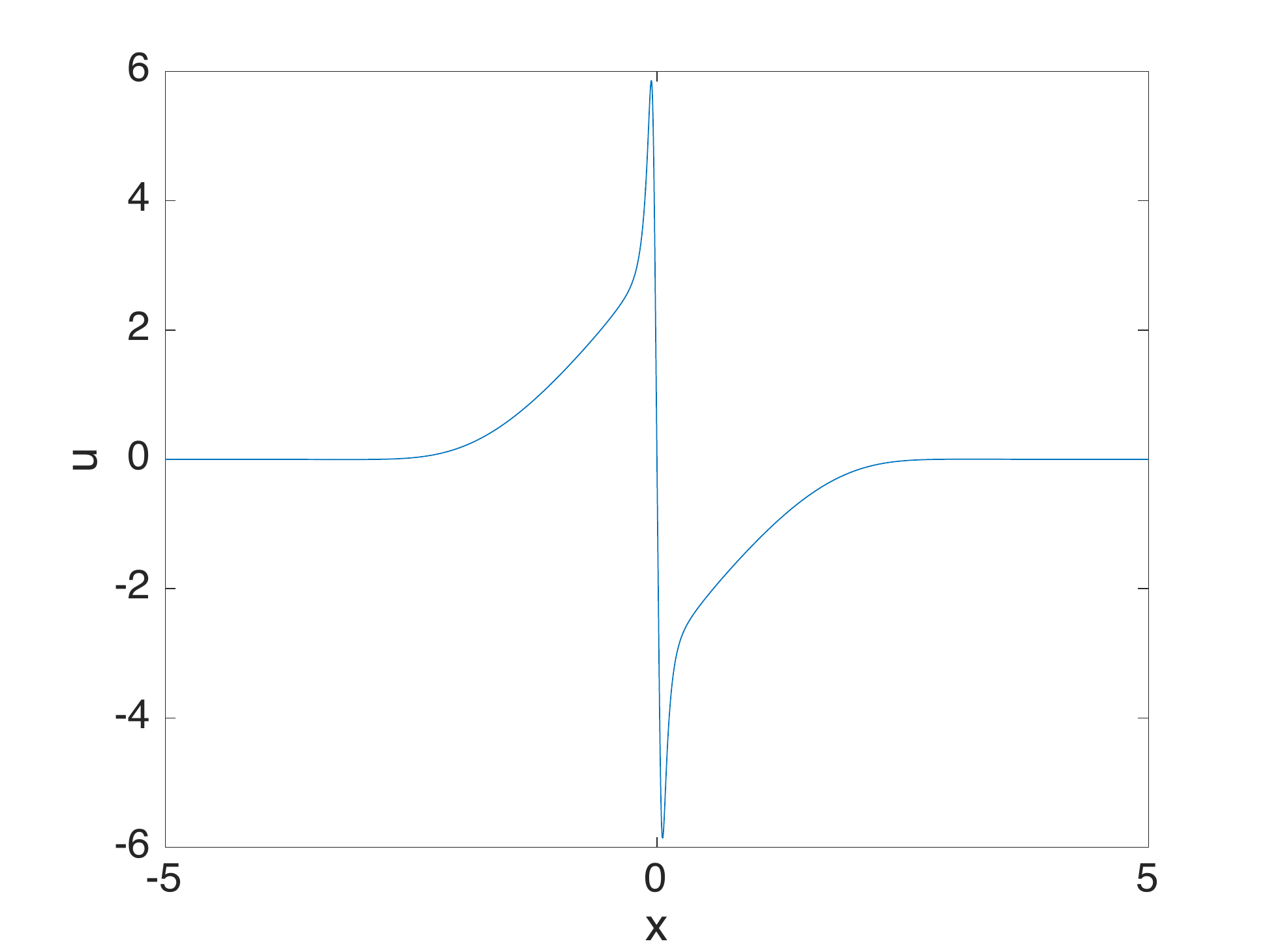}
  \includegraphics[width=0.49\textwidth]{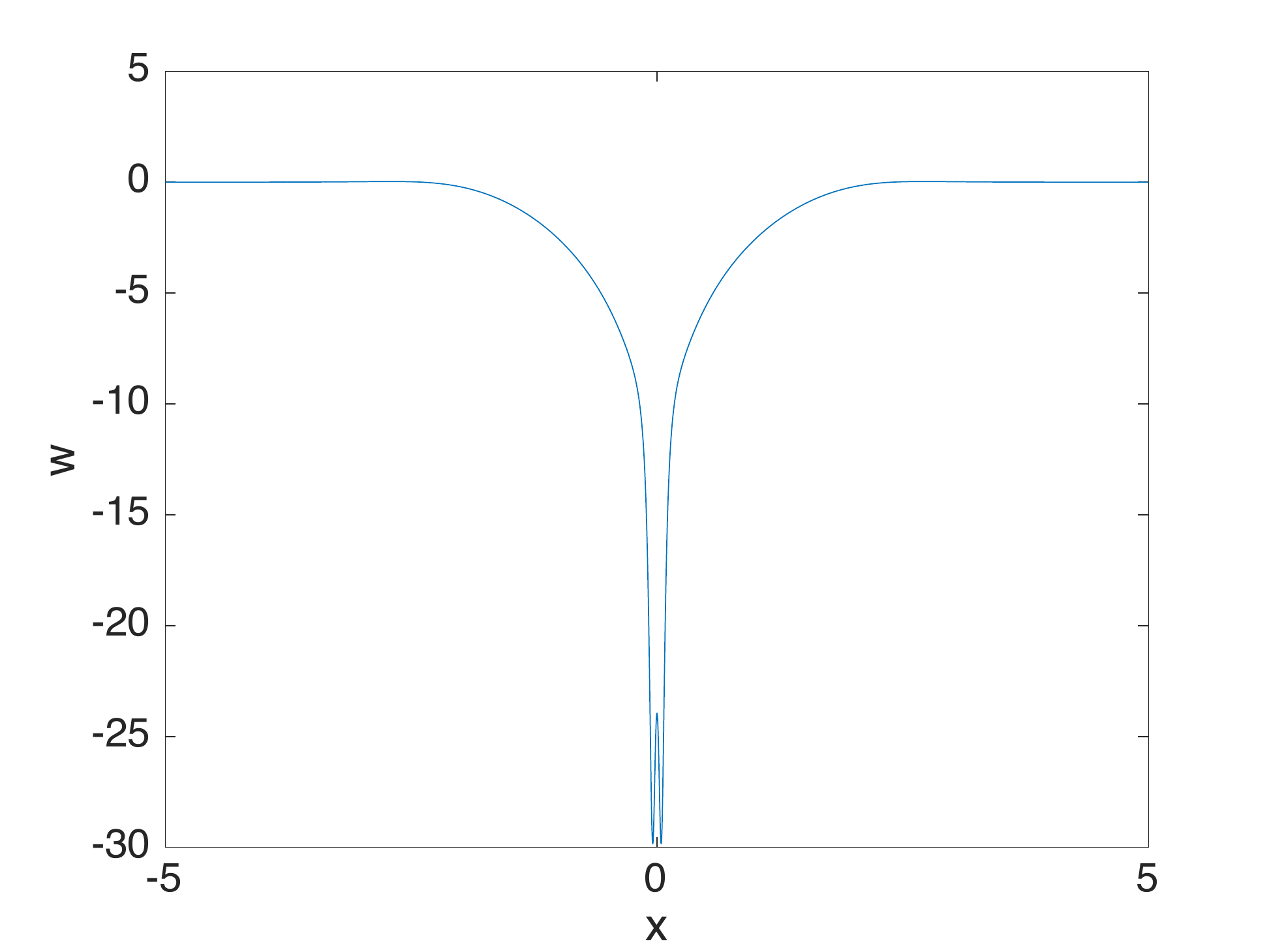}
 \caption{Solution to the Boussinesq system (\ref{sys}) with surface 
 tension $\beta=1$ for 
 $\epsilon=1$ and the initial 
data  $u(x,0)=0$, $\eta(x,0)=10\exp(-x^{2})$ at the time 
$t=2.4$.}
 \label{figsautsysgaussb1}
\end{figure}

For negative $\eta(x,0)$, the situation in Fig.~\ref{figsautsysmgauss} on 
the other hand is not really changed. Again a blow-up appears in 
finite time despite a surface tension of $\beta=1$ as can be seen in 
Fig.~\ref{figsautsysmgaussb1}. The code breaks at $t=0.2568$ with 
$\mu=-0.05$ after fitting the Fourier coefficients according to 
(\ref{fit}). 
\begin{figure}[htb!]
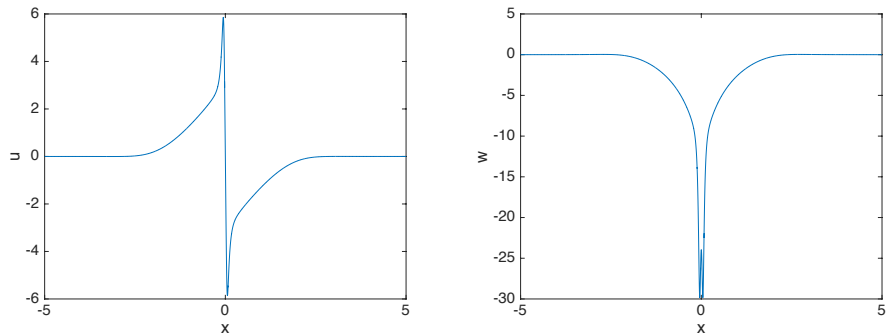

  \includegraphics[width=0.49\textwidth]{sautsysm10gaussb1t024u.pdf}
  \includegraphics[width=0.49\textwidth]{sautsysm10gaussb1t024w.pdf}
 \caption{Solution to the Boussinesq system (\ref{sys}) for 
 $\epsilon=1$ and the initial 
data  $u(x,0)=0$, $\eta(x,0)=-10\exp(-x^{2})$ at the time 
$t=0.085$.}
 \label{figsautsysmgaussb1}
\end{figure}

\vspace{0.5cm}





\section{Conclusion}

We have tried in this paper to answer some key questions on the Whitham equation and systems. As far as modeling of water waves is concerned, the Whitham equation
is relevant in the KdV (Boussinesq) regime and then it is probably 
not better than the KdV equation itself (and actually it is never used in realistic water waves modeling). On the other hand, because of its dispersion relation that behaves drastically differently for large and small frequencies, it has a fascinating variety of dynamical behaviors, most of them displayed in numerical simulations, that deserve further mathematical investigations.

The Boussinesq-Whitham systems, however, appear to have almost no interest for the modeling of water waves, since the local Cauchy problem (in absence of surface tension) can be solved only for initial data on the elevation that satisfies a rather unphysical condition. Its mathematical interest is also limited since its long wave limit is an ill-posed Boussinesq system.

The Cauchy problem for the system with surface tension while having similar shortcomings to model  water waves,  leads however  to an interesting open mathematical question.
 	
\begin{merci}
The Authors were partially  supported by the Brazilian-French program in mathematics and the MathAmSud program. J.-C. S. acknowledges support from the project ANR-GEODISP of the Agence Nationale de la Recherche. F.L and D.P. were partially supported by CNPq and FAPERJ/Brazil. J.-C. S. thanks David Lannes for useful discussions related to this work. 
\end{merci}
\bibliographystyle{amsplain}

\begin{thebibliography}{99}
\bibitem{AMP}\textsc{P. Acevez-Sanchez, A.A. Minzoni and P. Panayotaros}, {\it Numerical study of a nonlocal model for water-waves with variable depth}, Wave Motion, {\bf 50} (2013), 80-93.

\bibitem{AlBo} \textsc{J. P. Albert and J. L. Bona,}
\emph{Comparisons between model equations for long waves,}
J. Nonlinear Sci., \textbf{1} (1991), 345--374.

\bibitem{AlBoSa}\textsc{J. Albert, J. L. Bona and J.-C. Saut}, {\it Model equations for waves in stratified fluids},  Proc. Royal Soc. London A, {\bf 453}, (1997), 1233--1260.
\bibitem{ABM} \textsc{A. Ambrose, J.L. Bona and T. Milgrom}, {\it Global solutions and ill-posedness for the Kaup system and related Boussinesq systems}, Preprint (2017).
\bibitem{Ang}\textsc{J. Angulo Pava}, {\it Stability properties of solitary waves for fractional KdV and BBM equations}, arXiv:1701.06221v1  [math.AP]  22 Jan 2017.

\bibitem{Ar}\textsc{M.A. Arnesen},
{\it Existence of solitary-wave solutions to nonlocal equations}, Disc. Cont. Dyn. Syst. A, {\bf 36} (7)  (2016), 3483-3510.

\bibitem{BCS1} \textsc {J. L. Bona, M. Chen and J.-C. Saut}, {\it Boussinesq equations and
other systems for small-amplitude long waves in nonlinear dispersive
media I : Derivation and the linear theory}, J. Nonlinear Sci., {\bf
12} (2002), 283-318.

\bibitem{BoPrSc} \textsc{J. L. Bona, W. G. Pritchard and L. R. Scott,}
\emph{A comparison of solutions of two model equations for long waves,}
Lect. Appl. Math., \textbf{20} (1983), 235--267.

\bibitem{BS}\textsc{J.L. Bona and J.-C. Saut}, {\it Dispersive Blow-Up II.  Schr\"{o}dinger-TypeEquations, Optical and Oceanic Rogue Waves}, Chin. Ann. Math. Series B, {\bf 31}, (6), (2010), 793-810.

\bibitem{BoSm} \textsc{J.L. Bona and R. Smith,}
\newblock \emph{The initial value problem for the Korteweg-de Vries equation,}
\newblock Philos. Trans. R. Soc. Lond., Ser. A, \textbf{278} (1975), 555--601.

\bibitem{BMcR}\textsc{J.L. Bona, W.R. McKinney and J.M. Restrepo}, {\it Stable and unstable solitary-wave solutions of the generalized long-wave equation}, J. Nonlinear Sci., {\bf 10} (2000), 603-608.

\bibitem{BKN}\textsc{H. Borluk, H. Kalisch and D.P. Nicholls}, {\it A numerical study of the Whitham equation for steady surface water waves}, J. Comput. Appl. Math., {\bf 296} (2016), 293-302.

\bibitem{BEP}\textsc{G. Bruell, M. Ehrnstrom and L. Pei}, {\it Symmetry and decay of traveling wave solutions to the Whitham equation}, J. Diff. Eq., {\bf 262} (2017), 4232-4254.


\bibitem{Ca}\textsc{J.D. Carter}, {\it Bidirectional Whitham equations as models of waves in shallow water}, arXiv:1705.06503v1 [physics.flu-dyn] 18 May 2017.

\bibitem{CG} \textsc{J. D. Carter and D. George}, {\it The Whitham equation as a model of water waves}, preprint (2016).

\bibitem{DMDK}\textsc{E. Dinvay, D. Moldabayev, D. Dutykh and H. Kalisch}, {\it The Whitham equation with surface tension}, preprint (2016).

\bibitem{DGK} \textsc{B.~Dubrovin, T.~Grava and C.~Klein},  \emph{Numerical Study of breakup in generalized Korteweg--de Vries and Kawahara equations}, SIAM J. Appl. Math., {\bf 71} (2011),  983-1008.

\bibitem{DKM}\textsc{D. Dutykh, H. Kalish and D. Moldabayev}, {\it The Whitham equation as a model for surface waves}, Physica D, {\bf 309} (2015), 99-107.

\bibitem{EhEsPe} \textsc{M. Ehrnstr\"{o}m, J. Escher, and L. Pei,}
\newblock \emph{A note on the local well-posedness for
the Whitham equation},
\newblock Elliptic and parabolic equations, vol. 119 of Springer Proc.
Math. Stat., Springer, Cham, 2015, 63--75.

\bibitem{EhGr}\textsc{ M. Ehrnstr\"{o}m and M. Groves}, {\it Solitary wave solutions to the full dispersion Kadomtsev-Petviashvili equation}, to appear.

\bibitem {EGW} \textsc{M. Ehrnstr\"{o}m, M.D. Groves and E. Wahl\' en}, {\it On the existence and stability of solitary-wave solutions to a class of evolution equations of Whitham type}, Nonlinearity, {\bf 25} (2012), 2903--2936.

\bibitem{EK}\textsc{M. Ehrnstr\"{o}m and H. Kalisch}, {\it Global bifurcation  for the Whitham equation}, Math.Mod. Nat. Phenomena, {\bf 8} (2013), 13-30.

\bibitem{EK2}\textsc{M. Ehrnstr\"{o}m and H. Kalisch}, {\it Traveling waves for the Whitham equation}, Diff. and Int. Eq., {\bf 29} (11-12) (2009), 1193-1210.

\bibitem {EW} \textsc{M. Ehrnstr\"{o}m and E. Wahl\' en}, {\it  On Whitham's conjecture of a highest cusped wave for a nonlocal dispersive equation}, arXiv:1602.05384v1 [math.AP] 17 Feb 2016.

\bibitem{FL}\textsc{R.L. Frank and E. Lenzmann}, {\it On the uniqueness and non-degeneracy of ground states of $(-\Delta)^s Q+Q-Q^{\alpha +1}=0\; \text{in}\; \R$}, Acta Math., {\bf 210} (2) (2013), 261--318.
\bibitem{Hu}\textsc{V. Hur}, {\it Breaking in the Whitham equation for shallow water waves}, arXiv:1506.04075v1 [math.AP] 12 Jun 2015.

\bibitem{Hu2}\textsc{V. Hur}, {\it Norm inflation for equations of KdV type with fractional dispersion, } arXiv:1701.\newline 03354v1
 [math.AP] 12 Jan 2017.
 
\bibitem{HJ} \textsc{V. Hur and M. A. Johnson}, {\it Modulational instability in the Whitham equation for water waves}, Studies in Appl. Math.,  {\bf 134} (2014), 120-143.

\bibitem{Hu-Pa}\textsc{V. Hur and A.K. Pandey}, {\it Modulational instability in a full-dispersion shallow water model}, arXiv:1608.04685v1 [math.AP] 16 Aug 2016.

\bibitem{Hu-Tao}\textsc{V. Hur and  L. Tao}, {\it Wave breaking for the Whitham equation with fractional dispersion}, Nonlinearity, {\bf 27} (2014), 2937-2949.

\bibitem{Hu-Tao-II}\textsc{V. Hur and  L. Tao}, {\it Wave breaking in a shallow water model}, arXiv:1608.04681v2 [math.AP] 23 Sep 2016.

\bibitem{Ka}\textsc{D.J. Kaup}, {\it  A higher-order water-wave equation and the method for solving it}, Progr. Theoret. Phys., {\bf 54} (1975), 396-408.

\bibitem{KaPo} \textsc{T. Kato and G. Ponce,}
 \emph{Commutator estimates and the Euler
and Navier-Stokes equations}, Comm. Pure Appl. Math., {\bf 41} (1988), 891--907.

\bibitem{etna} \textsc{C. Klein}, {\it Fourth order time-stepping for low dispersion Korteweg-de Vries and nonlinear
Schr\"odinger equations}, ETNA {\bf 29} (2008), 116-135.

\bibitem{KP}\textsc{C. Klein and R. Peter}, {\it Numerical study of blow-up in solutions to generalized Kadomtsev-Petviashvili equations}, Discr. Cont.  Dyn.  Syst. B,  {\bf 19} (6), (2014), 1689-1717.

\bibitem{KP2}\textsc{C. Klein and R. Peter}, {\it Numerical study of blow-up in solutions to generalized Korteweg-de Vries equations},  Physica D,  {\bf 304-305} (2015), 52-78.

\bibitem{KS}\textsc{C. Klein and J.-C. Saut}, {\it A numerical approach to blow-up issues for dispersive perturbations of Burgers equation},  Physica D, {\bf 295-296} (2015), 46-65.

\bibitem{krasny} \textsc{R. Krasny}, {\it A study of singularity formation in a 
  vortex sheet by the point-vortex approximation}, J. Fluid Mech., {\bf 167} (1986), 65-93. 

\bibitem{Ku}\textsc{B.A. Kupperschmidt}, {\it Mathematics of dispersive water waves}, Commun. Math. Phys., {\bf 99} (1985), 51-73.

\bibitem{La1}\textsc{D. Lannes}, {\it Water waves : mathematical theory and asymptotics}, Mathematical Surveys and Monographs, vol 188 (2013), AMS, Providence.

\bibitem{LS} \textsc{D. Lannes and J.-C. Saut}, {\it Remarks on the full dispersion Kadomtsev-Petviashvili equation}, Kinetic and Related Models, American Institute of Mathematical Sciences, {\bf  6} (4) (2013), 989--1009.

 \bibitem{LPS}\textsc{F. Linares, D. Pilod and J.-C. Saut}, {\it Remarks on the orbital stability of ground state solutions of fKdV and related equations},  Advances Diff. Eq., {\bf 20} (9/10), (2015), 835-858.

 \bibitem{LPS2}
\textsc{F. Linares, D. Pilod, and J.-C. Saut},
{\it Dispersive perturbations of Burgers and hyperbolic equations I: local
theory},
SIAM J. Math. Anal., {\bf 46} (2014), 1505-1537.

 \bibitem{LPS3}
\textsc{F. Linares, D. Pilod, and J.-C. Saut}, {\it The Cauchy problem for the fractionary Kadomtsev-Petviashvili equations,} arXiv:1705.09744v1 [math.AP] 27 May 2017.

\bibitem{Ma}\textsc{A. Majda},  {\it Incompressible Fluid Flow and Systems of Conservation Laws in Several Space Variables},
(Heidelberg: Springer) 1984.

\bibitem{MM} \textsc{Y. Martel and F. Merle}, {\it Blow up in finite time and dynamics of blow up solutions for the $L^2$-critical generalized KdV equation}, J. Amer. Math. Soc., {\bf 15} (3)  (2002), 617--664.

\bibitem{MMR}\textsc{Y. Martel and F. Merle, and P. Rapha\"el}, \emph{Blow up for the 
critical gKdV equation I: dynamics near the solitary wave.}
Acta Math., \textbf{212} (1), 59--140 (2014).

\bibitem{MP}\textsc{Y. Martel and D. Pilod}, {\it Construction of a minimal mass blow up solution of the modified Benjamin-Ono},  preprint (2016), to appear in Mat. Annal., arXiv:1605.01837.

\bibitem{Me}\textsc{B. M\' elinand}, {\it A mathematical study of meteo and landslide tsunamis : the Proudman resonance}, Nonlinearity, {\bf 28} (2015), 4037-4080.

\bibitem{Mes}\textsc{B. M\' esognon-Gireau}, {\it A dispersive estimate for the linearized water-waves equations in finite depth}, J. Math. Fluid Mech., (2016). doi:10.1007/s00021-016-0286-1.

\bibitem{MPV}\textsc{L. Molinet, D. Pilod and S. Vento}, {\it On well-posedness for some dispersive perturbations of Burgers' equation}, arXiv:1702.03191v1  [math.AP]  10 Feb 2017.

\bibitem{MST}
\newblock \textsc{L. Molinet, J.-C.Saut and N. Tzvetkov},
\newblock {\it Ill-posedness issues for the Benjamin-Ono and related
equations},
\newblock SIAM J. Math. Anal., {\bf 33} (4) (2001), 982--988.

\bibitem{RK}\textsc{F. Remonato and H. Kalisch}, {\it Numerical bifurcation for the capillary Whitham equation}, Physica D, {\bf 343} (2017), 51- 62.

\bibitem{GMRES} \textsc{Y. Saad and M. Schultz}, {\it GMRES: A generalized 
 minimal residual algorithm for solving nonsymmetric linear systems}, SIAM J. Sci. Stat. Comput., {\bf 7} (3) (1986), 856-869. 
   
\bibitem{SKCK}\textsc{N. Sanford, K. Kodama, J.C. Carter and H. Kalisch}, {\it Stability of traveling wave solutions to the Whitham equation}, Physics Letters A, {\bf 378} (2014), 2100-2107.

\bibitem{RT}\textsc{F. Rousset and N. Tzvetkov}, {\it Transverse nonlinear instability for two-dimensional dispersive models}, Annales Inst. H. Poincar\'e Ana. N. Lin., {\bf 26} (2009), 477-496.


\bibitem{Sa} \textsc{J.-C. Saut},
\newblock \emph{Sur quelques g\'en\'eralisations de l'\'equation de Korteweg-de Vries},
\newblock J. Math. Pures Appl., \textbf{58} (1979), 21--61.

\bibitem{SWX}\textsc{ J.-C. Saut, Chao Wang and Li Xu}, {\it The  Cauchy problem on large time for surface waves Boussinesq systems II}, arXiv:1511.08824v1 [math.AP] 27 Nov 2015 and SIAM J. Math. Anal. to appear.

\bibitem{SSF} \textsc{C. Sulem, P. Sulem, and H. Frisch}, {\it Tracing complex  
singularities with spectral methods}, J. Comp. Phys., {\bf 50} (1983),  138-161.  

\bibitem{Whi}\textsc{G. B. Whitham}, {\it Variational methods and applications to water waves}, Proc.R. Soc. Lond. Ser. A., {\bf 299} (1967), 6-25.

\bibitem{Za}\textsc{V.E. Zakharov}, {\it Weakly nonlinear waves on the surface of an ideal finite depth fluid}, Amer. Math. Soc. Transl., {\bf 182} (2) (1998), 167-197.

\end{thebibliography}

\end{document}